\documentclass[11pt, reqno]{amsart}
\usepackage{amsfonts,latexsym,enumerate}
\usepackage{amsmath}
\usepackage{amscd}
\usepackage{float,amsmath,amssymb,mathrsfs,bm,multirow,graphics}
\usepackage[dvips]{graphicx}
\usepackage[percent]{overpic}
\usepackage{amsaddr}
\usepackage[numbers,sort&compress]{natbib}
\usepackage{xcolor}
\usepackage{enumitem}
\usepackage{subfigure}
\newcommand{\R}{{\Bbb R}}
\newcommand{\N}{{\Bbb N}}
\newcommand{\C}{{\Bbb C}}
\newcommand{\D}{{\Bbb D}}

\newcommand{\bU}{{\bf U}}
\newcommand{\bV}{{\bf V}}

\newcommand{\re}{\text{\upshape Re\,}}
\newcommand{\im}{\text{\upshape Im\,}}

\newcommand{\Ai}{\text{\upshape Ai}}

\newcommand\be{\begin{equation}}
\newcommand\ee{\end{equation}}
\newcommand{\bea}{\begin{eqnarray}}
\newcommand{\eea}{\end{eqnarray}}
\newcommand\berr{\begin{eqnarray*}}
\newcommand\eerr{\end{eqnarray*}}
{\setlength\arraycolsep{2pt}

\newcommand{\AS}{\textrm{AS}}
\newcommand{\PC}{\text{\upshape PC}}
\newcommand{\Boh}{\mathcal{O}}
\newcommand{\ud}{\,\mathrm{d}}

\def\XXint#1#2#3{{\setbox0=\hbox{$#1{#2#3}{\int}$}
\vcenter{\hbox{$#2#3$}}\kern-.5\wd0}}

\newtheorem{theorem}{Theorem}[section]
\newtheorem{proposition}{Proposition}[section]

\newtheorem{lemma}[proposition]{Lemma}

\newtheorem{remark}[proposition]{Remark}

\newtheorem{rhp}[proposition]{RH problem}
\numberwithin{equation}{section}

\date{\today}
\title[Higher order Airy and Painlev\'e asymptotics  for the mKdV hierarchy]
{Higher order Airy and Painlev\'e asymptotics \\ for the mKdV Hierarchy}

\author{Lin Huang$^{\dagger}$ and Lun Zhang$^{\ddagger}$}
\address{$^{\dagger}$School of Science, Hangzhou Dianzi University,\\ Zhejiang 310018, China. \\ $^{\ddagger}$School of Mathematical Sciences and Shanghai Key Laboratory for Contemporary Applied Mathematics,
 Fudan University, \\Shanghai 200433, China.}
\email{lin.huang@hdu.edu.cn, lunzhang@fudan.edu.cn}

\begin{document}

\begin{abstract}
\noindent
In this paper, we consider Cauchy problem for the modified Korteweg-de Vries hierarchy on the real line with decaying initial data. Using the Riemann--Hilbert formulation and nonlinear steepest descent method, we derive a uniform asymptotic expansion to all orders in powers of $t^{-1/(2n+1)}$ with smooth coefficients of the variable $(-1)^{n+1}x((2n+1) t)^{-1/(2n+1)}$ in the self-similarity region for the solution of $n$-th member of the hierarchy. It turns out that the leading asymptotics is described by a family of special solutions of the Painlev\'e II hierarchy, which generalize the classical Ablowitz-Segur solution for the Painelv\'{e} II equation and appear in a variety of random matrix and statistical physics models. We establish the connection formulas for this family of solutions. In the special case that the reflection coefficient vanishes at the origin, the solutions of Painlev\'e II hierarchy in the leading coefficient vanishes as well, the leading and subleading terms in the asymptotic expansion are instead given explicitly in terms of derivatives of the generalized Airy function.
\end{abstract}

\maketitle



\setcounter{tocdepth}{1}
\tableofcontents

\section{Introduction and main results}
In a seminal work \cite{DZ1993}, Deift and Zhou introduced the celebrated nonlinear steepest descent method to analyze asymptotics of oscillatory Riemann--Hilbert (RH) problems. This approach deals with the RH problem directly and consists of various contour deformations, following the spirit of classical steepest descent method. Compared with the inverse scattering method \cite{Zakharov}, it does not require any priori ansatz for the asymptotic form of the solution. Since the solutions for a variety of integrable nonlinear differential equations are closely related to RH problems, the nonlinear steepest descent method and its variants has been applied successfully to resolve many asymptotic problems arising from integrable systems.

The illustrative example used in \cite{DZ1993} is the modified Korteweg-de Vries (mKdV) equation
\begin{equation}\label{eq:mkdv}
u_{t}-6u^2u_x+u_{xxx}=0,\quad x\in\mathbb{R},\quad t>0,
\end{equation}
with Schwartz space initial data. Long-time asymptotics of the solution, which depends explicitly on the reflection coefficient associated with the initial data, is presented in six regions of the $(x,t)$-plane. In the similarity region $-M_1\leq x/t \leq -M_2$, $M_{1,2}>0$, the mKdV equation can be solved to any fixed order $\Boh(1/t^n)$, $n\in\mathbb{N}=\{1,2,\ldots\}$, and the leading asymptotics is described by a slowly decaying modulated sine wave. The full asymptotic expansion in this region is later derived in \cite{DZ1994}, and each higher order coefficient therein can be calculated recursively. In the self-similarity region $|x|\leq Mt^{1/3}$, $M>0$, leading asymptoics of the solution, however, is given in terms of the Ablowitz-Segur solution \cite{AS-1976,SA-1981} for the homogeneous Painelv\'{e} II equation
\begin{equation}\label{eq:PII}
q''(x)-2q(x)^3=xq(x),
\end{equation}
which is determined by the reflection coefficient. This result has recently been improved in a paper of Charlier and Lenells \cite{CL-2019} by showing that the solution actually admits a uniform expansion to all orders in powers of $t^{-1/3}$ with smooth coefficients. Moreover, if the reflection coefficient vanishes at the origin, they derived the leading and subleading terms in the expansion explicitly with the aid of the classical Airy function. Besides these studies of mKdV equation on the real line, we also refer to \cite{BFS2004,L2016} for the asymptotic results on the half-line.

In this paper, we are concerned with the mKdV hierarchy \cite{CJM-2006} which is defined by
\begin{align}\label{mkdvhierarchy}
u_t+\frac{\partial }{\partial x}\bigg(\frac{\partial }{\partial x} +2 u\bigg)\mathcal{L}_n[u_x-u^2]=0,\quad n\in\mathbb{N},
\end{align}
where the operator $\mathcal{L}_n$  satisfies the Lenard recursion relation \cite{Lax}:
\begin{align}\label{LOdef}
\begin{cases}
 \frac{d}{dx}\mathcal{L}_{j+1}f=\bigg(\frac{d^3}{dx^3}+4f\frac{d}{dx}+2f_x \bigg)\mathcal{L}_{j}f,\\
  \mathcal{L}_{0}f=\frac{1}{2},\quad \mathcal{L}_{j}0=0, \quad j=1,2,\cdots.
 \end{cases}
\end{align}
If $n=1$ in \eqref{mkdvhierarchy}, one recovers \eqref{eq:mkdv}, and the equation for $n=2$ reads
\begin{align*}
u_{t}-10u^2u_{xxx}-40 u u_x u_{xx}-10u_x^3+30u^4 u_x+u_{xxxxx}=0.
\end{align*}
Following the spirit in \cite{CL-2019}, emphasis will be put on the higher order Painlev\'{e}-type asymptotics for the mKdV hierarchy in the self-similarity region with initial data $u(x,0)= u_0(x) \in \mathcal{S}(\R)$, where $\mathcal{S}(\R)$ is the Schwartz class of smooth rapidly decaying functions. As we will show later, the role played by the Ablowitz-Segur solution and Airy function in the mKdV equation will be replaced by their higher-order generalizations. In the literatures, we note that the Painlev\'{e} transcendents and their higher-order analogues are crucial in asymptotic analysis of many integrable nonlinear differential equations, as can be seen from their appearances in the focusing nonlinear Schr\"{o}dinger equation \cite{BT13,BLM}, in critical asymptotics for Hamiltonian perturbations of hyperbolic and elliptic systems \cite{Claeys}, in the Camassa-Holm equation \cite{dIS-2010}, in the Sasa-Satsuma equation \cite{hl-2019}, in an extended mKdV equation \cite{LG21b,LGWW-2019} and in the sine-Gordon equation \cite{LM}. The higher order asymptotics in similarity region for other integrable equations can be found in \cite{HXF-2015,V-2000}.

Main results of this paper are stated in what follows.

\subsection{Main Results}
To state our results, we start with the Painlev\'{e} II hierarchy, which is a sequence of ordinary differential equations obtained from equations of the mKdV hierarchy via self-similar reduction \cite{FN80}; see also \cite{CJP,Kud02,Maz07}. The $n$-th member of the Painlev\'{e} II hierarchy is a non-linear ordinary differential equation of order $2n$ for $q=q(x)$, and reads as
\begin{equation}\label{def:PIIhierar}
\left(\frac{d}{dx}+2q\right)\mathcal{L}_n[q_x-q^2]=xq, \qquad n \in\mathbb{N},
\end{equation}
where the Lenard operators are defined in \eqref{LOdef}. The Painlev\'{e} II equation \eqref{eq:PII} corresponds to $n=1$ in \eqref{def:PIIhierar}, while for $n=2$, we have
\begin{equation*}
q''''(x)-10q(x)(q'(x))^2-10q(x)^2q''(x)+6q(x)^5=xq(x).
\end{equation*}
In the most general case, the Painlev\'{e} II hierarchy depends on several parameters, and the one in \eqref{def:PIIhierar} corresponds to the case that all the parameters are taken to be zero.

As is well-known, the solutions to Painlev\'{e} equations and hierarchy are transcendental in general, hence, we cannot expect simple closed forms for the solutions. By assuming that $q((-1)^{n+1}x)$ tends to zero exponentially fast as $x\to+\infty$, it is readily seen that the $n$-th member of the Painlev\'{e} II hierarchy \eqref{def:PIIhierar} is approximated by the generalized Airy equation
\begin{equation}\label{def:kAiry}
\frac{d^{2n}}{dx^{2n}}y(x)=xy(x).
\end{equation}
If $n=1$, the above equation is nothing but the classical Airy equation. It is straightforward to check that the function $\Ai_{2n+1}((-1)^{n+1}x)$ with
\begin{equation}\label{def:kAiry2}
\Ai_{2n+1}(x):=\frac{(-1)^{n+1}}{2 \pi i}\int_{\gamma}e^{(-1)^n\frac{s^{2n+1}}{2n+1}+x s}ds
\end{equation}
solves \eqref{def:kAiry}, where $\gamma$ is a curve in the left half of the complex plane that is asymptotic to straight lines with arguments $\pm \frac{n+1}{2n+1}\pi$ at infinity with the orientation from the bottom to the top. Note that $\Ai_3(x)=\frac{1}{2 \pi i}\int_{\gamma}e^{-\frac{s^{3}}{3}+x s}ds$ is the standard Airy function $\Ai(x)$ \cite{DLMF}.

Let $-1<\rho<1$ be a real number, it has recently been shown in \cite{CCG2019} that each of the $n$-th member of the Painlev\'{e} II hierarchy \eqref{def:PIIhierar} admits a one-parameter family of real solutions\footnote{In \cite{CCG2019}, the authors only considered the case that $0<\rho<1$, but it is clear that the arguments therein can be extended to $-1<\rho<1$.} denoted by $q_{\AS,n}(x;\rho)$ that are pole-free on the real line with the asymptotics
\begin{equation*}
q_{\AS,n}((-1)^{n+1}x;\rho)\sim \rho \Ai_{2n+1}(x), \qquad x\to +\infty,
\end{equation*}
where $\Ai_{2n+1}(x)$ is defined in \eqref{def:kAiry2}. This family of solutions plays an important role in multicritical edge statistics for the momenta of fermions in nonharmonic traps \cite{DMS18} and other statistical physics model \cite{ACV12}. They are natural generalizations of the classical real Ablowitz-Segur solutions for the Painelv\'{e} II equation, and are determined by the Stokes multipliers
\begin{equation}\label{eq:SMs}
s_1=-s_{2n+1}=\rho, \qquad s_2=\cdots=s_{2n}=0.
\end{equation}
Since $q_{\AS,n}$ will also appear in long-time asymptotics of the mKdV hierarchy, our first result
is the following asymptotics of $q_{\AS,n}((-1)^{n+1}x;\rho)$ as $x\to -\infty$. This particularly establishes the so-called connection formulas for this special family of solutions.
\begin{theorem}\label{prop:connection}
For each of the $n$-th member of the Painlev\'{e} II hierarchy \eqref{def:PIIhierar}, there exists a one-parameter family of real solutions $q_{\AS,n}(x;\rho)$ such that as $x \to +\infty$,
\begin{equation}\label{eq:asypinfty}
q_{\AS,n}((-1)^{n+1}x;\rho) = \rho \Ai_{2n+1}(x)(1+o(1)), \qquad -1<\rho<1,
\end{equation}
where $\Ai_{2n+1}(x)$ is defined in \eqref{def:kAiry2}, and as $x \to -\infty$,
\begin{multline}\label{eq:asypminusinfty}
q_{\AS,n}((-1)^{n+1}x;\rho) = \frac{\ud }{\sqrt{n}(-x)^{\frac{2n-1}{4n}}}\cos\left(\frac {2n}{2n+1}(-x)^{\frac{2n+1}{2n}}-\frac{2n+1}{4n}\ud^2\ln(-x)+\varphi\right)
\\+\Boh ((-x)^{-1}),
\end{multline}
where the constants $\ud$ and $\varphi$ are related to the parameter $\rho$ through the connection formulas
\begin{align}
\ud & =\sqrt{-\frac{\ln(1-\rho^2)}{\pi}},
\\
\varphi & =-\frac{\ud^2}{2}\ln(8n)+\arg \Gamma\left(\frac{\ud^2}{2}i\right)+\frac{\pi}{2} \mathrm{sgn}(\rho)-\frac{\pi}{4},
\label{eq:connection}
\end{align}
with $\Gamma(z)$ being the Gamma function.
\end{theorem}
If $n=1$, Theorem \ref{prop:connection} is due to Ablowitz and Segur \cite{AS-1976,SA-1981}; see also \cite{CM-1988,DZ95,HM80} for rigorous derivations using different methods.

We are now ready to state long-time asymptotics of the mKdV hierarchy. It comes out that, in the self-similarity region $|x| < C t^{1/(2n+1)}$, where $C= C(n)$ is a positive constant for fixed $n$, the solution of the mKdV hierarchy \eqref{mkdvhierarchy} with Schwartz space initial data admits a uniform expansion to all orders in powers of $t^{-1/(2n+1)}$ with smooth coefficients. Furthermore, the leading coefficient is described by the generalized Ablowitz-Segur solution $q_{\AS,n}$ with the parameter explicitly determined by the reflection coefficient according to the inverse scattering transform on the real line.

\begin{theorem}\label{mainth1}
Let $u(x,t)$ be the solution for each of the $n$-th member of the mKdV hierarchy \eqref{mkdvhierarchy} with initial condition $u_0(x)=u(x,0)\in \mathcal{S}(\R)$. As $t\to \infty$, we have

\begin{align}\label{asymptoticsinIV}
u(x,t) = \sum_{j=1}^N \frac{u_j(y)}{t^{\frac{j}{2n+1}}} + \Boh \big(t^{-\frac{N+1}{2n+1}}\big),  \qquad y \doteq (-1)^{n+1} \frac{x} {((2n+1)t)^{\frac{1}{2n+1}}},
\end{align}
uniformly for $|x| \leq C t^{1/(2n+1)}$ with fixed $C > 0$ and $N \geq 1$, where $\{u_j(y)\}_{j\in \N}$ are smooth functions of $y \in \R$ and
\begin{align}\label{u1expression}
u_1(y)=(2n+1)^{-\frac{1}{2n+1}}q_{\AS,n}((-1)^{n+1}y,ir(0)).
\end{align}
In \eqref{u1expression}, $q_{\AS,n}$ is the generalized Ablowitz-Segur solution of the Painlev\'{e} II hierarchy \eqref{def:PIIhierar} as stated in Theorem \ref{prop:connection}, and $r(\lambda)$ is the reflection coefficient associated with the initial date $u_0$.

\end{theorem}
We shall see existence of the solution for each $n$-th member of the mKdV hierarchy with Schwartz class initial data from Lemma \ref{solutionexists} below.

By substituting \eqref{asymptoticsinIV} into \eqref{mkdvhierarchy} and comparing the coefficients of powers of $t^{-1/(2n+1)}$, it follows that the coefficients $u_j(y)$, $j=2,3,\ldots$, in \eqref{asymptoticsinIV} satisfy coupled differential equations, which in general cannot be solved explicitly. If the reflection coefficient vanishes at the origin, however, one can calculate the first few terms as shown in the following theorem.
\begin{theorem}\label{mainth2}
If $r(0) = 0$, then the asymptotic formula (\ref{asymptoticsinIV}) still holds with
\begin{align}\label{u2u3def}
u_1(y) \equiv 0,
\quad
u_2(y)=\frac{r'(0)}{2 \times (2n+1)^{\frac{2}{2n+1}}} \Ai'_{2n+1}(y),
\end{align}
and
\begin{align}\label{u2u3def2}
u_3(y)= -\frac{ir''(0)}{8 \times  (2n+1)^{\frac{3}{2n+1}}} \Ai^{''}_{2n+1}(y),
\end{align}
where the $\Ai_{2n+1}$ is the generalized Airy function defined in \eqref{def:kAiry2}.
\end{theorem}
Theorems \ref{mainth1} and \ref{mainth2} extend the results for the mKdV equation \eqref{eq:mkdv} in \cite{CL-2019} to the mKdV hierarchy \eqref{mkdvhierarchy}. For $n=2$, the leading asymptotics in \eqref{asymptoticsinIV} is also known in \cite{LGWW-2019}.

\subsection{Organization of the paper and notation}\label{notation}
The rest of this paper is devoted to the proofs of our main results, which rely on the Deift/Zhou steepest descent analysis of the associated RH problems and a technique introduced by Charlier and Lenells in \cite{CL-2019} to derive the higher order asymptotic expansion. In Section \ref{background}, we present RH representations for the mKdV hierarchy and the Painlev\'e II hierarchy, respectively. Asymptotic analysis of these RH problems are scattered in Sections \ref{proof:Theorem1}--\ref{proof:th2th3-2}. To analyse the RH problem for the mKdV hierarchy, it is necessary to divide the discussion into several different cases. The asymptotic outcomes will finally lead to the proofs of our main results, i.e., Theorems \ref{prop:connection}--\ref{mainth2}, as shown in Section \ref{sec:proofs}.

We conclude this section with some notation used throughout this paper.
\begin{itemize}

\item If $A$ is an $n \times m$ matrix, the $A_{ij}$ stands for the $(i,j)$-th entry of $A$. We define $|A| \geq 0$ by
$|A|^2 = \sum_{i,j} |A_{ij}|^2$. It is then easily seen that $|A + B| \leq |A| + |B|$ and $|AB| \leq |A| |B|$.

\item For a (piecewise smooth) contour $\gamma \subset \C$ and $1 \leq p \leq \infty$, we write $A \in L^p(\gamma)$  if $|A|$ belongs to $L^p(\gamma)$. $A \in L^p(\gamma)$ if and only if each entry $A_{ij}$ belongs to $L^p(\gamma)$. We also define $\|A\|_{L^p(\gamma)} := \| |A|\|_{L^p(\gamma)}$.

\item For a complex-valued function $f(k)$ of $k \in \C$, we use
\begin{equation}\label{def:schconj}
f^*(k):= \overline{f(\bar{k})}
\end{equation}
to denote its Schwartz conjugate.

\item As usual, the three Pauli matrices $\{\sigma_j\}_{j=1}^3$ are defined by
\begin{equation}\label{def:Pauli}
\sigma_1=\begin{pmatrix}
           0 & 1 \\
           1 & 0
        \end{pmatrix},
        \quad
        \sigma_2=\begin{pmatrix}
        0 & -i \\
        i & 0
        \end{pmatrix},
        \quad
        \sigma_3=
        \begin{pmatrix}
        1 & 0 \\
         0 & -1
         \end{pmatrix}.
\end{equation}

\item Let $D$ be an open connected subset of $\C$ bounded by a piecewise smooth curve $\gamma \subset \hat{\C} := \C \cup \{\infty\}$ and $z_0 \in \C \setminus \bar{D}$. We use $\dot{E}^p(D)$, $1 \leq p < \infty$, to denote the space of all analytic functions $f: D \to \C$ with the property that there exist piecewise smooth curves $\{C_n\}_{n=1}^\infty$ in $D$ tending to $\gamma$ in the sense that $C_n$ eventually surrounds each compact subset of $D$ and such that
$$\sup_{n \geq 1} \int_{C_n} |z - z_0|^{p-2} |f(z)|^p |dz| < \infty.$$
If $D = D_1 \cup \cdots \cup D_n$ is a finite union of such open subsets, then $\dot{E}^p(D)$ denotes the space of analytic functions $f:D\to \C$ such that $f|_{D_j} \in \dot{E}^p(D_j)$ for each $j=1,\ldots,n$.

\item For a (piecewise smooth) oriented contour $\gamma \subset \hat{\C}$ and a function $h$ defined on $\gamma$, the Cauchy transform of $h$ is defined by
\begin{align*}
(\mathcal{C}h)(z) = \frac{1}{2\pi i} \int_\gamma \frac{h(z')dz'}{z' - z}, \qquad z \in \C \setminus \gamma,
\end{align*}
whenever the integral converges. If $h \in L^2(\gamma)$, the left and right non-tangential boundary values of $\mathcal{C}h$ exist a.e. on $\gamma$ and belong to $L^2(\gamma)$, which we denote by $\mathcal{C}_+ h$ and $\mathcal{C}_- h$, respectively. Moreover, $\mathcal{C}_\pm \in \mathcal{B}(L^2(\gamma))$, where $\mathcal{B}(L^2(\gamma))$ is the space of bounded linear operators on $L^2(\gamma)$, and by the Sokhotski-Plemelj relation, it follows that $\mathcal{C}_+ - \mathcal{C}_- = I$,

\item
In what follows, it is understood that most of the quantities encountered should depend on a parameter $n$, which corresponds to the $n$-th member of the mKdV hierarchy or the Painlev\'{e} II hierarchy. For simplicity, we omit this dependence unless otherwise specified.
\end{itemize}

\section{Preliminaries}\label{background}
\subsection{Lax Pair for the mKdV hierarchy}
According to \cite{AKNS} (see also \cite[Proposition 1]{CJM-2006}), the Lax pair for $n$-th equation of the mKdV hierarchy \eqref{mkdvhierarchy} is given by
\begin{align}\label{def:Lax1}
\begin{cases}\phi_x=
\begin{pmatrix}-i \lambda & u \\u & i\lambda\end{pmatrix}\phi,
\\
\phi_{t}=\begin{pmatrix}A & B \\ D & -A \end{pmatrix}\phi,
\end{cases}
\end{align}
where
\begin{equation*}
A=\sum_{j=0}^{2n+1}A_j(i\lambda)^j,\quad B=\sum_{j=0}^{2n}B_j(i\lambda)^j, \quad D=\sum_{j=0}^{2n}D_j(i\lambda)^j,
\end{equation*}
with
\begin{align}
\nonumber A_{2n+1}&=4^n,\quad A_{2k}=0,~~ k=0,1,\ldots,n,\\
\nonumber A_{2k+1}&=\frac{4^{k+1}}{2}\left\{\mathcal{L}_{n-k}[u_x-u^2]-\frac{\partial }{\partial x}\left(\frac{\partial }{\partial x} +2 u \right)\mathcal{L}_{n-k-1}[u_x-u^2]\right\},
~~k=0,1,\ldots,n-1,\\
\nonumber B_{2k+1}&=\frac{4^{k+1}}{2}\left\{\frac{\partial }{\partial x}\left(\frac{\partial }{\partial x} +2 u\right)\mathcal{L}_{n-k-1}[u_x-u^2]\right\},~~ k=0,1, \ldots, n-1,\\
\nonumber B_{2k}&=-4^k\bigg(\frac{\partial }{\partial x} +2 u\bigg)\mathcal{L}_{n-k}[u_x-u^2],~~ k=0,1,\ldots,n,
\\
 \label{eq:DjBj}
D_{k}&=(-1)^kB_{k},~~ k=0,1,\ldots,2n,
\end{align}
and $\mathcal{L}_k$, $k=0,1,\ldots,n$, being the Lenard operators defined in \eqref{LOdef}.
\begin{remark}
If the initial data $u_0(x) \in \mathcal{S}(\R)$,  then $\begin{pmatrix}B(x,0) \\D(x,0)\end{pmatrix}\to \begin{pmatrix}0 \\0\end{pmatrix}$ as $x\to \infty$.
\end{remark}

By introducing the matrix-valued functions
\begin{align*}
Q(x,t)=&\begin{pmatrix}0& u\\u&0 \end{pmatrix}, \\
\tilde{Q}(x,t,\lambda)=&\begin{pmatrix}A(x,t,\lambda)-2^{2n}(i\lambda)^{2n+1} &B(x,t,\lambda)\\D(x,t,\lambda) &-A(x,t,\lambda)+2^{2n}(i\lambda)^{2n+1}\end{pmatrix},
\end{align*}
we could rewrite the Lax pair \eqref{def:Lax1} as
\begin{align}\label{eq:Lax2}
\begin{cases}\psi_x=(i (-1)^{n}\lambda \sigma_3+Q(x,t))\psi\doteq {\bf U}(x,t,\lambda)\psi,\\
\psi_t=(-i 2^{2n} \lambda^{2n+1}\sigma_3+\tilde{Q}(x,t,(-1)^{n+1}\lambda))\psi\doteq {\bf V}(x,t,\lambda)\psi,
\end{cases}
\end{align}
where $\sigma_3$ is the Pauli matrix defined in \eqref{def:Pauli}.
Equivalently, let
\begin{align*}
\psi=\Phi e^{-i((-1)^{n+1}\lambda x+2^{2n} \lambda^{2n+1} t)\sigma_3},
\end{align*}
we have
\begin{align*}
\begin{cases}
\Phi_x+i(-1)^{n+1}\lambda[\sigma_3,\Phi]=Q(x,t) \Phi,\\
\Phi_t+i 2^{2n} \lambda^{2n+1} [\sigma_3,\Phi]=\tilde{Q}(x,t,(-1)^{n+1} \lambda)\Phi.
\end{cases}
\end{align*}
From \eqref{eq:DjBj}, it follows that ${\bf U}(x,t,\lambda)$ and ${\bf V}(x,t,\lambda)$ in \eqref{eq:Lax2} satisfy the following symmetry relations:
\begin{align*}
\bU(x,t,-\lambda)&=\overline{\bU(x,t,\bar{\lambda})}, \qquad\, \bV(x,t,-\lambda)=\overline{\bV(x,t,\bar{\lambda})},
\\
\sigma_1\overline{\bU(x,t,\bar{\lambda})}\sigma_1&=\bU(x,t,\lambda),\quad \sigma_1\overline{\bV(x,t,\bar{\lambda})}\sigma_1=\bV(x,t,\lambda).
\end{align*}
where $\sigma_1$ is given in \eqref{def:Pauli}. This, in turn, implies that
\begin{align*}
\Phi(x,t,-\lambda)=\overline{\Phi(x,t,\bar{\lambda})}, \qquad \sigma_1\overline{\Phi(x,t,\bar{\lambda})}\sigma_1=\Phi(x,t,\lambda).
\end{align*}

\subsection{RH problem for the mKdV hierarchy}\label{overviewsec}
For the Cauchy problem of the mKdV hierarchy with initial data $u_0(x)$, we see from the nonlinear Fourier transform formalism \cite{zmn-1984} and the standard unified method introduced by Fokas \cite{F-2008} that the associated reflection coefficient $r$ is defined by
\begin{align}\label{def:r}
r(\lambda)=\frac{b^\ast(\lambda)}{a(\lambda)},
\end{align}
where $f^\ast$ stands for its Schwartz conjugate \eqref{def:schconj}, and the spectral functions $a(\lambda)$ and $b(\lambda)$ constitute the scattering matrix
\begin{align*}
\begin{pmatrix}a^\ast(\lambda)& b(\lambda)
\\
b^\ast(\lambda)& a(\lambda)
\end{pmatrix}.
\end{align*}
We note that $r$ satisfies the symmetry relation
\begin{align}\label{rsymm}
r(\lambda) = -\overline{r(-\lambda)}, \quad \lambda \in \R,
\end{align}
and if $u_0(x) \in \mathcal{S}(\R)$, then
\begin{align*}
\sup_{\lambda \in \R}|r (\lambda)|<1.
\end{align*}

With the reflection coefficient $r$ in \eqref{def:r}, we define a $2 \times 2$  matrix-valued function $v(x, t, \lambda)$ by
\begin{equation}\label{def:v}
 v(x,t,\lambda) = \begin{pmatrix} 1 - |r(\lambda)|^2 & -\overline{r(\lambda)}e^{-t\Theta(\xi,\lambda)} \\ r(\lambda)e^{t \Theta(\xi,\lambda)} & 1 \end{pmatrix},
\end{equation}
where \begin{equation}\label{def:Thetaxi}
 \Theta(\xi, \lambda) := 2i( (-1)^{n+1}\xi \lambda+ 2^{2n} \lambda^{2n+1} ), \qquad \xi := \frac{x}{t}.
\end{equation}
We then formulate the following RH problem:
\begin{rhp}\label{rhp-org}
\hfill
\begin{enumerate}[label=\emph{(\alph*)}, ref=(\alph*)]
\item $m(x, t, \lambda)$ is analytic for $\lambda \in \C \setminus \R$.
\item For a.e. $\lambda \in \R$, the limiting values $$m_{+/-}(x,t,\lambda):=\lim_{\substack{\lambda' \to \lambda \\\lambda'\textrm{ on the upper/lower half-plance}}}m(x,t,\lambda')$$
exist, and satisfy the jump condition
\begin{align}\label{preRHm}
 m_+(x,t,\lambda) = m_-(x, t, \lambda) v(x, t, \lambda),
 \end{align}
 where the jump matrix $v$ is given in \eqref{def:v}.
\item As $\lambda \to \infty$, we have $m(x,t,\lambda) = I + \Boh(\lambda^{-1})$.
\end{enumerate}
\end{rhp}

From the standard argument (cf. \cite[Theorem 1]{hl-2019}), the relation between the above RH problem and the solution of the mKdV hierarchy is given in the following lemma.
\begin{lemma}\label{solutionexists}
The RH problem \ref{rhp-org} for $m$ has a unique solution for the each $(x,t)\in \R^2$ and the limit $\lim_{\lambda\to \infty} (\lambda m(x,t,\lambda))_{21}$ exists. Moreover, the function $u(x,t)$ defined by
\begin{align}\label{recoveru}
u(x,t) =2\lim_{\lambda\to \infty} (\lambda m(x,t,\lambda))_{21},
\end{align}
is a smooth function with rapid decay as $|x|\to \infty$, which satisfies the $n$-th member of the mKdV hierarchy \eqref{mkdvhierarchy} with Schwartz class initial data.
\end{lemma}

\begin{remark}
By \eqref{rsymm}, it is clear that $r(0)$ is purely imaginary and the jump matrix $v$ defined in \eqref{def:v} satisfies
$$v(x,t,\lambda) = \sigma_1\overline{v(x,t,\bar{\lambda})}^{-1}\sigma_1 = \sigma_1 \sigma_3 v(x,t,-\lambda)^{-1}\sigma_3\sigma_1,\qquad \lambda \in \R.
$$
Thus, by uniqueness of the solution of the RH problem \ref{rhp-org}, it follows that
\begin{align}\label{msymm}
m(x,t,\lambda) = \sigma_1\overline{m(x,t,\bar{\lambda})}\sigma_1 = \sigma_1 \sigma_3 m(x,t,-\lambda)\sigma_3\sigma_1, \qquad \lambda \in \C \setminus \R.
\end{align}
\end{remark}

We will perform asymptotic analysis of the RH problem \ref{rhp-org} for $m$ as $t\to \infty$. Based on the parity of $n$ and the range of $x$, we split the analysis into four cases, namely,
\begin{align*}
\text{\bf Case I} &\doteq \{(x,t) \mid 0\leq {x}\leq M t^{\frac{1}{2n+1}}\}, &\quad n \text{ is odd}, \\
\text{\bf Case II} &\doteq \{(x,t) \mid 0\leq{x}\leq M t^{\frac{1}{2n+1}}\}, &\quad n \text{ is even}, \\
\text{\bf Case III} &\doteq \{(x,t) \mid - M t^{\frac{1}{2n+1}}\leq {x}\leq 0\},&\quad n \text{ is even}, \\
\text{\bf Case IV} &\doteq \{(x,t) \mid -M t^{\frac{1}{2n+1}}\leq x\leq0\}, &\quad n \text{ is odd},
\end{align*}
where $M$ is a positive constant. Since the analysis for {\bf Case III} and {\bf Case IV} is similar to that for {\bf Case I} and {\bf Case II} (see Remark \ref{rk:caseIII} below for a brief comment), it suffices to focus on the first two cases, which is presented in Sections \ref{proof:th2th3-1} and \ref{proof:th2th3-2} below, respectively.

\subsection{RH problem for the Painlev\'e II hierarchy}
We finally give an RH characterization of the generalized Ablowitz-Segur solution $q_{\AS,n}$ of the Painlev\'{e} II hierarchy \eqref{def:PIIhierar}. As aforementioned, the RH problem below is obtained from the general one (cf. \cite{CIK-2010}) by choosing the specified Stokes multipliers \eqref{eq:SMs}.
\begin{rhp}\label{rhp:psi}
\hfill
\begin{enumerate}[label=\emph{(\alph*)}, ref=(\alph*)]
\item $\Psi(\zeta)=\Psi(x, \rho,\zeta)$ is defined and analytic in $\mathbb{C} \setminus
\Upsilon $,
where
\begin{equation}\label{def:Upsilon}
\Upsilon:=\Upsilon_1\cup\Upsilon_{2n+1}\cup\Upsilon_{2n+2}\cup\Upsilon_{4n+2}
\end{equation}
with
\begin{equation}\label{def:Upsilonj}
\Upsilon_j:=\left\{\zeta\in \mathbb{C} ~\bigg |~ \arg \zeta = \frac{2j-1}{4n+2}\pi\right\}, \qquad j=1,\ldots,4n+2.
\end{equation}

\item For $\zeta \in \Upsilon$, we have
\begin{equation*}
\Psi_{+}(\zeta) = \Psi_{-}(\zeta)J_{\Psi}(\zeta),
\end{equation*}
where
\begin{equation*}
J_{\Psi}(\zeta):= \left\{
        \begin{array}{ll}
          \begin{pmatrix} 1 & 0 \\ \rho & 1 \end{pmatrix}, & \qquad \hbox{$\zeta \in \Upsilon_1$,} \\
          \begin{pmatrix} 1 & 0 \\ -\rho & 1 \end{pmatrix}, & \qquad \hbox{$\zeta \in \Upsilon_{2n+1}$,} \\
          \begin{pmatrix} 1 & \rho \\ 0 & 1 \end{pmatrix}, & \qquad \hbox{$\zeta\in \Upsilon_{2n+2}$,} \\
          \begin{pmatrix} 1 & -\rho \\ 0 & 1 \end{pmatrix}, & \qquad \hbox{$\zeta \in \Upsilon_{4n+2}$,}
        \end{array}
      \right.
\end{equation*}
and the orientation of $\Upsilon$ is shown in Figure \ref{fig:jumps-Psi}.

\item As $\zeta \to \infty$, we have
\begin{equation}\label{eq:asyPsi}
 \Psi(\zeta)= \left(I+ \frac{\Psi_1(x)}{\zeta} +\mathcal O(\zeta^{-2}) \right) e^{-i\Xi(\zeta)\sigma_3}
\end{equation}
for some function $\Psi_1$, where
\begin{equation}\label{def:Xi}
\Xi(\zeta)=\Xi(x,\zeta):=\frac{(2\zeta)^{2n+1}}{4n+2}+x\zeta.
\end{equation}

\item $\Psi(\zeta)$ is bounded near the origin.
\end{enumerate}
\end{rhp}
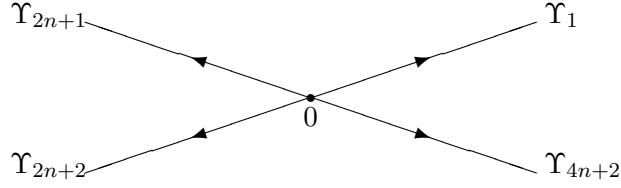
\begin{figure}[t]
\begin{center}
   \setlength{\unitlength}{1truemm}
   \begin{picture}(100,70)(-5,2)
       \put(40,40){\line(-3,-1){30}}
       \put(40,40){\line(-3,1){30}}
       \put(40,40){\line(3,1){30}}
       \put(40,40){\line(3,-1){30}}

       \put(25,45){\thicklines\vector(-3,1){1}}
       \put(25,35){\thicklines\vector(-3,-1){1}}
       \put(55,45){\thicklines\vector(3,1){1}}
       \put(55,35){\thicklines\vector(3,-1){1}}

       \put(39,36.3){$0$}
       \put(40,40){\thicklines\circle*{1}}

       \put(71,50){$\Upsilon_1$}
       \put(71,30){$\Upsilon_{4n+2}$}
       \put(0,50){$\Upsilon_{2n+1}$}
       \put(0,30){$\Upsilon_{2n+2}$}

\end{picture}
   \caption{The jump contour $\Upsilon$ for the RH problem \ref{rhp:psi} for $\Psi$.}
   \label{fig:jumps-Psi}
\end{center}
\end{figure}

Let $\Psi$ be a solution of the above RH problem, by \cite[Proposition 2.3]{CCG2019}, it follows that the function
\begin{equation}\label{eq:qPsi}
q_{\AS,n}(x;\rho):=\left\{
             \begin{array}{ll}
               2i(\Psi_1)_{12}(x)=-2i(\Psi_1)_{21}(x), & \hbox{$n$ odd,} \\
               2i(\Psi_1)_{12}(-x)=-2i(\Psi_1)_{21}(-x), & \hbox{$n$ even,}
             \end{array}
           \right.
\end{equation}
is real for $x\in \mathbb{R}$ and $-1<\rho<1$, where $\Psi_1(x)$ is given in \eqref{eq:asyPsi}. Moreover, it satisfies the Painlev\'{e} II hierarchy \eqref{def:PIIhierar} and the boundary condition \eqref{eq:asypinfty}. We will prove Theorem \ref{prop:connection} by analysing RH problem \ref{rhp:psi} for large negative $x$ in Section \ref{proof:Theorem1} below.
\begin{remark}\label{rk:PII}
One can strengthen the asymptotic behavior \eqref{eq:asyPsi} to be
\begin{align}\label{mPasymptotics}
\Psi(\zeta) = \left(I + \sum_{j=1}^N \frac{\Psi_j(x)}{\zeta^j} + \Boh(\zeta^{-N-1})\right) e^{-i\Xi(\zeta)\sigma_3}, \qquad \zeta \to \infty,
\end{align}
uniformly for $x$ in compact subsets of $\C \setminus \Upsilon$, where $\Psi_j(x)$, $j=1,\ldots, N$, are smooth functions. If $\rho=0$, we have $\Psi(\zeta)\equiv e^{-i\Xi(\zeta)\sigma_3}$, which implies that $q_{\AS,n}\equiv 0$. Moreover, since it is readily seen that both $\sigma_3\Psi(x, \rho,\zeta)\sigma_3$ and $\Psi(x, -\rho,\zeta)$ satisfy the same RH problem, by \eqref{eq:qPsi}, it follows that
\begin{equation}\label{eq:qsym}
q_{\AS,n}(x;-\rho)=-q_{\AS,n}(x;\rho).
\end{equation}
\end{remark}

\section{Asymptotic analysis of the RH problem for $\Psi$}\label{proof:Theorem1}
In this section, we perform a Deift-Zhou steepest descent analysis \cite{DZ1993} to the RH problem \ref{rhp:psi}
for $\Psi$ as $x\to -\infty$. It consists of a series of explicit and invertible transformations which leads
to an RH problem tending to the identity matrix for large negative $x$.
\subsection{First transformation: $\Psi \to X$}
The first transformation is a rescaling and normalization of the RH problem for $\Psi$, which is defined by
\begin{equation}\label{def:X}
X(\zeta)=\Psi(|x|^{\frac{1}{2n}}\zeta)e^{i|x|^{\frac{2n+1}{2n}}\widetilde \Xi(\zeta)\sigma_3},
\end{equation}
where
\begin{equation}\label{def:Theta}
\widetilde \Xi(\zeta):=\frac{2^{2n}}{2n+1}\zeta^{2n+1}-\zeta.
\end{equation}

It is then straightforward to check that $X$ satisfies the following RH problem.
\begin{rhp}\label{rhp:X}
\hfill
\begin{enumerate}[label=\emph{(\alph*)}, ref=(\alph*)]
\item $X(\zeta)$ is defined and analytic in $\mathbb{C} \setminus
\Upsilon $,
where $\Upsilon$ is defined in \eqref{def:Upsilon}.

\item For $\zeta \in \Upsilon$, we have
\begin{equation*}
X_{+}(\zeta) = X_{-}(\zeta)J_{X}(\zeta),
\end{equation*}
where
\begin{align*}
J_{X}(\zeta)&=e^{-i|x|^{\frac{2n+1}{2n}}\widetilde \Xi(\zeta)\sigma_3}J_{\Psi}(\zeta)e^{i|x|^{\frac{2n+1}{2n}}\widetilde \Xi(\zeta)\sigma_3}
\\
&=\left\{
        \begin{array}{ll}
          \begin{pmatrix} 1 & 0 \\ \rho e^{2i|x|^{\frac{2n+1}{2n}}\widetilde \Xi(\zeta)} & 1 \end{pmatrix}, & \qquad \hbox{$\zeta \in \Upsilon_1$,} \\
          \begin{pmatrix} 1 & 0 \\ -\rho e^{2i|x|^{\frac{2n+1}{2n}}\widetilde \Xi(\zeta)} & 1 \end{pmatrix}, & \qquad \hbox{$\zeta \in \Upsilon_{2n+1}$,} \\
          \begin{pmatrix} 1 & \rho e^{-2i|x|^{\frac{2n+1}{2n}}\widetilde \Xi(\zeta)} \\ 0 & 1 \end{pmatrix}, & \qquad \hbox{$\zeta \in \Upsilon_{2n+2}$,} \\
          \begin{pmatrix} 1 & -\rho e^{-2i|x|^{\frac{2n+1}{2n}}\widetilde \Xi(\zeta)} \\ 0 & 1 \end{pmatrix}, & \qquad \hbox{$\zeta \in \Upsilon_{4n+2}$.}
        \end{array}
      \right.
\end{align*}

\item As $\zeta \to \infty$, we have
\begin{equation*}
 X(\zeta)= I+ \frac{\Psi_1(x)}{|x|^{\frac{1}{2n}}\zeta} +\mathcal O(\zeta^{-2}),
\end{equation*}
where $\Psi_1(x)$ is given in \eqref{eq:asyPsi}.

\item $X(\zeta)$ is bounded near the origin.
\end{enumerate}
\end{rhp}

\subsection{Second transformation: $X \to Y$}\label{sectrans}
In the second transformation we apply contour deformations. The four rays $\Upsilon_j$, $j=1,2n+1,2n+2,4n+2$, are replaced by their parallel lines emanating from some special points on the real line. More precisely, we replace $\Upsilon_1$ and $\Upsilon_{4n+2}$ by their parallel rays $\widetilde \Upsilon_1$ and $\widetilde \Upsilon_{4n+2}$ emanating from the point 1/2, and replace $\Upsilon_{2n+1}$ and $\Upsilon_{2n+2}$ by their parallel rays $\widetilde \Upsilon_{2n+1}$ and $\widetilde \Upsilon_{2n+2}$ emanating from the point $-1/2$; see Figure \ref{fig:jumpY} for an illustration.
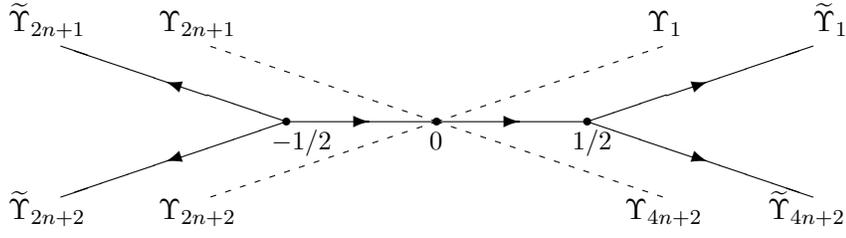
\begin{figure}[t]
\begin{center}
   \setlength{\unitlength}{1truemm}
   \begin{picture}(100,70)(-5,2)

       \put(60,40){\line(-1,0){40}}
       \put(60,40){\line(3,1){30}}
       \put(60,40){\line(3,-1){30}}
       \put(20,40){\line(-3,-1){30}}
       \put(20,40){\line(-3,1){30}}

       \dashline{0.8}(40,40)(10,30)
        \dashline{0.8}(40,40)(10,50)
         \dashline{0.8}(40,40)(70,30)
         \dashline{0.8}(40,40)(70,50)

       \put(5,45){\thicklines\vector(-3,1){1}}
       \put(5,35){\thicklines\vector(-3,-1){1}}
       \put(75,45){\thicklines\vector(3,1){1}}
       \put(75,35){\thicklines\vector(3,-1){1}}
       \put(30,40){\thicklines\vector(1,0){1}}
       \put(50,40){\thicklines\vector(1,0){1}}

       \put(39,36.3){\small{$0$}}
       \put(18,36.3){\small{$-1/2$}}
       \put(58,36.3){\small{$1/2$}}

       \put(40,40){\thicklines\circle*{1}}
       \put(20,40){\thicklines\circle*{1}}
       \put(60,40){\thicklines\circle*{1}}

       \put(68,52){$\Upsilon_1$}
       \put(65,27){$\Upsilon_{4n+2}$}
       \put(3,52){$\Upsilon_{2n+1}$}
       \put(3,27){$\Upsilon_{2n+2}$}
       \put(90,52){$\widetilde\Upsilon_1$}
       \put(-17,52){$\widetilde\Upsilon_{2n+1}$}
       \put(-17,27){$\widetilde\Upsilon_{2n+2}$}
       \put(84,27){$\widetilde\Upsilon_{4n+2}$}
\end{picture}
   \caption{The jump contour $\Sigma_{Y}$ for the RH problem \ref{rhp:Y} for $Y$.}
   \label{fig:jumpY}
\end{center}
\end{figure}

The second transformation is defined as follows.
\begin{equation}\label{def:Y}
Y(\zeta)=\left\{
           \begin{array}{ll}
             X(\zeta)\begin{pmatrix} 1 & 0 \\ \rho e^{2i|x|^{\frac{2n+1}{2n}}\widetilde \Xi(\zeta)} & 1 \end{pmatrix}, & \text{$\zeta$ between $\Upsilon_1$ and $\widetilde\Upsilon_1$}\\ & \text{and $\zeta$ between $\Upsilon_{2n+1}$ and $\widetilde\Upsilon_{2n+1}$, } \\
             X(\zeta)\begin{pmatrix} 1 & \rho e^{-2i|x|^{\frac{2n+1}{2n}}\widetilde \Xi(\zeta)} \\ 0 & 1 \end{pmatrix}, & \text{$\zeta$ between $\Upsilon_{2n+2}$ and $\widetilde\Upsilon_{2n+2}$}
 \\
& \text{and $\zeta$ between $\Upsilon_{4n+2}$ and $\widetilde\Upsilon_{4n+2}$,}  \\
             X(\zeta), & \hbox{elsewhere.}
           \end{array}
         \right.
\end{equation}

In view of the RH problem \ref{rhp:X} for $X$ and \eqref{def:Y}, it is readily seen that $Y$ satisfies the following RH problem.
\begin{rhp}\label{rhp:Y}
\hfill
\begin{enumerate}[label=\emph{(\alph*)}, ref=(\alph*)]
\item $Y(\zeta)$ is defined and analytic in $\mathbb{C} \setminus
\Sigma_Y $,
where
\begin{equation*}
\Sigma_Y:=\widetilde\Upsilon_1\cup\widetilde\Upsilon_{2n+1}\cup\widetilde\Upsilon_{2n+2}\cup\widetilde\Upsilon_{4n+2}\cup[-1/2,1/2];
\end{equation*}
see the solid lines in Figure \ref{fig:jumpY}.

\item For $\zeta \in \Sigma_Y$, we have
\begin{equation*}
Y_{+}(\zeta) = Y_{-}(\zeta)J_{Y}(\zeta),
\end{equation*}
where
\begin{equation}\label{def:JY}
J_{Y}(\zeta)=\left\{
        \begin{array}{ll}
          \begin{pmatrix} 1 & 0 \\ \rho e^{2i|x|^{\frac{2n+1}{2n}}\widetilde \Xi(\zeta)} & 1 \end{pmatrix}, & \qquad \hbox{$\zeta \in \widetilde\Upsilon_1$,} \\
          \begin{pmatrix} 1 & 0 \\ -\rho e^{2i|x|^{\frac{2n+1}{2n}}\widetilde \Xi(\zeta)} & 1 \end{pmatrix}, & \qquad \hbox{$\zeta \in \widetilde\Upsilon_{2n+1}$,} \\
          \begin{pmatrix} 1 & \rho e^{-2i|x|^{\frac{2n+1}{2n}}\widetilde \Xi(\zeta)} \\ 0 & 1 \end{pmatrix}, & \qquad \hbox{$\zeta \in \widetilde\Upsilon_{2n+2}$,} \\
          \begin{pmatrix} 1 & -\rho e^{-2i|x|^{\frac{2n+1}{2n}}\widetilde \Xi(\zeta)} \\ 0 & 1 \end{pmatrix}, & \qquad \hbox{$\zeta \in \widetilde\Upsilon_{4n+2}$,} \\
          \begin{pmatrix} 1-\rho^2 & -\rho e^{-2i|x|^{\frac{2n+1}{2n}}\widetilde \Xi(\zeta)} \\ \rho e^{2i|x|^{\frac{2n+1}{2n}}\widetilde \Xi(\zeta)} & 1 \end{pmatrix}, & \qquad \hbox{$\zeta \in (-\frac12,\frac12)$.}
        \end{array}
      \right.
\end{equation}

\item As $\zeta \to \infty$, we have
\begin{equation*}
 Y(\zeta)= I+ \frac{\Psi_1(x)}{|x|^{\frac{1}{2n}}\zeta} + \mathcal O(\zeta^{-2}),
\end{equation*}
where $\Psi_1(x)$ is given in \eqref{eq:asyPsi}.

\item $Y(\zeta)$ is bounded near $\zeta=\pm 1/2$.
\end{enumerate}
\end{rhp}

\subsection{Third transformation: $Y \to T$}

As $x \to -\infty$, it comes out that $J_Y(\zeta)$ tends to the identity matrix exponentially fast except for $\zeta \in (-1/2,1/2)$, as evidenced in Figure \ref{fig:jumpT}. Since $\im \widetilde \Xi(\zeta)=0$ for $\zeta \in (-1/2,1/2)$, we have that
$J_Y(\zeta)$ is highly oscillatory for large negative $x$. The third transformation then involves the so-called lens opening, which is based on the following factorization:
\begin{equation*}
J_{Y}(\zeta)
=
\begin{pmatrix}
1 & 0
\\
\frac{\rho}{1-\rho^2}e^{2i|x|^{\frac{2n+1}{2n}}\widetilde \Xi(\zeta)} & 1
\end{pmatrix}
\begin{pmatrix}
1-\rho^2 & 0
\\
0 & \frac{1}{1-\rho^2}
\end{pmatrix}
\begin{pmatrix}
1 & -\frac{\rho}{1-\rho^2}e^{-2i|x|^{\frac{2n+1}{2n}}\widetilde \Xi(\zeta)}
\\
0 & 1
\end{pmatrix}
\end{equation*}
for $\zeta\in(-1/2,1/2)$.

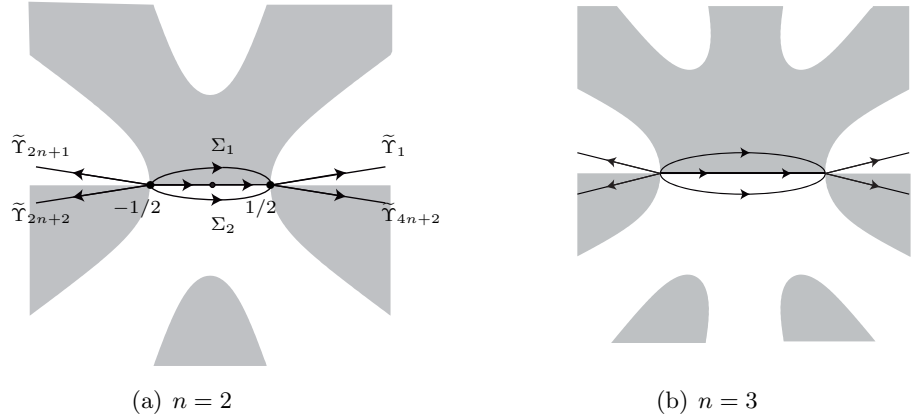
\begin{figure}
\centering
\subfigure[$n=2$]{
    \begin{minipage}[b]{0.3\linewidth} 
    \begin{overpic}[width=5cm]{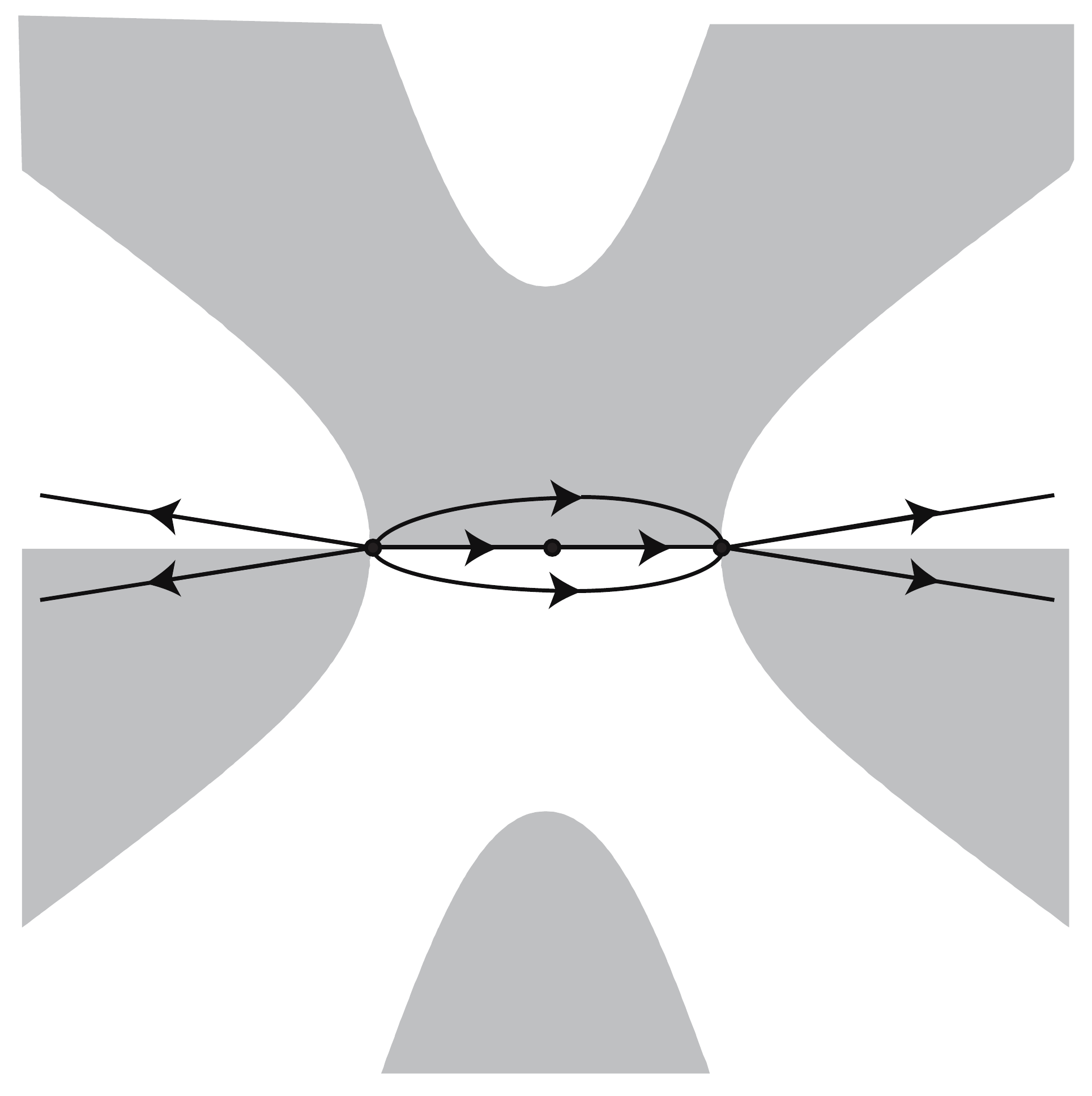}\vspace{1pt}
    \put(50,38){\tiny $\Sigma_2$}
      \put(50,59){\tiny $\Sigma_1$}
      \put(24,42.5){\tiny $-1/2$}
       \put(59,42.5){\tiny $1/2$}
        \put(-3,40.5){\tiny $\widetilde\Upsilon_{2n+2}$}
         \put(-3,58.5){\tiny $\widetilde\Upsilon_{2n+1}$}
         \put(95,40.5){\tiny $\widetilde\Upsilon_{4n+2}$}
         \put(95,58.5){\tiny $\widetilde\Upsilon_{1}$}
      \put(34,50){\circle*{2}}
       \put(65.8,50){\circle*{2}} 
    \end{overpic}
\end{minipage}}
\quad 
\qquad
\qquad
\subfigure[$n=3$]{
    \begin{minipage}[b]{0.35\linewidth}
    \includegraphics[width=5.3cm]{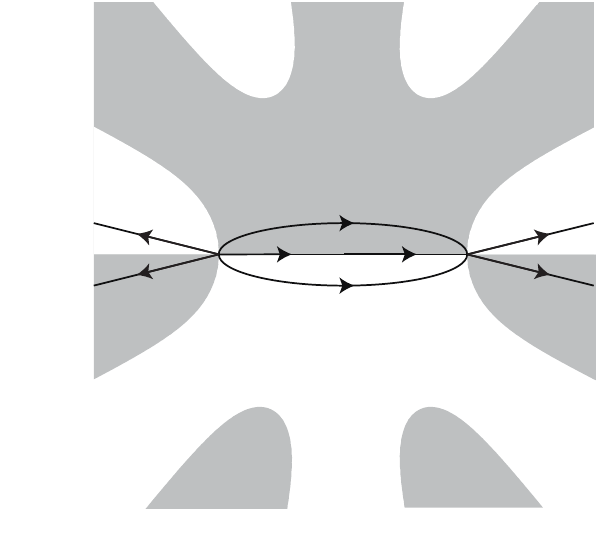}\vspace{1pt}
    \end{minipage}
}
\caption{\label{fig:jumpT}
       The jump contour $\Sigma_{T}$ for the RH problem \ref{rhp:T} for $T$. The shaded areas indicate the regions where $\re i \widetilde \Xi > 0$ for $n=2$ (left) and $n=3$ (right).}
\end{figure}

By opening lens around $(-1/2,1/2)$ as shown in Figure \ref{fig:jumpT} so that $\re i \widetilde \Xi(\zeta)>0$ for $\zeta \in \Sigma_1$ and $\re i \widetilde \Xi(\zeta) < 0$ for $\zeta \in \Sigma_2$, the third transformation is defined by
\begin{equation} \label{def:T}
  T(\zeta) =
  \begin{cases}
    Y(\zeta), & \text{$\zeta$ outside the lens}, \\
    Y(\zeta)
    \begin{pmatrix}
1 & \frac{\rho}{1-\rho^2}e^{-2i|x|^{\frac{2n+1}{2n}}\widetilde \Xi(\zeta)}
\\
0 & 1
\end{pmatrix},
    & \text{$\zeta$ in the upper part of the lens}, \\
    Y(\zeta)
    \begin{pmatrix}
1 & 0
\\
\frac{\rho}{1-\rho^2}e^{2i|x|^{\frac{2n+1}{2n}}\widetilde \Xi(\zeta)} & 1
\end{pmatrix},
    & \text{$\zeta$ in the lower part of the lens.}
  \end{cases}
\end{equation}
It is then straightforward to check that $T$ satisfies the following RH problem
\begin{rhp}\label{rhp:T}
\hfill
\begin{enumerate}[label=\emph{(\alph*)}, ref=(\alph*)]
\item $T(\zeta)$ is defined and analytic in $\mathbb{C} \setminus
\Sigma_T $,
where
\begin{equation}\label{def:sigmaT}
\Sigma_T:=\widetilde\Upsilon_1\cup\widetilde\Upsilon_{2n+1}\cup\widetilde\Upsilon_{2n+2}\cup\widetilde\Upsilon_{4n+2}\cup[-1/2,1/2]\cup \Sigma_1 \cup \Sigma_2;
\end{equation}
see Figure \ref{fig:jumpT} for an illustration.

\item For $\zeta \in \Sigma_T$, we have
\begin{equation*}
T_{+}(\zeta) = T_{-}(\zeta)J_{T}(\zeta),
\end{equation*}
where
\begin{equation}\label{def:JT}
J_{T}(\zeta)=\left\{
        \begin{array}{ll}
          \begin{pmatrix} 1 & -\frac{\rho}{1-\rho^2}e^{-2i|x|^{\frac{2n+1}{2n}}\widetilde \Xi(\zeta)} \\ 0 & 1 \end{pmatrix}, & \qquad \hbox{$\zeta \in \Sigma_1$,} \\
           \begin{pmatrix} 1 & 0 \\ \frac{\rho}{1-\rho^2}e^{2i|x|^{\frac{2n+1}{2n}}\widetilde \Xi(\zeta)} & 1 \end{pmatrix}, & \qquad \hbox{$\zeta \in \Sigma_2$,} \\
           \begin{pmatrix} 1-\rho^2 & 0 \\ 0 & \frac{1}{1-\rho^2} \end{pmatrix}, & \qquad \hbox{$\zeta \in (-\frac12,\frac12)$,} \\
           J_Y(\zeta), & \qquad \hbox{elsewhere,} \\
         \end{array}
      \right.
\end{equation}
and where $J_Y(\zeta)$ is given in \eqref{def:JY}.

\item As $\zeta \to \infty$, we have
\begin{equation*}
 T(\zeta)= I+ \frac{\Psi_1(x)}{|x|^{\frac{1}{2n}}\zeta} +\mathcal O(\zeta^{-2}),
\end{equation*}
where $\Psi_1(x)$ is given in \eqref{eq:asyPsi}.

\item $T(\zeta)$ is bounded near $\zeta=\pm 1/2$.
\end{enumerate}
\end{rhp}

\subsection{Global parametrix}\label{sec:global}
It is now easily seen that the jump matrix $J_T$ tends to the identity matrix except for $\zeta \in (-1/2,1/2)$. We are then led to consider the following RH problem for the global parametrix $P^{(\infty)}$.
\begin{rhp}\label{rhp:N}
\hfill
\begin{enumerate}[label=\emph{(\alph*)}, ref=(\alph*)]
\item $P^{(\infty)}(\zeta)$ is defined and analytic in $\mathbb{C} \setminus [-1/2,1/2]$.

\item For $\zeta \in (-1/2,1/2)$, we have
\begin{equation*}
P^{(\infty)}_{+}(\zeta) = P^{(\infty)}_{-}(\zeta)\begin{pmatrix} 1-\rho^2 & 0 \\ 0 & \frac{1}{1-\rho^2} \end{pmatrix}.
\end{equation*}

\item As $\zeta \to \infty$, we have
\begin{equation*}
P^{(\infty)}(\zeta)= I +\mathcal O(\zeta^{-1}).
\end{equation*}

\end{enumerate}
\end{rhp}

One can check that the solution to the above RH problem is explicitly given by
\begin{equation}\label{def:N}
P^{(\infty)}(\zeta)=\left(\frac{\zeta+1/2}{\zeta-1/2}\right)^{\nu \sigma_3}, \qquad \nu=-\frac{1}{2\pi i}\ln (1-\rho^2),
\end{equation}
where the branch cut of $\left(\frac{\zeta+1/2}{\zeta-1/2}\right)^{\nu}$ is taken along the interval $[-1/2,1/2]$ such that $\left(\frac{\zeta+1/2}{\zeta-1/2}\right)^{\nu } \to 1$ as $\zeta \to \infty$.

\subsection{Local parametrices near $\zeta=\pm 1/2$}\label{sec:local}
Since the convergence of $J_T$ to the identity matrix is not uniform near $\zeta=\pm 1/2$, we have to build local parametrices around these two points. Denote by $D(z_0,\delta)$ a fixed open disc centered at $z_0$ with radius $\delta>0$, the local parametrix near $\zeta=1/2$ reads as follows
\begin{rhp}\label{rhp:pr}
\hfill
\begin{enumerate}[label=\emph{(\alph*)}, ref=(\alph*)]
\item $P^{(\frac12)}(\zeta)$ is defined and analytic in $D(1/2,\delta) \setminus \Sigma_T$, where $\Sigma_T$ is defined \eqref{def:sigmaT}.

\item For $\zeta \in D(1/2,\delta) \cap \Sigma_T$, we have
\begin{equation*}
P^{(\frac12)}_{+}(\zeta) = P^{(\frac12)}_{-}(\zeta)J_T(\zeta),
\end{equation*}
where $J_T(\zeta)$ is defined in \eqref{def:JT}.

\item As $x \to -\infty$, $P^{(\frac12)}(\zeta)$ matches $P^{(\infty)}(\zeta)$ on the boundary $\partial D(1/2,\delta)$ of $D(1/2,\delta)$, i.e.,
\begin{equation}\label{eq:asypr}
 P^{(\frac12)}(\zeta)= (I+ \mathcal O(|x|^{-\frac{2n+1}{4n}}))P^{(\infty)}(\zeta),
\end{equation}
where $P^{(\infty)}(\zeta)$ is given in \eqref{def:N}.
\end{enumerate}
\end{rhp}

We can construct $P^{(\frac12)}(\zeta)$ explicitly by using the parabolic cylinder parametrix $\Psi^{(\PC)}$ introduced in Appendix \ref{sec:para}, following the strategy in \cite[Section 9.4]{FIKN2006}. To proceed, we define
\begin{equation}\label{def:eta}
\eta(\zeta):=2(-i\widetilde \Xi(\zeta)+i\widetilde \Xi(1/2))^{\frac12},
\end{equation}
where $\widetilde \Xi(\zeta)$ is defined in \eqref{def:Theta} and the branch cut of $(\cdot)^{1/2}$ is chosen such that $\arg (\zeta-1/2)\in (-\pi,\pi)$. It is readily seen that
\begin{equation}\label{eq:etalocal}
\eta(\zeta)\sim e^{\frac{3\pi i}{4}}2\sqrt{2n}(\zeta-\frac12), \qquad \zeta \to \frac12.
\end{equation}
Let $\Psi^{(\PC)}(\zeta;\nu)$ be the parabolic cylinder parametrix with $\nu$ given in \eqref{def:N} (see Appendix \ref{sec:para} below), we set, for $\zeta \in D(\frac12,\delta) \setminus \Sigma_T$,
\begin{equation}\label{def:phalf}
P^{(\frac12)}(\zeta)=E(\zeta)\Psi^{(\PC)}(|x|^{\frac{2n+1}{4n}}\eta(\zeta);\nu)e^{i|x|^{\frac{2n+1}{2n}}\widetilde \Xi(\zeta)\sigma_3}\left(i\frac{h_1}{\rho}\right)^{\sigma_3/2}
\begin{pmatrix}
1 & 0
\\
0 & -i
\end{pmatrix}
\end{equation}
where $\eta(\zeta)$ is defined in \eqref{def:eta}, $h_1=\frac{\sqrt{2 \pi}}{\Gamma(-\nu)}e^{i\pi \nu}$, and
\begin{equation*}
E(\zeta):=\begin{pmatrix}
1 & 0
\\
0 & i
\end{pmatrix}
(\beta(\zeta))^{\sigma_3}\left(i\frac{2h_1}{\rho}\right)^{-\sigma_3/2}e^{in|x|^{\frac{2n+1}{2n}}\sigma_3/(2n+1)}
\begin{pmatrix}
|x|^{\frac{2n+1}{4n}}\eta(\zeta) & 1
\\
1 & 0
\end{pmatrix}
\end{equation*}
and where
\begin{equation}\label{def:beta}
\beta(\zeta):=\left(|x|^{\frac{2n+1}{4n}}\eta(\zeta)\frac{\zeta+1/2}{\zeta-1/2}\right)^{\nu}.
\end{equation}
By \eqref{eq:etalocal}, it follows that $\beta(\zeta)$ is analytic near $\zeta=1/2$ with
$$\beta(1/2)=(2\sqrt{2n}|x|^{\frac{2n+1}{4n}})^\nu e^{3\nu\pi i/4},$$ which also implies that $E(\zeta)$ is an analytic prefactor in $ D(\frac12,\delta)$. From the RH problem \ref{rhp:PC} for $\Psi^{(\PC)}$, it is straightforward to show (cf. \cite{FIKN2006} for the case $k=1$) that $P^{(\frac12)}(\zeta)$ defined in \eqref{def:phalf} indeed solves the RH problem \ref{rhp:pr}. Moreover, the matching condition \eqref{eq:asypr} now reads
\begin{multline}\label{eq:prmatching}
 P^{(\frac12)}(\zeta)=
\Bigg(I+
\begin{pmatrix}
0 & -\frac{\rho\nu}{h_1}e^{2n|x|^{\frac{2n+1}{2n}}\frac{i}{(2n+1)}}\frac{\beta(\zeta)^2}{\eta(\zeta)}
\\
-\frac{h_1}{\rho}e^{-2n|x|^{\frac{2n+1}{2n}}\frac{i}{(2n+1)}}\frac{1}{\beta(\zeta)^{2}\eta(\zeta)} & 0
\end{pmatrix}
|x|^{-\frac{2n+1}{4n}}
\\
+\mathcal O(|x|^{-\frac{2n+1}{2n}})\Bigg)P^{(\infty)}(\zeta),
\end{multline}
as $x\to -\infty$, uniformly for $\zeta \in \partial D(1/2, \delta)$.

Similarly, near $\zeta=-1/2$, we intend to find a function $P^{(-\frac12)}$ satisfying the following RH problem.
\begin{rhp}\label{rhp:pl}
\hfill
\begin{enumerate}[label=\emph{(\alph*)}, ref=(\alph*)]
\item $P^{(-\frac12)}(\zeta)$ is defined and analytic in $D(-1/2,\delta) \setminus \Sigma_T$, where $\Sigma_T$ is defined \eqref{def:sigmaT}.

\item For $\zeta \in D(-1/2,\delta) \cap \Sigma_T$, we have
\begin{equation*}
P^{(-\frac12)}_{+}(\zeta) = P^{(-\frac12)}_{-}(\zeta)J_T(\zeta),
\end{equation*}
where $J_T(\zeta)$ is defined in \eqref{def:JT}.

\item As $x \to -\infty$, we have, for $\zeta \in \partial D(-1/2,\delta)$
\begin{equation}\label{eq:asypl}
 P^{(-\frac12)}(\zeta)= (I+ \mathcal O(|x|^{-\frac{2n+1}{4n}}))P^{(\infty)}(\zeta),
\end{equation}
where $P^{(\infty)}(\zeta)$ is given in \eqref{def:N}.
\end{enumerate}
\end{rhp}

From the symmetry of $J_{T}$, one can check directly that
\begin{equation}\label{def:Pl}
 P^{(-\frac12)}(\zeta)=\sigma_1  P^{(\frac12)}(-\zeta) \sigma_1
\end{equation}
with $\sigma_1$ and $P^{(\frac12)}(\zeta)$ given \eqref{def:Pauli} and \eqref{def:phalf} solves the above RH problem.

\subsection{Final transformation}
The final transformation is defined by
\begin{equation}\label{def:R}
R(\zeta)=\left\{
           \begin{array}{ll}
             T(\zeta)P^{(\infty)}(\zeta)^{-1}, & \hbox{$\zeta \in \mathbb{C} \setminus \{D(\frac12,\delta)\cup D(-\frac12,\delta)\cup \Sigma_T \}$,} \\
             T(\zeta)P^{(\frac12)}(\zeta)^{-1}, & \hbox{$\zeta  \in  D(\frac12,\delta)$,} \\
             T(\zeta)P^{(-\frac12)}(\zeta)^{-1}, & \hbox{$\zeta \in D(-\frac12,\delta)$.}
           \end{array}
         \right.
\end{equation}
It is then easily seen that $R$ satisfies the following RH problem.

\begin{rhp}\label{rhp:R}
\hfill
\begin{enumerate}[label=\emph{(\alph*)}, ref=(\alph*)]
\item $R(\zeta)$ is defined and analytic in $\mathbb{C} \setminus \Sigma_R$,
where
$$
\Sigma_R:=\Sigma_T \cup \partial D(\frac12,\delta) \cup \partial D(-\frac12,\delta) \setminus \{(-\frac12,\frac12)\cup D(\frac12,\delta)\cup D(-\frac12,\delta)\}.
$$

\item For $\zeta \in \Sigma_R$, we have
\begin{equation}\label{jumps:R}
R_{+}(\zeta) = R_{-}(\zeta)J_R(\zeta),
\end{equation}
where
\begin{equation}\label{def:JR}
J_R(\zeta)=\left\{
             \begin{array}{ll}
              P^{(\frac12)}(\zeta)P^{(\infty)}(\zeta)^{-1}, & \hbox{$\zeta \in \partial D(\frac12,\delta)$,} \\
              P^{(-\frac12)}(\zeta)P^{(\infty)}(\zeta)^{-1}, & \hbox{$\zeta \in \partial D(-\frac12,\delta)$,} \\
              P^{(\infty)}(\zeta)J_T(\zeta)P^{(\infty)}(\zeta)^{-1}, & \hbox{$\zeta \in \Sigma_R  \setminus \left\{\partial D(\frac12,\delta)\cup \partial D(-\frac12,\delta)\right\}$,}
             \end{array}
           \right.
\end{equation}
and where $J_T(\zeta)$ is defined in \eqref{def:JT} and the orientation of $\partial D(\pm1/2,\delta)$ is taken in a clock-wise manner.

\item As $\zeta \to \infty$, we have,
\begin{equation*}
 R(\zeta)=I+\mathcal O(\zeta^{-1}).
\end{equation*}
\end{enumerate}
\end{rhp}

From \eqref{def:JT}, \eqref{eq:asypr} and \eqref{eq:asypl}, it is readily seen from \eqref{jumps:R} that, as $x\to -\infty$,
\begin{equation}\label{eq:JRest}
J_R(\zeta)=\left\{
             \begin{array}{ll}
               I+\Boh(|x|^{-\frac{2n+1}{4n}})), & \hbox{$\zeta \in \partial D(\frac12,\delta) \cup \partial D(-\frac12,\delta)$,}
               \\
               I+\Boh(e^{-c|x|^{-\frac{2n+1}{2n}})}) , & \hbox{$\zeta \in \Sigma_R \setminus \{\partial D(\frac12,\delta) \cup \partial D(-\frac12,\delta)\}$,}
             \end{array}
           \right.
\end{equation}
for some $c>0$. An appeal to the standard small norm arguments (cf. \cite{DZ1993,Deiftbook}) then shows that
\begin{equation}\label{eq:Rest}
R(\zeta)=I+\Boh(|x|^{-\frac{2n+1}{4n}}), \qquad x\to -\infty,
\end{equation}
uniformly for $\zeta \in \mathbb{C} \setminus \Sigma_R$.

For later use, it is worth noting that $R$ is also characterized by the following integral equation
\begin{equation}\label{eq:Rintegral}
R(\zeta)=I+\frac{1}{2\pi i}\int_{\Sigma_R}\frac{R_-(s)(J_R(s)-I)}{s-\zeta}ds.
\end{equation}

\section{Asymptotic analysis of the RH problem for $m$: {\bf Case I}} \label{proof:th2th3-1}
In this section, we analyse the RH problem \ref{rhp-org} for $m$ as $t\to \infty$ in {\bf Case~ I}, that is,
\begin{align*}
\{(x,t) \mid 0 \leq x \leq M t^{\frac{1}{2n+1}}\}, \qquad n \text{ is odd},
\end{align*}
where $M>0$. In this case, the critical points of phase function $\Theta$ in \eqref{def:Thetaxi} are
\begin{equation*}
\lambda_0^{(n,j)} = \sqrt[2n]{\frac{\xi}{(2n+1) 2^{2n} }} e^{i\frac{(2j-1)\pi}{2n}},\qquad j=1,2,\ldots,2n,
\end{equation*}
which are not on the real line and approach $0$ at least as fast as $t^{-1/(2n+1)}$ as $t \to \infty$. Thus, it follows that $|\lambda_0^{(n,j)}| \leq C t^{-1/(2n+1)}$ for some $C>0$. For $n=3$ and $\xi=1$, an illustration of the six critical points and signature of $\re \Theta$ are shown in Figure \ref{criticalpointsfig}.

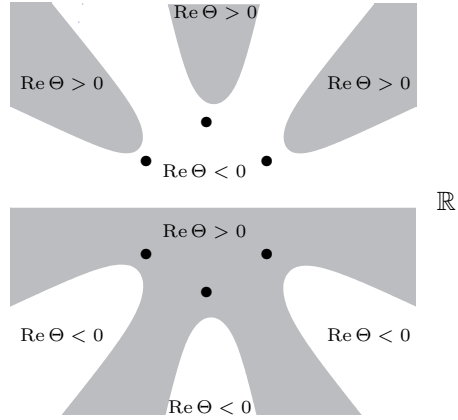
\begin{figure}
\begin{center}
\begin{overpic}[width=0.5\textwidth]{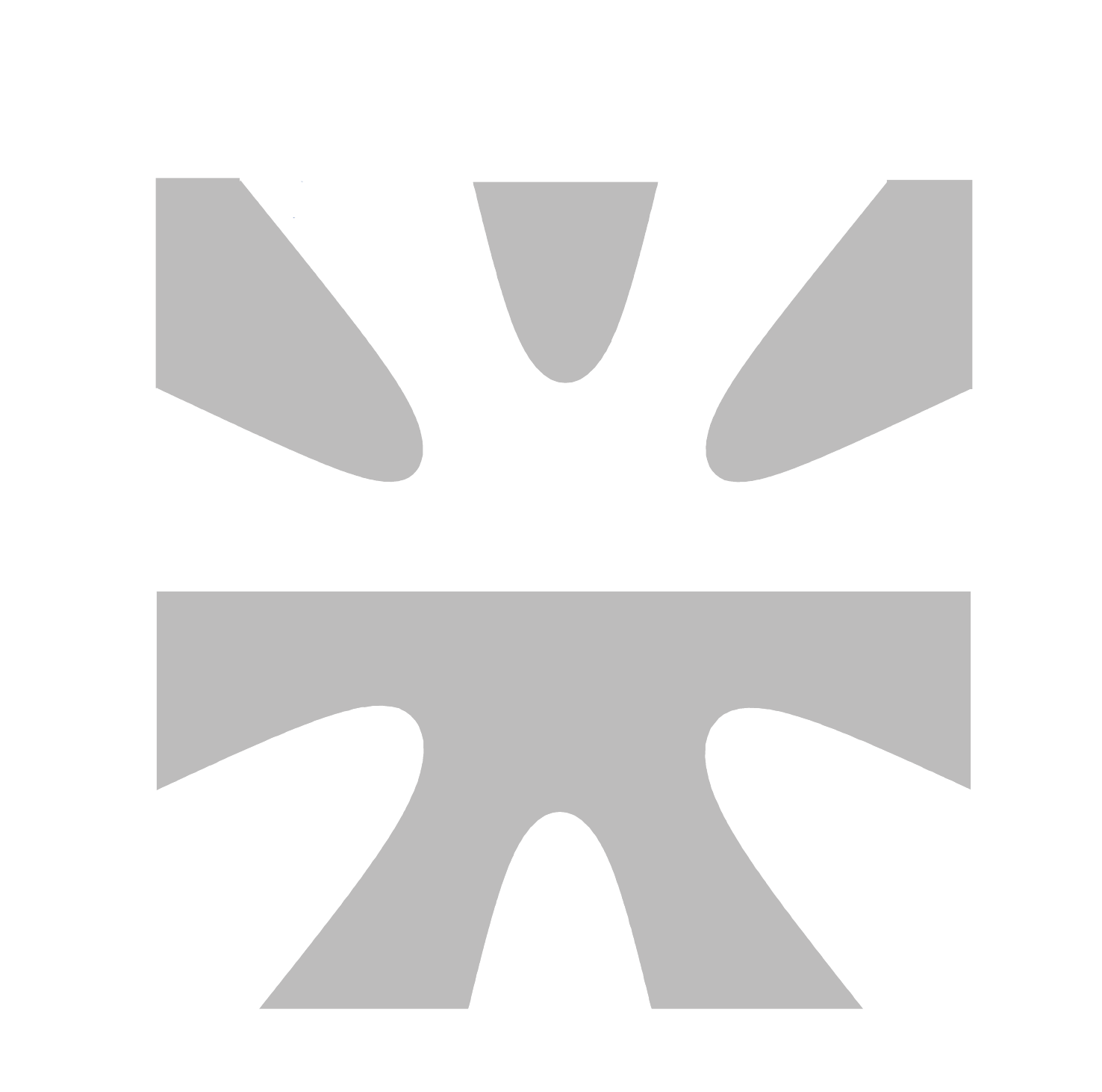}
      \put(92,45){\small $\R$}
      \put(50,30){\circle*{2}}
      \put(50,61){\circle*{2}}
      \put(39,54){\circle*{2}}
       \put(39,37){\circle*{2}}
        \put(61,37){\circle*{2}}
        \put(61,54){\circle*{2}}

      \put(42,51){\tiny $\re \Theta < 0$}
      \put(42,40){\tiny $\re \Theta > 0$}
      \put(72,67){\tiny $\re \Theta > 0$}
      \put(72,21){\tiny $\re \Theta < 0$}
       \put(44,80){\tiny $\re \Theta > 0$}
      \put(43,8){\tiny $\re \Theta <0$}
       \put(16,67){\tiny $\re \Theta > 0$}
      \put(16,21){\tiny $\re \Theta <0$}
    \end{overpic}
     \caption{\label{criticalpointsfig}
        The critical points $\lambda^{(3,j)}_0$, $j=1,\ldots,6$, and signature of $\re \Theta$ in the complex $\lambda$-plane for $\xi=1$ in {\bf Case I}.}
      \end{center}
\end{figure}

\subsection{First transformation: $m \to m^{(1)}$}\label{sec:firstranI}
The first transformation involves deformation of the jump matrix along the real axis and we need to decompose the reflection coefficient $r$ into an analytic part $r_a$ and a small remainder $r_r$. To proceed, recall the rays $\Upsilon_j$, $j=1,2n+1,2n+2,4n+2$, defined in \eqref{def:Upsilonj} with the orientations shown in Figure \ref{fig:jumps-Psi}. By reversing the orientations of $\Upsilon_{j}$, $j=2n+1,2n+2$, we obtain two new rays denoted by $\Upsilon_{j}^*$. Clearly, the three curves $\Upsilon_{1}\cup\Upsilon_{2n+1}^*$, $\mathbb{R}$ and $\Upsilon_{2n+2}^*\cup\Upsilon_{4n+1}$ formulate the boundaries of two open subsets $V$ and $V^*$, as shown in Figure \ref{Gamma1geq.pdf}. The decomposition of $r$ is then given in the following lemma. Since the proof is similar to that of \cite[Lemma 2.1]{CL-2019} or \cite[Lemma 7.1]{hl-2018}, we omit the details here.

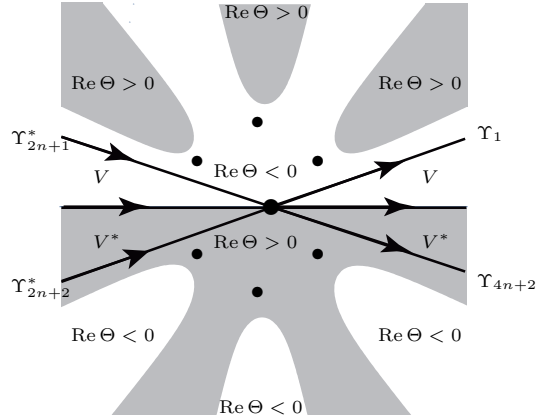
\begin{figure}
\begin{center}
\begin{overpic}[width=.5\textwidth]{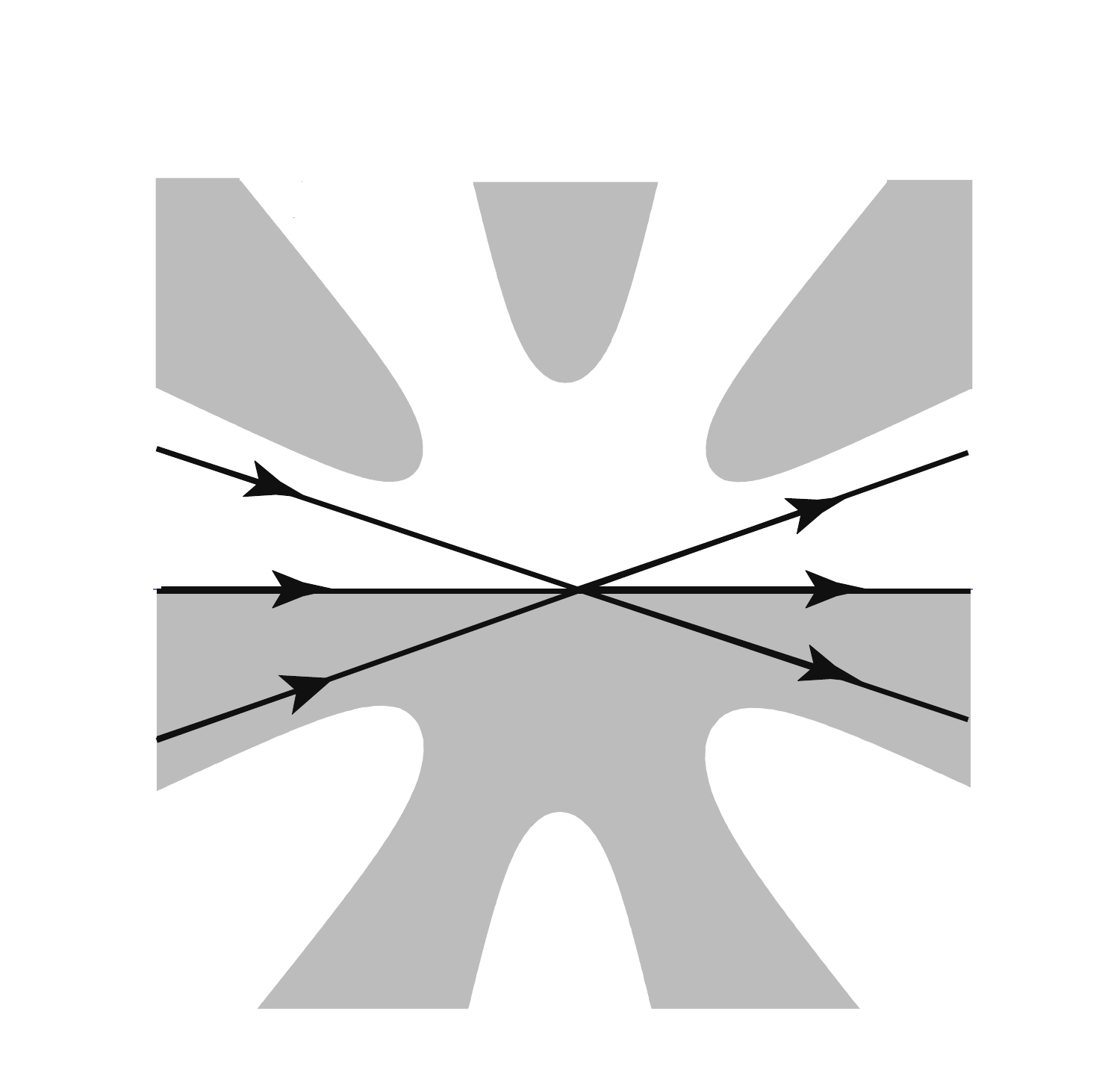}
      \put(50,30){\circle*{2}}
      \put(50,61){\circle*{2}}
      \put(52.5,45.5){\circle*{3}}
      \put(39,54){\circle*{2}}
       \put(39,37){\circle*{2}}
        \put(61,37){\circle*{2}}
        \put(61,54){\circle*{2}}

      \put(42,51){\tiny $\re \Theta < 0$}
      \put(42,38){\tiny $\re \Theta > 0$}
      \put(72,67){\tiny $\re \Theta > 0$}
      \put(72,21){\tiny $\re \Theta < 0$}
       \put(44,80){\tiny $\re \Theta > 0$}
      \put(43,8){\tiny $\re \Theta <0$}
       \put(16,67){\tiny $\re \Theta > 0$}
      \put(16,21){\tiny $\re \Theta <0$}
 \put(5,57){\tiny $\Upsilon_{2n+1}^*$}
 \put(5,30){\tiny $\Upsilon_{2n+2}^*$}
  \put(90,58){\tiny $\Upsilon_{1}$}
 \put(90,31){\tiny $\Upsilon_{4n+2}$}
      \put(20,50){\tiny $V$}
      \put(80,50){\tiny $V$}
       \put(20,38){\tiny $V^*$}
      \put(80,38){\tiny $V^*$}
    \end{overpic}
    \caption{\label{Gamma1geq.pdf}
        The open subsets $V$, $V^*$ in {\bf Case I} and the contour $\Gamma^{(1)}$ for the RH problem \ref{RHm1} for $m^{(1)}$.}
     \end{center}
\end{figure}

\begin{lemma}[Analytic approximation for $\xi \geq 0$]\label{decompositionlemmageq}
Let $r$ be the reflection coefficient defined in \eqref{def:r}, we have
\begin{equation}\label{def:decom}
 r(\lambda) = r_{a}(x,t, \lambda) + r_{r}(x, t, \lambda), \qquad t \geq C(n), \quad \lambda \in \R,
\end{equation}
where $C(n)$ is a positive constant depending on $n$, and for a fixed $N\in\N$ the functions $r_{a}$ and $r_{r}$ have the following properties:
\begin{enumerate}[label=\emph{(\alph*)}, ref=(\alph*)]
\item For each $t\geq C(n)$, $r_{a}(t, \lambda)$ is defined and continuous for $\lambda \in \bar{V}$ and analytic for $\lambda \in V$, where $\bar{V}$ denotes the closure of $V$.

\item There exists a constant $C>0$ such that
\begin{align*}
& |r_{a}(x, t, \lambda)| \leq \frac{C}{1 + |\lambda|} e^{\frac{t}{4}|\re \Theta(\xi,\lambda)|}, \qquad
  \lambda \in \bar{V},
\end{align*}
and
\begin{align}\label{raat1IV}
\bigg|r_{a}(x, t, \lambda) - \sum_{j=0}^N \frac{r^{(j)}(0)}{j!} \lambda^j\bigg| \leq C |\lambda|^{N+1} e^{\frac{t}{4}|\re \Theta(\xi,\lambda)|}, \qquad \lambda \in \bar{V},
\end{align}
for each $\xi \geq 0$ and $t \geq C(n)$.
\item As $t \to \infty$, the $L^1$ and $L^\infty$ norms of $r_{r}(x, t, \cdot)$ on $\R$ are $\Boh(t^{-N})$.

\item The functions $r_a$ and $r_{r}$ satisfy the following symmetries:
\begin{align*}
\begin{cases}
r_{a}(x, t, \lambda) = -r_{a}^*(x, t, -\lambda), &\lambda \in \bar{V},\\
r_{r}(x, t, \lambda) = -r_{r}^*(x, t, -\lambda), & \lambda \in \R.
\end{cases}
\end{align*}
\end{enumerate}
\end{lemma}

With the aid of decomposition \eqref{def:decom}, we define a sectionally analytic matrix-valued function $m^{(1)}$ by
\begin{align}\label{m1def}
m ^{(1)}(x,t,\lambda) = m(x,t,\lambda)G(x,t,\lambda),
\end{align}
where
\begin{align}\label{GdefII}
G(x,t,\lambda): = \begin{cases}
\begin{pmatrix}
 1 & 0 \\
 -r_{a} e^{t\Theta}  & 1
\end{pmatrix}, & \lambda \in V,
	\\
\begin{pmatrix}
 1 & - r_{a}^* e^{-t\Theta}  \\
0 & 1
\end{pmatrix}, & \lambda \in V^*,
	\\
I, & \text{elsewhere}.
\end{cases}
\end{align}
By Lemma \ref{decompositionlemmageq} and the notation $E^p(D)$ introduced at the end of Section \ref{notation}, it is readily seen that
$$G(x,t,\cdot) \in I + (\dot{E}^2 \cap E^\infty)(V \cup V^*),$$
and $m$ satisfies the RH problem \ref{rhp-org} if and only if $m^{(1)}$ defined in \eqref{m1def} solves the following RH problem.
\begin{rhp}\label{RHm1}
\hfill
\begin{enumerate}[label=\emph{(\alph*)}, ref=(\alph*)]
\item $m^{(1)}(x, t, \cdot) \in I + \dot{E}^2(\C \setminus \Gamma^{(1)})$,
where
\begin{equation}\label{def:gamma1}
\Gamma^{(1)} := \R \cup \Upsilon_1\cup\Upsilon_{2n+1}^*\cup\Upsilon_{2n+2}^*\cup\Upsilon_{4n+2};
\end{equation}
see Figure \ref{Gamma1geq.pdf} for an illustration and the orientation.

\item For a.e. $\lambda \in \Gamma^{(1)}$, we have
\begin{equation*}
m^{(1)}_+(x,t,\lambda) = m^{(1)}_-(x, t, \lambda) v^{(1)}(x, t, \lambda),
\end{equation*}
where
\begin{align}\label{v1def}
 v^{(1)}(x, t, \lambda) = \begin{cases}
\begin{pmatrix}
 1 & 0 \\
 r_{a} e^{t\Theta}  & 1
\end{pmatrix}, & \lambda \in \Upsilon_1\cup\Upsilon_{2n+1}^*,
	\\
\begin{pmatrix}
 1 & - r_{a}^* e^{-t\Theta}  \\
 0  & 1
\end{pmatrix}, & \lambda \in \Upsilon_{2n+2}^*\cup\Upsilon_{4n+2},
	\\
\begin{pmatrix}
 1 - |r_{r}|^2 & - r_{r}^* e^{-t\Theta}  \\
 r_{r} e^{t\Theta}  & 1
\end{pmatrix}, & \lambda \in \R.
\end{cases}
\end{align}
\end{enumerate}
\end{rhp}

\subsection{Local parametrix near the origin}\label{sec:m0I}
As $t \to \infty$, it is easily seen from the signature of $\re \Theta$ (c.f. Figure \ref{criticalpointsfig}) and Item (c) of Lemma \ref{decompositionlemmageq} that $v^{(1)}$ in \eqref{v1def} tends to the identity matrix. The convergence, however, is not uniform for $\lambda \in \Gamma^{(1)} \setminus \R$ and close to the origin, which means we have to construct a local parametrix in a small neighborhood of the origin $D(0,\epsilon)$ with $\epsilon>0$ being small and fixed.

To formulate the local parametrix, we make a local change of variable for $\lambda$ near origin and introduce two new variables $y$ and $z$ by
\begin{align}\label{yzdefIV}
y := ((2n+1)t)^{-\frac{1}{2n+1}}x, \qquad z := ((2n+1)t)^{{\frac{1}{2n+1}}}\lambda,
\end{align}
so that
\begin{align*}
t\Theta(\xi, \lambda) = 2i\left(y z + \frac{(2z)^{2n+1}}{4n+2}\right).
\end{align*}
Since $0\leq {x}\leq M t^{1/(2n+1)}$, we have that $y$ is bounded and the map $\lambda \mapsto z$ maps $D(0,\epsilon)$ onto the open disk $D(0,((2n+1)t)^{{\frac{1}{2n+1}}}\epsilon)$ in the complex $z$-plane. The idea now is, in view of Item (b) of Lemma \ref{decompositionlemmageq}, to replace the analytic part $r_a$ of $r$ in \eqref{v1def} by its $N$-th order polynomial approximation
\begin{align}\label{pNdefIV}
  p_N(t,z): = \sum_{j=0}^N \frac{r^{(j)}(0)}{j!} \lambda^j = \sum_{j=0}^N \frac{r^{(j)}(0)}{j!(2n+1)^{\frac{j}{2n+1}}} \frac{z^j}{t^{\frac{j}{2n+1}}}.
\end{align}
and approximate $v^{(1)}$ by
\begin{align}\label{def:v0}
v_0(x, t, \lambda):=\left\{
                     \begin{array}{ll}
                        \begin{pmatrix}
 1	& 0 \\
p_N(t, z)e^{t\Theta}  & 1
\end{pmatrix}, & \hbox{$\lambda \in D(0,\epsilon)\cap \{\Upsilon_1\cup\Upsilon_{2n+1}^*\},$} \\
                       \begin{pmatrix} 1 & -p_N^*(t, z)e^{-t\Theta}	\\
0	& 1
\end{pmatrix}, & \hbox{$\lambda \in D(0,\epsilon)\cap \{\Upsilon_{2n+2}^*\cup\Upsilon_{4n+2}\}$.}
                     \end{array}
                   \right.
\end{align}
Indeed, for $T>0$ large enough, we define
$${\bf Case~~I^{T}} \doteq {\bf Case~~I} \cap \{t \geq T \}, $$
and the following estimate holds.
\begin{lemma}\label{lem:estv0}
For $(x,t) \in {\bf Case~~I^{T}}$, we have
\begin{align}\label{v2v0estimateIV}
 \|v^{(1)} - v_0\|_{L^p(D(0,\epsilon)\cap\Gamma^{(1)} \setminus \R)} \leq Ct^{-\frac{N+1}{2n+1}},
\end{align}	
for each $1\leq p\leq \infty$ and some $C>0$, where $v^{(1)}$ and  $v_0$ are defined in \eqref{v1def} and \eqref{def:v0}, respectively.
\end{lemma}
\begin{proof}
It is readily seen from \eqref{v1def} and \eqref{def:v0} that
\begin{align}\label{v2minusv0IV}
& v^{(1)} - v_0 = \begin{cases}
   \begin{pmatrix}
 0 & 0  \\
 (r_a(x, t,\lambda) - p_N(t,z)) e^{t\Theta}   &  0
  \end{pmatrix}, & \lambda \in D(0,\epsilon)\cap \{\Upsilon_1\cup\Upsilon_{2n+1}^*\},
  	\\
 \begin{pmatrix}
 0 & -(r_a^*(x, t,\lambda) - p_N^*(t, z)) e^{-t\Theta} \\
0 & 0
  \end{pmatrix}, & \lambda \in D(0,\epsilon)\cap \{\Upsilon_{2n+2}^*\cup\Upsilon_{4n+2}\}.
\end{cases}
\end{align}
We will only prove \eqref{v2v0estimateIV} for $\lambda \in D(0,\epsilon)\cap \Upsilon_1$, since the proofs for the other contours follow from similar arguments.
Let $\lambda \in D(0,\epsilon)\cap \Upsilon_1$ and $ (x,t) \in {\bf Case~~I^{T}}$, for $\xi \geq 0$, we have from \eqref{def:Thetaxi} that
\begin{multline}\label{PhionY1IV}
\re \Theta(\xi, \lambda)=\re \Theta(\xi, |\lambda| e^{\frac{\pi i}{4n+2}})
\\
= -2^{2n+1}  |\lambda|^{2n+1} - 2\xi  |\lambda| \sin\left(\frac{\pi}{4n+2}\right)
\leq -2^{2n+1}|\lambda|^{2n+1}.
\end{multline}
In particular,
\begin{align*}
e^{-\frac{3t}{4} |\re \Theta|}  \leq Ce^{-\frac{3}{4}t 2^{2n+1}|\lambda|^{2n+1}} \leq C e^{-{2^{2n-1}}|z|^{2n+1}}, \qquad \lambda \in D(0,\epsilon)\cap \Upsilon_1,
\end{align*}
for some $C>0$.
Thus, by (\ref{v2minusv0IV}), (\ref{raat1IV}) and (\ref{pNdefIV}), it follows that
\begin{align*}\nonumber
|v^{(1)} - v_0|
& \leq C|r_a(t,\lambda) - p_N(t,z)| e^{t \re \Theta}
\leq C|z t^{-\frac{1}{2n+1}}|^{N+1} e^{-\frac{3t}{4} |\re \Theta|}
	\\
& \leq C|zt^{-\frac{1}{2n+1}}|^{N+1} e^{-{2^{2n-1}}|z|^{2n+1}}, \qquad \lambda \in D(0,\epsilon) \cap \Upsilon_1.
\end{align*}
Consequently, writing $\varrho = |z|$,
\begin{align*}
& \|v^{(1)} - v_0\|_{L^\infty(D(0,\epsilon)\cap \Upsilon_1)}
 \leq C\sup_{0 \leq \varrho < \infty} (\varrho t^{-\frac{1}{2n+1}})^{N+1} e^{-{2^{2n-1}} \rho^{2n+1}}
\leq C t^{-\frac{N+1}{2n+1}}
\end{align*}
and
\begin{align*}
& \|v^{(1)} - v_0\|_{L^1(D(0,\epsilon)\cap \Upsilon_1)}
 \leq C\int_0^{\infty} (\varrho t^{-\frac{1}{2n+1}})^{N+1} e^{-{2^{2n-1}} \varrho^{2n+1}} \frac{d\varrho}{t^{\frac{1}{2n+1}}}
\leq C t^{-\frac{N+2}{2n+1}},
\end{align*}
which gives us \eqref{v2v0estimateIV}.
\end{proof}

In the virtue of Lemma \ref{lem:estv0}, we are then led to consider the following local parametrix.
\begin{rhp}\label{rhp:m0}
\hfill
\begin{enumerate}[label=\emph{(\alph*)}, ref=(\alph*)]
\item $m_0(x, t, \lambda)$ is analytic for $D(0,\epsilon)\cap\Gamma^{(1)} \setminus \R$.

\item For a.e. $\lambda \in D(0,\epsilon)\cap\Gamma^{(1)} \setminus \R$, we have
\begin{equation*}
m_{0,+}(x,t,\lambda) = m_{0,-}(x, t, \lambda) v_0(x, t, \lambda),
\end{equation*}
where $ v_0(x, t, \lambda)$ is defined in \eqref{def:v0}.

\item For $\lambda \in \partial D(0,\epsilon)$, we have $m_{0}(x,t,\lambda) \to I$ as $t \to \infty$.
\end{enumerate}
\end{rhp}

We can solve the above RH problem by using the solution of model RH problem introduced in Section \ref{sec:A2}. Let $m^{\Upsilon^*}(y,t,z)$ be the solution of RH problem \ref{RHmYIV} with the polynomial \eqref{psumIV} given by \eqref{pNdefIV}, i.e., the parameters $s$ and $p_j$, $j=1,\ldots,N$, therein are chosen to be
\begin{align}\label{sp1p2}
s = r(0) \in i \R, \qquad p_j = \frac{r^{(j)}(0)}{j!(2n+1)^{\frac{j}{2n+1}}}.
\end{align}
Due to the symmetry of reflection coefficient in \eqref{rsymm}, the polynomial $p_N(t,z)$ satisfies the symmetry relation
$$p_N(t,z) = -\overline{p_N(t,-\bar{z})},$$
which implies that $p_1\in \R$ and $p_2\in i\R$. We then define
\begin{align}\label{m0defIV}
m_0(x, t,\lambda) = m^{\Upsilon^*}(y, t, z), \qquad \lambda \in D(0,\epsilon),
\end{align}
where $y$ and $z$ are defined in \eqref{yzdefIV}. By Lemma \ref{YlemmaIV}, it follows that we can choose $T \geq C(n)$ such that $m_0$ in \eqref{m0defIV} is well-defined whenever $(x,t) \in {\bf Case~~I^{T }}$ and indeed solves the RH problem \ref{rhp:m0}. More properties of $m_0$ are collected in the following lemma for later use.

\begin{lemma}\label{m0lemmaIV}
For each $(x,t) \in {\bf Case~~I^{T }}$, we have
\begin{align}\label{m0boundIV}
|m_0(x,t,\lambda) | \leq C, \qquad  \lambda \in D(0,\epsilon)\cap\Gamma^{(1)} \setminus \R,
\end{align}
for some $C>0$. Moreover,
\begin{align}\label{m0invexpansionIV}
 m_0(x,t,\lambda)^{-1}  = I +  \sum_{j=1}^N \sum_{l=0}^N \frac{\Phi^{(0)}_{jl}(y)}{\lambda^j t^{\frac{j+l}{2n+1}}} + \Boh (t^{-\frac{N+1}{2n+1}})
\end{align}
uniformly for $(x,t) \in {\bf Case~~I^{T }}$ and $\lambda \in \partial D(0,\epsilon)$, where the coefficients $\Phi^{(0)}_{jl}(y)$ are smooth functions of $y \in [0,\infty)$. In particular,
\begin{equation}\label{m0LinftyestimateIV}
||m_0(x,t,\cdot)^{-1} - I||_{L^{\infty}(\partial D(0,\epsilon))}=\Boh (t^{-\frac{1}{2n+1}}),
\end{equation}
and
\begin{align}\label{m0circleIV}
& \frac{1}{2\pi i}\int_{\partial D(0,\epsilon)}(m_0(x,t,\lambda)^{-1} - I) d\lambda
= \sum_{l=1}^N \frac{g_l(y)}{t^{\frac{l}{2n+1}}} + \Boh (t^{-\frac{N+1}{2n+1}}),
\end{align}	
uniformly for $(x,t) \in {\bf Case~~I^{T }}$, where $g_{l+1}(y) = \Phi^{(0)}_{1l}(y)$ for each $l \geq 0$ and
\begin{align}\label{g1explicitIV}
(g_1)_{12}(y) = (g_1)_{21}(y)= -
\frac{(2n+1)^{-\frac{1}{2n+1}}}{2} q_{AS,n}(y;ir(0))
\end{align}
with $q_{\AS,n}$ being the generalized Ablowitz-Segur solution of the Painlev\'{e} II hierarchy \eqref{def:PIIhierar}.
If $r(0) = 0$, we have
\begin{align}\label{eq:g1}
g_1(y) &= 0,
	\\
g_2(y) &=  -\frac{r'(0)}{4\times (2n+1)^{\frac{2}{2n+1}}} \Ai'_{2n+1}(y)\sigma_1,
	\\ \label{g1g2g3explicit}
g_3(y) &=  \frac{i r'(0)^2}{8 \times (2n+1)^{\frac{3}{2n+1}}} \left(\int_y^{\infty} (\Ai'_{2n+1}(y'))^2dy'\right) \sigma_3
 +  \frac{ir''(0)}{16 \times  (2n+1)^{\frac{3}{2n+1}}} \Ai^{''}_{2n+1}(y) \sigma_1.
\end{align}
\end{lemma}
\begin{proof}
The proof essentially follows from Lemma \ref{YlemmaIV}. By \eqref{mYboundedIV}, we obtain the bound (\ref{m0boundIV}).
To show \eqref{m0invexpansionIV}, we note that if $|\lambda| = \epsilon$, the variable $z=((2n+1)t)^{{\frac{1}{2n+1}}}\lambda$ satisfies $|z| = ((2n+1)t)^{{\frac{1}{2n+1}}}\epsilon$ and $z\to \infty$ if $t\to \infty$.
Thus equation (\ref{mYasymptoticsIV}) yields
\begin{align*}
  m_0(x,t,\lambda) = I + \sum_{j=1}^N \sum_{l=0}^N \frac{\Phi_{jl}^{\Upsilon^*}(y)}{(((2n+1)t)^{{\frac{1}{2n+1}}}\lambda )^j t^{\frac{l}{2n+1}}}
 + \Boh \big(t^{-\frac{N+1}{2n+1}}\big),
\end{align*}
uniformly for $(x,t) \in {\bf Case~~I^{T }}$ and $\lambda \in \partial D(0,\epsilon)$. It is then straightforward to check that the expansion (\ref{m0invexpansionIV}) exists where the coefficients $\Phi^{(0)}_{jl}(y)$ can be expressed in terms of $\Phi_{jl}^{\Upsilon^*}(y)$. In particular, we have $\Phi^{(0)}_{10}(y) = - (2n+1)^{-\frac{1}{2n+1}}\Phi_{10}^{\Upsilon^*}(y)$. By (\ref{m0invexpansionIV}), the estimate (\ref{m0LinftyestimateIV}) holds and we obtain from the Cauchy's formula that
\begin{align*}
\frac{1}{2\pi i}\int_{\partial D(0,\epsilon)}(m_0^{-1} - I) d\lambda
= \sum_{l=0}^{N} \frac{\Phi^{(0)}_{1l}(y)}{t^{\frac{l+1}{2n+1}}}
+ \Boh (t^{-\frac{N+2}{2n+1}}),
\end{align*}
which is \eqref{m0circleIV}.

Finally, note that
$$g_1(y) = \Phi_{10}^{(0)}(y) = - (2n+1)^{-\frac{1}{2n+1}} \Phi_{10}^{\Upsilon^*}(y),$$
we then obtain from \eqref{m10Yexplicit} and \eqref{eq:qsym} that
\begin{align*}
(g_1)_{12}(y) = (g_1)_{21}(y)=\frac{(2n+1)^{-\frac{1}{2n+1}}}{2} q_{AS,n}(y;-ir(0))=-\frac{(2n+1)^{-\frac{1}{2n+1}}}{2} q_{AS,n}(y;ir(0)),
\end{align*}
which is \eqref{g1explicitIV}. If $r(0) = 0$, a combination of the fact that
$$g_{l+1}(y) = \Phi^{(0)}_{1l}(y)=-(2n+1)^{-\frac{1}{2n+1}} \Phi_{1l}^{\Upsilon^*}(y), \quad l=0,1,2,$$
\eqref{sp1p2} and \eqref{eq:phi10}--\eqref{eq:phi12} gives us \eqref{eq:g1}--\eqref{g1g2g3explicit}.
\end{proof}

\subsection{Final transformation}
The final transformation is defined by
\begin{align}\label{mhatdef}
\hat{m}(x,t,\lambda) = \begin{cases}
m^{(1)}(x, t, \lambda)m_0(x,t,\lambda)^{-1}, & \lambda \in D(0,\epsilon),\\
m^{(1)}(x, t, \lambda), & \lambda \in \C \setminus D(0,\epsilon),
\end{cases}
\end{align}
and it is readily seen that $\hat{m}$ satisfies the following RH problem.
\begin{rhp}\label{RHmhat}
\hfill
\begin{enumerate}[label=\emph{(\alph*)}, ref=(\alph*)]
\item $\hat{m}(x, t, \cdot) \in I + \dot{E}^2(\C \setminus \hat{\Gamma})$, where $\hat{\Gamma} := \Gamma^{(1)} \cup \partial D(0,\epsilon)$ with $\Gamma^{(1)}$ defined in \eqref{def:gamma1}; see Figure \ref{Gammahatgeq.pdf} for an illustration.
\item For a.e. $\lambda \in \hat{\Gamma}$, we have
\begin{equation}\label{eq:hatmjump}
\hat{m}_+(x,t,\lambda) = \hat{m}_-(x, t, \lambda) \hat{v}(x, t, \lambda),
\end{equation}
where
\begin{align*}
\hat{v}
=
\begin{cases}
m_{0,-} v^{(1)} m_{0,+}^{-1}, & \lambda \in \hat{\Gamma} \cap D(0,\epsilon), \\
m_0^{-1}, & \lambda \in \partial D(0,\epsilon), \\
v^{(1)},  & \lambda \in \hat{\Gamma} \setminus \overline{D(0,\epsilon)},
\end{cases}
\end{align*}
and where $v^{(1)}$ is defined in \eqref{v1def}.
\end{enumerate}
\end{rhp}

\begin{figure}
\begin{center}
\begin{overpic}[width=.5\textwidth]{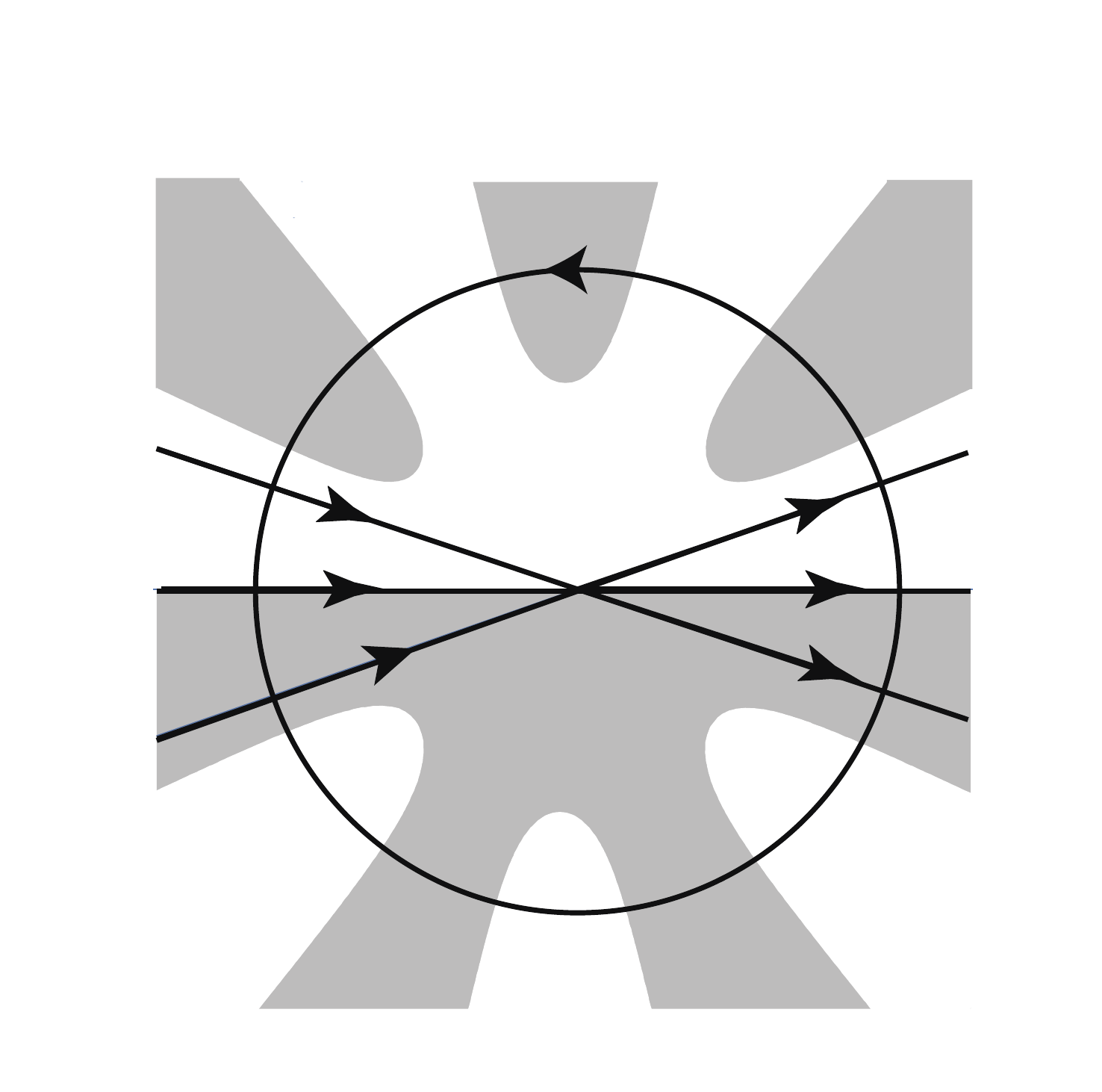}
      \put(99,48){\small $\hat{\Gamma}$}
      \put(83,46.5){\small $\epsilon$}
      \put(15.5,46.5){\small $-\epsilon$}
    \end{overpic}
\caption{\label{Gammahatgeq.pdf}
        The contour $\hat{\Gamma}$ for the RH problem \ref{RHmhat} for $\hat{m}$.
}
     \end{center}
\end{figure}

The jump matrix $\hat v$ in \eqref{eq:hatmjump} tends to the identity matrix as $t\to \infty$. More precisely,
let
\begin{equation}\label{def:hatw}
\hat{w} := \hat{v} - I,
\end{equation} we have the following estimates.
\begin{lemma}\label{wlemmaIV}
For each $1 \leq p \leq \infty$, the following estimates hold uniformly for $(x,t) \in {\bf Case~~I^{T }}$:
\begin{subequations}\label{westimateIV}
\begin{align}\label{westimateIVa}
& \|\hat{w}\|_{L^p(\partial D(0,\epsilon))} \leq Ct^{-\frac{1}{2n+1}},
	\\\label{westimateIVb}
& \|\hat{w}\|_{L^p(D(0,\epsilon)\cap\Gamma^{(1)} \setminus \R)} \leq Ct^{-\frac{N+1}{2n+1}},
	\\ \label{westimateIVc}
& \|\hat{w}\|_{L^p(\R)} \leq C t^{-N},
	\\\label{westimateIVd}
& \|\hat{w}\|_{L^p(\Gamma^{(1)} \setminus \{\R\cup \overline{D(0,\epsilon)}\})} \leq Ce^{-ct},
\end{align}
for some positive constants $c$ and $C$.
\end{subequations}
\end{lemma}
\begin{proof}
The estimate (\ref{westimateIVa}) follows from (\ref{m0LinftyestimateIV}).
For $\lambda \in D(0,\epsilon)\cap\Gamma^{(1)} \setminus \R$, we have
$$\hat{w} = m_{0,-} (v^{(1)} - v_0) m_{0,+}^{-1}.$$
This, together with \eqref{v2v0estimateIV} and \eqref{m0boundIV}, gives us \eqref{westimateIVb}.
According to Item (c) of Lemma \ref{decompositionlemmageq} and the boundedness \eqref{m0boundIV} of $m_0$, one has \eqref{westimateIVc}. Finally, \eqref{westimateIVd} follows because $e^{-t|\re \Theta|} \leq Ce^{-ct}$ uniformly on $\Gamma^{(1)} \setminus \{\R\cup \overline{D(0,\epsilon)}\}$.
\end{proof}

\section{Asymptotic analysis of the RH problem for $m$: {\bf Case II}}
\label{proof:th2th3-2}
In this section, we analyse the RH problem \ref{rhp-org} for $m$ as $t\to \infty$ in {\bf Case~ II}, that is,
\begin{align*}
\{(x,t) \mid 0\leq{x}\leq M t^{\frac{1}{2n+1}}\}, &\qquad n \text{ is even},
\end{align*}
where $M > 0$. In this case, the critical points of phase function $\Theta$ in \eqref{def:Thetaxi} are
$$\lambda^{(n,j)}_0 = \sqrt[2n]{\frac{\xi}{2^{2n}(2n+1)}} e^{j\frac{\pi i}{n}},\qquad j=1,2,\cdots,2n, $$
which approach $0$ at least as fast as $t^{-1/(2n+1)}$ as $t \to \infty$. Thus, it follows that $|\lambda^{(n,j)}_0| \leq C t^{-1/2n+1}$ for some $C>0$. The difference between the present case and {\bf Case I} lies in the fact that there are two critical points lying on the real line, which we denote  by
\begin{equation}\label{def:lambda0}
\pm \lambda_0:=\pm \sqrt[2n]{\frac{\xi}{2^{2n}(2n+1) }}.
\end{equation} For $n=2$ and $\xi=1$, an illustration of the four critical points and signature of $\re \Theta$ are shown in Figure \ref{Gamma1leq.pdf}.

During our analysis performed in what follows, the same notations $m^{(1)},V,\ldots$, will be used to emphasize the analogies with the previous section, and we believe this will not cause any confusion.

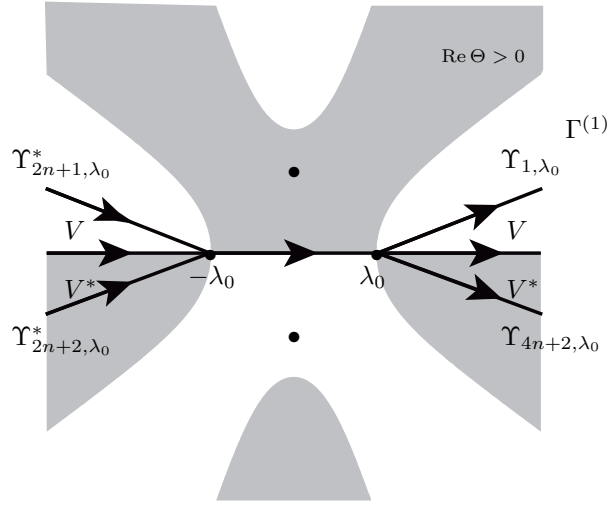
\begin{figure}
\begin{center}
\begin{overpic}[width=.5\textwidth]{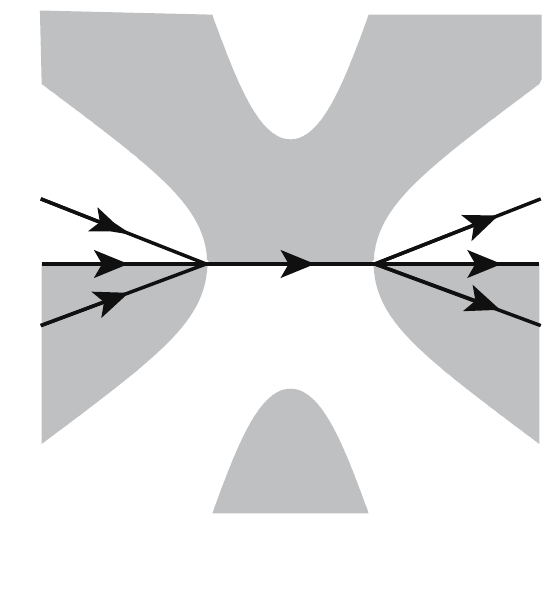}
      \put(95,75){\small $\Gamma^{(1)}$}
      \put(84,70){\small $\Upsilon_{1,\lambda_0}$}
      \put(1,70){\small $\Upsilon_{2n+1,\lambda_0}^*$}
      \put(84,40){\small $\Upsilon_{4n+2,\lambda_0}$}
      \put(1,40){\small $\Upsilon_{2n+2,\lambda_0}^*$}
      \put(10,58){\small $V$}
      \put(10,48){\small $V^*$}
      \put(85,58){\small $V$}
      \put(85,48){\small $V^*$}
      \put(74,88){\tiny $\re \Theta > 0$}
        \put(49,41.3){\circle*{2}}
       \put(49,69.3){\circle*{2}}
       \put(34.8,55.3){\circle*{2}}
      \put(31,50.3){\small $-\lambda_0$}
       \put(63,55.3){\circle*{2}}
      \put(60,50.3){\small $\lambda_0$}
    \end{overpic}
\caption{\label{Gamma1leq.pdf}
       The contour $\Gamma^{(1)}$ for the RH problem \ref{rhp:m1II} for $m^{(1)}$ and the sets $V$ and $V^*$ in ${\bf Case~~II}$. The region where $\re \Theta > 0$ is shaded and the dark points stand for the critical points $\lambda_{0}^{(2,j)}$, $j=1,\ldots,4$, for $\xi=1$.
     }
     \end{center}
\end{figure}

\subsection{First transformation: $m \to m^{(1)}$}
As in Section \ref{sec:firstranI}, we need to decompose the reflection coefficient $r$ into an analytic part and a small remainder to define the first transformation. In this case, we define
\begin{equation*}
\Upsilon_{i,\lambda_0}:=\lambda_0+\Upsilon_i, \quad i=1,4n+2, \qquad \Upsilon_{i,\lambda_0}^*:=-\lambda_0+\Upsilon_i^*,  \quad i=2n+1,2n+2.
\end{equation*}
These four rays, together with $(-\infty,-\lambda_0]\cup [\lambda_0,\infty)$, formulate the boundaries of open subsets $V \equiv V(\xi)$ and $V^* \equiv V^*(\xi)$, as illustrated in Figure \ref{Gamma1leq.pdf}.
Similar to \cite[Lemma 5.7]{CL-2019}, \cite[Lemma 5.2]{hl-2019} and \cite[Lemma 7.1]{hl-2018}, the decomposition of $r$ in {\bf Case II} reads as follows and we again omit the proof.

\begin{lemma}[Analytic approximation for $0\leq \xi \leq A$]\label{decompositionlemmaleq}
Let $r$ be the reflection coefficient defined in \eqref{def:r}, we have
\begin{equation}\label{eq:decomII}
r(\lambda) = r_{a}(x, t, \lambda) + r_{r}(x, t, \lambda), \qquad \lambda \in (-\infty, -\lambda_0)\cup (\lambda_0, \infty),
\end{equation}
and for a fixed $N\in \N$ the functions $r_{a}$ and $r_{r}$ have the following properties:
\begin{enumerate}[label=\emph{(\alph*)}, ref=(\alph*)]
\item For each $\xi \in [0, A]$ and $t\geq C(n)$ with positive $A$ and $C(n)$, $r_{a}(x, t,\lambda)$ is defined and continuous for $\lambda \in \bar{V}$ and analytic for $\lambda \in V$.

\item There exists a constant $C>0$ such that
\begin{align*}
& |r_{a}(x, t, \lambda)| \leq \frac{C}{1 + |\lambda|} e^{\frac{t}{4}|\re \Theta(\xi,\lambda)|}, \qquad
  \lambda \in \bar{V},
\end{align*}
and
\begin{equation}\label{raatk0}
\bigg|r_{a}(x, t, \lambda) - \sum_{j=0}^N \frac{r^{(j)}(\lambda_0)}{j!} (\lambda-\lambda_0)^j\bigg| \leq C |\lambda-\lambda_0|^{N+1} e^{\frac{t}{4}|\re \Theta(\xi,\lambda)|}, \qquad \lambda \in \bar{V},
\end{equation}
uniformly for $\xi \in [0, A]$ and $t \geq C(n)$.

\item The $L^1$ and $L^\infty$ norms of $r_{r}(x, t, \cdot)$ on $(-\infty, -\lambda_0)\cup (\lambda_0, \infty)$ are $\Boh(t^{-N})$ as $t \to \infty$ uniformly with respect to $\xi \in [0, A]$.

\item The functions $r_a$ and $r_{r}$ satisfy the following symmetries:
\begin{align*}
\begin{cases}
r_{a}(x,t, \lambda) = -r_{a}^*(x,t, -\lambda), &\lambda \in \bar{V},\\
r_{r}(x,t, \lambda) = -r_{r}^*(x,t, -\lambda), & \lambda \in (-\infty,-\lambda_0)\cup (\lambda_0,\infty).
\end{cases}
\end{align*}

\end{enumerate}
\end{lemma}

We now define $m^{(1)}$ by \eqref{m1def} with $G(x,t,\lambda)$ given by \eqref{GdefII} and emphasize that the function $r_a$ therein appears in the decomposition \eqref{eq:decomII}. By Lemma \ref{decompositionlemmaleq}, one has
$$G(x,t,\cdot) \in I + (\dot{E}^2 \cap E^\infty)(V \cup V^*),$$
and $m$ satisfies the RH problem \ref{rhp-org} if and only if $m^{(1)}$ solves the following RH problem.
\begin{rhp}\label{rhp:m1II}
\hfill
\begin{enumerate}[label=\emph{(\alph*)}, ref=(\alph*)]
\item $m^{(1)}(x, t, \cdot) \in I + \dot{E}^2(\C \setminus \Gamma^{(1)})$,
where
\begin{equation}\label{def:gamma1II}
\Gamma^{(1)} := \R \cup \Upsilon_{1,\lambda_0}\cup\Upsilon_{2n+1,\lambda_0}^*\cup\Upsilon_{2n+2,\lambda_0}^*\cup\Upsilon_{4n+2,\lambda_0};
\end{equation}
see Figure \ref{Gamma1leq.pdf} for an illustration and the orientation.

\item For a.e. $\lambda \in \Gamma^{(1)}$, we have
\begin{equation*}
m^{(1)}_+(x,t,\lambda) = m^{(1)}_-(x, t, \lambda) v^{(1)}(x, t, \lambda),
\end{equation*}
where
\begin{align}\label{v1defII}
 v^{(1)}(x, t, \lambda) = \begin{cases}
\begin{pmatrix}
 1 & 0 \\
 r_{a} e^{t\Theta}  & 1
\end{pmatrix}, & \lambda \in \Upsilon_{1,\lambda_0}\cup\Upsilon_{2n+1,\lambda_0}^*,
	\\
\begin{pmatrix}
 1 & - r_{a}^* e^{-t\Theta}  \\
 0  & 1
\end{pmatrix}, & \lambda \in \Upsilon_{2n+2,\lambda_0}^*\cup\Upsilon_{4n+2,\lambda_0},
	\\
\begin{pmatrix}
 1 - |r_{r}|^2 & - r_{r}^* e^{-t\Theta}  \\
 r_{r} e^{t\Theta}  & 1
\end{pmatrix}, & \lambda \in (-\infty,-\lambda_0)\cup(\lambda_0,\infty),
\\
\begin{pmatrix}
 1 - |r|^2 & - r^* e^{-t\Theta}  \\
 r e^{t\Theta}  & 1
\end{pmatrix}, & \lambda  \in (-\lambda_0,\lambda_0).
\end{cases}
\end{align}
\end{enumerate}
\end{rhp}

\subsection{Local parametrix near the origin}
As $t \to \infty$, it is easily seen from the signature of $\re \Theta$ (c.f. Figure \ref{Gamma1leq.pdf}) and Item (c) of Lemma \ref{decompositionlemmaleq} that $v^{(1)}$ in \eqref{v1defII} tends to the identity matrix. The convergence is not uniform for $\lambda \in \Gamma^{(1)} \setminus (-\infty, -\lambda_0]\cup[\lambda_0,\infty)$ and close to $\pm \lambda_0$. Since $\lambda_0=\Boh(t^{-1/(2n+1)})$, we need to build a local parametrix in a small neighborhood of the origin $D(0,\epsilon)$.

To formulate the local parametrix, as in {\bf Case I}, we introduce the new variables $y$ and $z$ by
\begin{align}\label{yzNewdefIVII}
y := -((2n+1)t)^{-\frac{1}{2n+1}}x, \qquad z := ((2n+1)t)^{{\frac{1}{2n+1}}}\lambda,
\end{align}
so that
\begin{align*}
t\Theta(\xi, \lambda) = 2i\left(y z + \frac{(2z)^{2n+1}}{4n+2}\right).
\end{align*}
Since $0\leq {x}\leq M t^{1/(2n+1)}$, it is easily seen that $-C \leq y \leq 0 $ for some $C>0$. Moreover, the critical point $\lambda_0$ is mapped to the point
\begin{equation}\label{def:z_0}
z_0 := ((2n+1)t)^{\frac{1}{2n+1}}\lambda_0 = \sqrt[2n]{|y|}/2
\end{equation}
on the $z$-plane. Following the strategy used in Section \ref{sec:m0I}, we approximate $r_a$ and $r$ by the $N$-th order polynomial $p_N$ given in \eqref{pNdefIV} for large $t$, which also leads to the approximation of $v^{(1)}$ in \eqref{v1defII} by
\begin{align}\label{def:v0II}
v_0(x, t, \lambda):=\left\{
                     \begin{array}{ll}
                        \begin{pmatrix}
 1	& 0 \\
p_N(t, z)e^{t\Theta}  & 1
\end{pmatrix}, & \hbox{$\lambda \in D(0,\epsilon)\cap \{\Upsilon_{1,\lambda_0}\cup\Upsilon_{2n+1,\lambda_0}^*\},$}
\\
                       \begin{pmatrix} 1 & -p_N^*(t, z)e^{-t\Theta}	\\
0	& 1
\end{pmatrix}, & \hbox{$\lambda \in D(0,\epsilon)\cap \{\Upsilon_{2n+2,\lambda_0}^*\cup\Upsilon_{4n+2,\lambda_0}\}$,}
\\
\begin{pmatrix}
 1 - |p_N(t,z)|^2 & - p_N^*(t,z) e^{-t\Theta}  \\
 p_N(t,z) e^{t\Theta}  & 1
\end{pmatrix}, & \lambda  \in (-\lambda_0,\lambda_0).
 \end{array}
 \right.
\end{align}
For $T>0$ large enough, by setting
$${\bf Case~~II^{T}} \doteq  {\bf Case~~II} \cap \{t \geq T \}, $$
the following estimate holds.
\begin{lemma}\label{lem:estv0II}
For $(x,t) \in {\bf Case~~II^{T}}$, we have
\begin{align}\label{v2v0estimateIVg}
 \|v^{(1)} - v_0\|_{L^p(D(0,\epsilon)\cap\Gamma^{(1)} \setminus \{(\infty,-\lambda_0]\cup[\lambda_0,\infty)\})} \leq Ct^{-\frac{N+1}{2n+1}},
\end{align}	
for each $1\leq p\leq \infty$ and some $C>0$, where $v^{(1)}$ and  $v_0$ are defined in \eqref{v1defII} and \eqref{def:v0II}, respectively.
\end{lemma}
\begin{proof}
It is readily seen from \eqref{v1defII} and \eqref{def:v0II} that
\begin{align}
& v^{(1)} - v_0
\nonumber
\\
& = \begin{cases}
   \begin{pmatrix}
 0 & 0  \\
 (r_a(x,t,\lambda) - p_N(t,z)) e^{t\Theta}   & 0
  \end{pmatrix}, & \lambda \in D(0,\epsilon)\cap \{\Upsilon_{1,\lambda_0}\cup\Upsilon_{2n+1,\lambda_0}^*\},
  	\\
 \begin{pmatrix}
 0 & -(r_a^*(x, t,\lambda) - p_N^*(t, z)) e^{-t\Theta} \\
0 & 0
  \end{pmatrix}, & \lambda \in D(0,\epsilon)\cap \{\Upsilon_{2n+2,\lambda_0}^*\cup\Upsilon_{4n+2,\lambda_0}\},
\\
 \begin{pmatrix} - |r|^2 + |p_N|^2
  &  -(r^*(\lambda) - p_N^*(t, z)) e^{-t\Theta} \\
(r(\lambda) - p_N(t,z)) e^{t\Theta} 	& 0
  \end{pmatrix}, & \lambda \in (-\lambda_0,\lambda_0). \label{v1minusv0IVg}
\end{cases}
\end{align}
We will prove \eqref{v2v0estimateIVg} for $\lambda\in (-\lambda_0,\lambda_0)$ and $\lambda \in D(0, \epsilon) \cap \Upsilon_{1,\lambda_0}$, since the proofs for the other contours follow from similar arguments.

By expanding $r^{(j)}(\lambda_0)$, $j=1,\ldots,N$, around the origin, it is easily seen that
\begin{align}\nonumber
\sum_{j=0}^N \frac{r^{(j)}(\lambda_0)}{j!} (\lambda-\lambda_0)^j
& = \sum_{j=0}^N\sum_{l=0}^N \frac{r^{(j+l)}(0)}{j! l!} \lambda_0^l(\lambda-\lambda_0)^j  + \Boh(\lambda_0^{N+1})
	\\\nonumber
& = \sum_{l=0}^N \sum_{\kappa=l}^{N+l} \frac{r^{(\kappa)}(0)}{(\kappa-l)! l!} \lambda_0^l(\lambda-\lambda_0)^{\kappa-l}   + \Boh(\lambda_0^{N+1})
	\\\label{sumsumdifference}
&= \sum_{\kappa=0}^{N} \frac{r^{(\kappa)}(0)}{\kappa!}\lambda^{\kappa}  + \Boh(|\lambda|^{N+1})= P_N(t,z) + \Boh(|\lambda|^{N+1}),
\end{align}
uniformly for $0 \leq \lambda_0 \leq C$ and $\lambda \in D(0,\epsilon)\cap\Gamma^{(1)} \setminus \{(\infty,-\lambda_0]\cup[\lambda_0,\infty)\})$. If $\lambda \in (-\lambda_0,\lambda_0)$, we have $\re \Theta = 0$. A combination of \eqref{v1minusv0IVg}, \eqref{sumsumdifference} and Lemma \ref{decompositionlemmaleq} shows that
\begin{align*}
 |v^{(1)} - v_0| \leq C |r(\lambda) - p_N(t,z)| \leq C |\lambda|^{N+1} \leq Ct^{-\frac{N+1}{2n+1}},
\end{align*}
which gives us \eqref{v2v0estimateIVg} for $\lambda\in(-\lambda_0,\lambda_0)$. If $\lambda = \lambda_0 + \varrho e^{\frac{\pi i}{4n+2}} \in D(0,\epsilon) \cap \Upsilon_{1,\lambda_0}$, we see from the fact  $\xi=(2n+1)(2\lambda_0)^{2n}>0$ and \eqref{def:Thetaxi} that
\begin{align*}
\re \Theta(\xi, \lambda)& = \re \Theta(\xi, \lambda_0 + \varrho e^{\frac{\pi i}{4n+2}})\\
&=\re 2i(-\xi (\lambda_0+\varrho e^{\frac{\pi i}{4n+2}})+2^{2n}(\lambda_0+\varrho e^{\frac{\pi i}{4n+2}})^{2n+1})\\
&= - 2^{2n+1}\bigg( \begin{pmatrix}2n+1\\2\end{pmatrix}\lambda_0^{2n-1} \varrho^2 \sin \frac{n}{2n+1} \pi +\cdots+\varrho^{2n+1} \bigg)\\
& \leq  -2^{2n+1}\varrho^{2n+1}=-2^{2n+1}|\lambda-\lambda_0|^{2n+1},
\end{align*}
where $\begin{pmatrix}n\\k\end{pmatrix}$ is the binomial number. Thus,
\begin{align}\label{etPhiIVg}
e^{-\frac{3t}{4}|\re \Theta|} \leq e^{-\frac{3}{4}t2^{2n+1}\varrho^{2n+1}} \leq & Ce^{-\frac{3}{4}t|\lambda|^{2n+1}} \leq C e^{-\frac{3}{4(2n+1)}|z|^{2n+1}}.
\end{align}
It then follows from \eqref{raatk0}, \eqref{v1minusv0IVg}, \eqref{sumsumdifference} and \eqref{etPhiIVg} that
\begin{align*}
 |v^{(1)} - v_0| & \leq C|r_a(x,t,\lambda) - p_N(t,z)| e^{t \re \Theta}
\nonumber
\\
 & \leq C\bigg|r_{a}(x, t, \lambda) - \sum_{j=0}^N \frac{r^{(j)}(\lambda_0)}{j!} (\lambda-\lambda_0)^j\bigg| e^{t \re \Theta}
 	\\
& \quad + C\bigg|\sum_{j=0}^N \frac{r^{(j)}(\lambda_0)}{j!} (\lambda-\lambda_0)^j - \sum_{j=0}^N \frac{r^{(j)}(0)}{j!} \lambda^j\bigg| e^{t \re \Theta}
	\\
 & \leq  C |\lambda-\lambda_0|^{N+1} e^{-\frac{3}{4}t |\re \Theta|} +  C|\lambda|^{N+1} e^{- t |\re \Theta|}
 \nonumber
\\
& \leq C\big|zt^{-\frac{1}{2n+1}}\big|^{N+1} e^{-\frac{3}{4(2n+1)|z|^{2n+1}}}.
\end{align*}
As in the proof of Lemma \ref{m0lemmaIV}, this implies
$$\|v^{(1)} - v_0\|_{(L^1 \cap L^\infty)(D(0,\epsilon)\cap \Upsilon_{1,\lambda_0})} \leq Ct^{-\frac{N+1}{2n+1}},$$
as required.
\end{proof}

In the virtue of Lemma \ref{lem:estv0II}, we are then led to consider the following local parametrix.
\begin{rhp}\label{rhp:m0II}
\hfill
\begin{enumerate}[label=\emph{(\alph*)}, ref=(\alph*)]
\item $m_0(x, t, \lambda)$ is analytic for $D(0,\epsilon)\cap\Gamma^{(1)} \setminus \{(-\infty,-\lambda_0]\cup [\lambda_0,\infty\}$, where $\Gamma^{(1)} $ is defined in \eqref{def:gamma1II}.

\item For a.e. $\lambda \in D(0,\epsilon)\cap\Gamma^{(1)} \setminus \{(-\infty,-\lambda_0]\cup [\lambda_0,\infty\}$, we have
\begin{equation}
m_{0,+}(x,t,\lambda) = m_{0,-}(x, t, \lambda) v_0(x, t, \lambda),
\end{equation}
where $ v_0(x, t, \lambda)$ is defined in \eqref{def:v0II}.

\item For $\lambda \in \partial D(0,\epsilon)$, we have $m_{0}(x,t,\lambda) \to I$ as $t \to \infty$.
\end{enumerate}
\end{rhp}

We can solve the above RH problem with the aid of the solution of model RH problem introduced in Section \ref{IVgeqapp}. Let $m^Z(y,t,z_0,z)$ be solution of RH problem \ref{modelRHPII} with the polynomial \eqref{psumIV} in the jump condition given by \eqref{pNdefIV}, we define
\begin{align}\label{m0defIVg}
m_0(x, t, \lambda) := m^Z(y, t, z_0, z), \qquad \lambda \in D(0,\epsilon),
\end{align}
where $y$, $z$ and $z_0$ are defined in \eqref{yzNewdefIVII} and \eqref{def:z_0}, respectively. If $(x,t) \in {\bf Case~~II^T}$, then $(y,t,z_0) \in \mathcal{P}_T$, where $\mathcal{P}_T$ defined in (\ref{parametersetIVg}). By Lemma \ref{ZlemmaIVg}, we can choose suitable $T$ such that $m_0$ in \eqref{m0defIVg} is well-defined and check that it indeed solves the RH problem \ref{rhp:m0II}. The function $m_0$ in this case has similar properties as that in ${\bf Case~~I}$, which are collected in the following lemma and follow directly from Lemma \ref{ZlemmaIVg}.

\begin{lemma}\label{m0lemmaIVg}
For each $(x,t) \in {\bf Case~~II^T}$, the function $m_0(x,t,\lambda)$ defined in (\ref{m0defIVg}) is uniformly bounded, i.e.,
\begin{align}\label{m0boundIVg}
|m_0(x,t,\lambda)| \leq C, \qquad \lambda \in D(0,\epsilon) \cap \Gamma^{(1)} \setminus \{(-\infty,-\lambda_0]\cup [\lambda_0,\infty\}.
\end{align}
Moreover,
\begin{align}\label{m0circleIV-2}
 m_0(x,t,\lambda)^{-1}  = I +  \sum_{j=1}^N \sum_{l=0}^N \frac{\Phi^{(0)}_{jl}(y)}{\lambda^j t^{\frac{j+l}{2n+1}}} +  O(t^{-\frac{N+1}{2n+1}})
\end{align}
uniformly for $(x,t) \in {\bf Case~~II^T}$ and $\lambda \in \partial D(0,\epsilon)$. The function $m_0(x,t,\lambda)$ satisfies \eqref{m0LinftyestimateIV} and \eqref{m0circleIV} uniformly for $(x,t) \in {\bf Case~~II^T}$, where the coefficients $\{g_j(y)\}_1^N$ are smooth functions of $y\in \R$ with
\begin{align}\label{g1explicitII}
(g_1)_{12}(y) = (g_1)_{21}(y)= -
\frac{(2n+1)^{-\frac{1}{2n+1}}}{2} q_{AS,n}(-y;ir(0)).
\end{align}
If $r(0)=0$, we still have \eqref{eq:g1}--\eqref{g1g2g3explicit} for the first few terms in \eqref{m0circleIV-2}.
\end{lemma}

\subsection{Final transformation}
As in ${\bf Case~~I}$, we define $\hat m$ by \eqref{mhatdef} in the final transformation, but it is understood that the functions $m^{(1)}$ and $m_0$ therein solve RH problems \ref{rhp:m1II} and \ref{rhp:m0II}, respectively. Thus, $\hat{m}$ satisfies RH problem \ref{RHmhat} with the contour $\hat \Gamma:=\Gamma^{(1)}\cup \partial D(0,\epsilon)$ illustrated in Figure \ref{Gammahatleq.pdf}.
\begin{figure}
\begin{center}
\begin{overpic}[width=.5\textwidth]{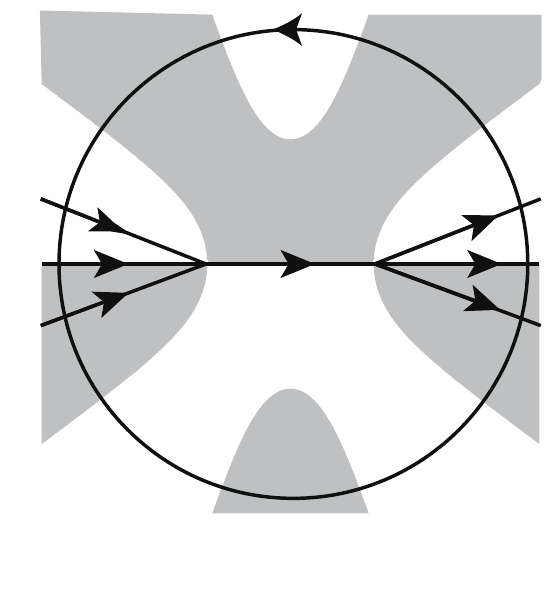}
      \put(92,60){\small $\hat{\Gamma}$}
            \put(90,50){\small $\epsilon$}
      \put(3,50){\small $-\epsilon$}
    \end{overpic}
\caption{\label{Gammahatleq.pdf}
        The contour  $\hat{\Gamma}$ in ${\bf Case~~II}$.
}
     \end{center}
\end{figure}

The estimates of $\hat w$ in \eqref{def:hatw} for {\bf Case~~II} are given in the following lemma, which can be proved in a manner similar to that of Lemma \ref{wlemmaIV}.
\begin{lemma}\label{wlemmaIVg}
For each $1 \leq p \leq \infty$, the following estimates hold uniformly for $(x,t) \in {\bf Case~~II^T}$:
\begin{subequations}\label{westimateIVg}
\begin{align}\label{westimateIVga}
& \|\hat{w}\|_{L^p(\partial D(0,\epsilon))} \leq Ct^{-\frac{1}{2n+1}},
	\\\label{westimateIVgb}
& \|\hat{w}\|_{L^p(D(0,\epsilon)\cap\Gamma^{(1)} \setminus \{(-\infty,-\lambda_0]\cup [\lambda_0,\infty\})} \leq Ct^{-\frac{N+1}{2n+1}},
	\\ \label{westimateIVgc}
& \|\hat{w}\|_{L^p(\R \setminus [-\lambda_0,\lambda_0])} \leq C t^{-N},
	\\\label{westimateIVgd}
& \|\hat{w}\|_{L^p(\Gamma^{(1)} \setminus \{\R\cup \overline{D(0,\epsilon)}\})} \leq Ce^{-ct},
\end{align}
for some positive constants $c$ and $C$.
\end{subequations}
\end{lemma}
\begin{remark}\label{rk:caseIII}
We have completed asymptotic analysis of the RH problem \ref{rhp-org} for $m$ in \text{\bf Case I} and \text{\bf Case II}.
For \text{\bf Case III}, critical points of the phase function $\Theta$ in \eqref{def:Thetaxi} are not on the real line; see the left picture in Figure \ref{fig:caseIIIV}. Therefore, the asymptotic analysis is analogous to that in \text{\bf Case I} and local analysis near the origin is related to the model RH problem \ref{RHmYIV}. For \text{\bf Case IV}, the phase function $\Theta$ has two critical points lying on the real line; see the right picture in Figure \ref{fig:caseIIIV}, which implies that the analysis can be carried out in a way similar to that in \text{\bf Case II} and the model RH problem \ref{modelRHPII} will play a crucial role in the local analysis.
\end{remark}
\begin{figure}
\centering
\subfigure[$n=2$]{
    \begin{minipage}[b]{0.3\linewidth} 
    \begin{overpic}[width=5cm]{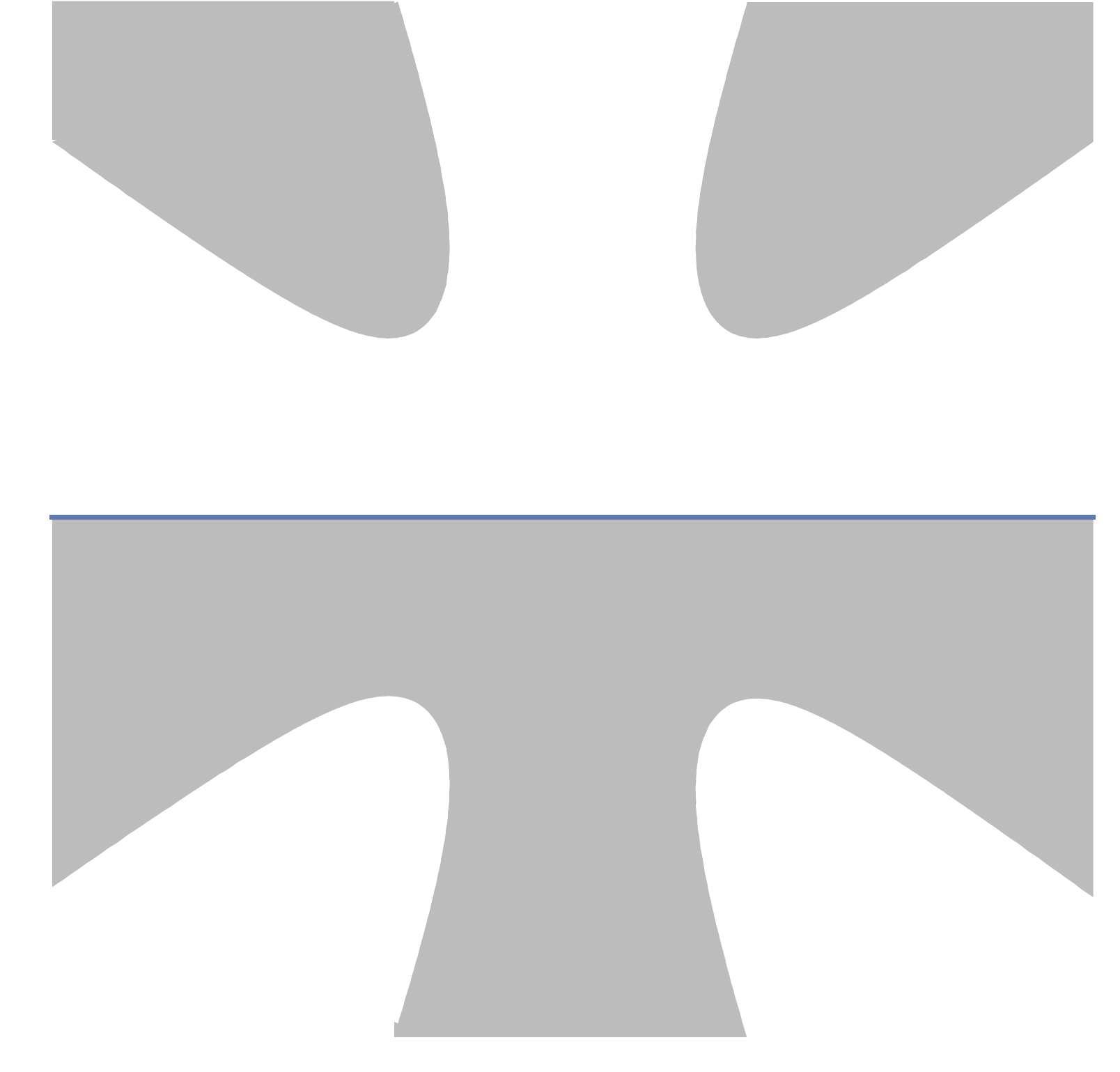}\vspace{1pt}
    \put(39,65){\circle*{2}}
       \put(39,35){\circle*{2}}
        \put(61,35){\circle*{2}}
        \put(61,65){\circle*{2}}
       \put(41,88){\tiny $\re \Theta < 0$}
      \put(41,8){\tiny $\re \Theta >0$}
    \end{overpic}
\end{minipage}}
\quad 
\qquad
\qquad
\subfigure[$n=3$]{
    \begin{minipage}[b]{0.35\linewidth}
    \begin{overpic}[width=5cm]{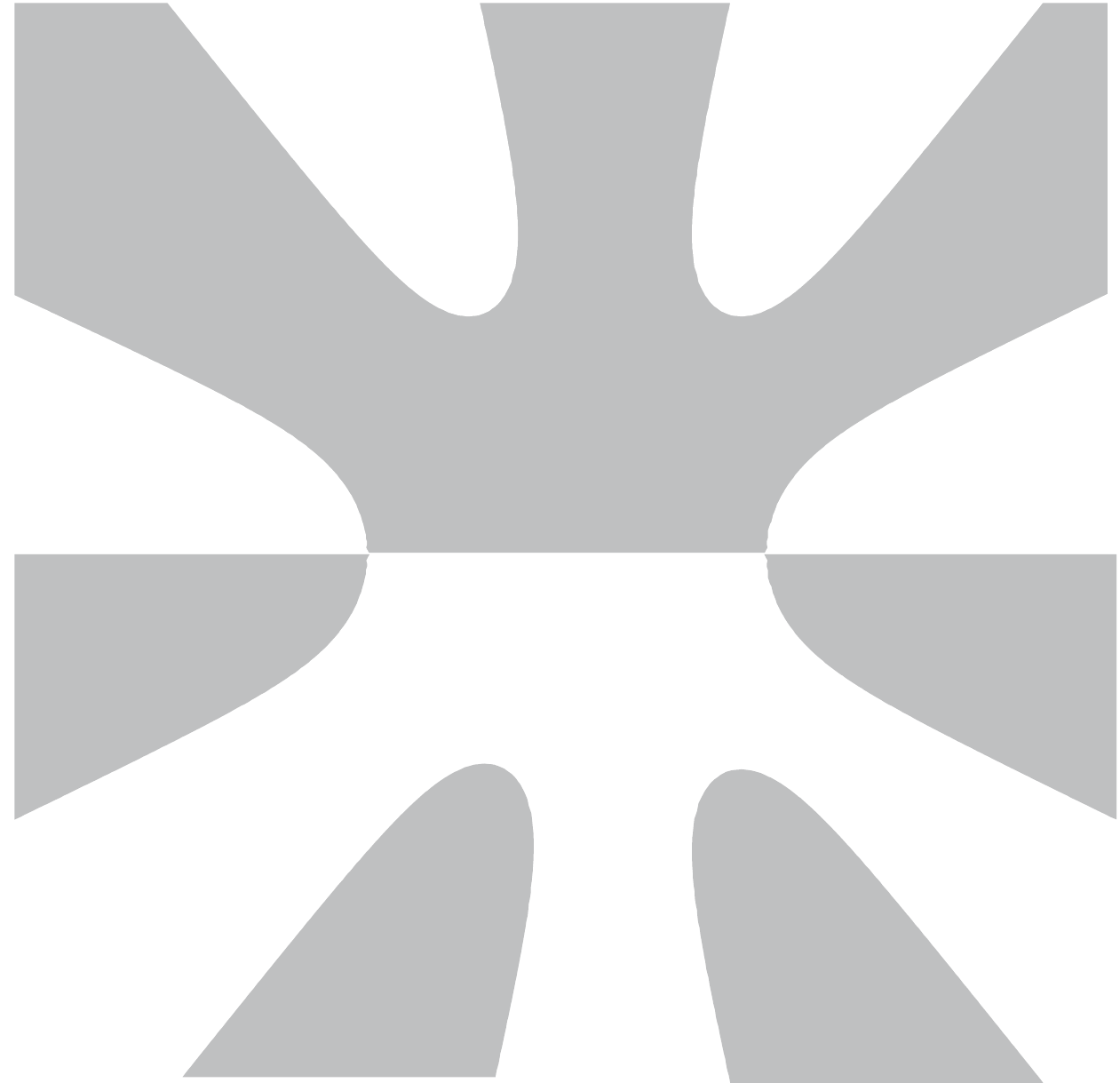}\vspace{1pt}
      \put(44,31){\circle*{2}}
      \put(44,66){\circle*{2}}
      \put(33,47){\circle*{2}}
       \put(68,47){\circle*{2}}
        \put(61,31){\circle*{2}}
        \put(61,66){\circle*{2}}

      \put(42,55){\tiny $\re \Theta > 0$}
      \put(42,38){\tiny $\re \Theta < 0$}
    \end{overpic}
    \end{minipage}
}
\caption{\label{fig:caseIIIV}
       The critical points and signature of $\re \Theta$ in the complex $\lambda$-plane with $\xi=-1$ for \text{\bf Case III}  and \text{\bf Case IV}. The shaded areas indicate the regions where $\re \Theta > 0$ for $n=2$ (left) and $n=3$ (right).}
\end{figure}
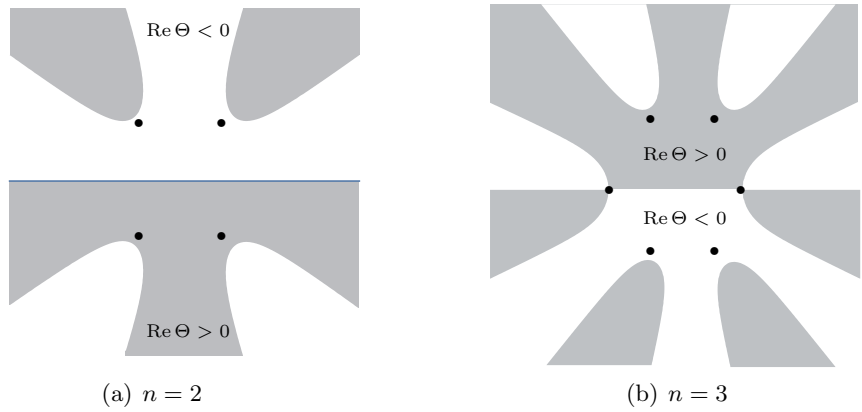


\section{Proofs of main results}\label{sec:proofs}

\subsection{Proof of Theorem \ref{prop:connection}}
The asymptotics \eqref{eq:asypinfty} is given in \cite{CCG2019} by analysing the RH problem \ref{rhp:psi} for $\Psi$ as $x\to \infty$ and we only need to show \eqref{eq:asypminusinfty}. By tracing back the transformations $\Psi \to X \to Y \to T \to R$ given in \eqref{def:X}, \eqref{def:Y}, \eqref{def:T} and \eqref{def:R}, it is readily seen from \eqref{eq:qPsi} and \eqref{eq:Rintegral} that
\begin{align*}
&q_{\AS,n}((-1)^{n+1}x;\rho) = 2i(\Psi_1)_{12}(x)
\\
&=2|x|^{\frac{1}{2n}}i \lim_{\zeta \to \infty}\zeta R_{12}(\zeta)
=-\frac{|x|^{\frac{1}{2n}}}{\pi}\int_{\Sigma_R}\bigg(R_-(s)(J_R(s)-I)\bigg)_{12} d s.
\end{align*}
This, together with \eqref{eq:JRest} and \eqref{eq:Rest}, implies that, as $x\to -\infty$,
\begin{multline}\label{eq:qAS1}
q_{\AS,n}((-1)^{n+1}x;\rho)
\\
=-\frac{|x|^{\frac{1}{2n}}}{\pi}\left(\int_{\partial D(1/2,\delta)\cup \partial D(-1/2,\delta)}(J_R)_{12}(s)ds+\Boh(|x|^{-\frac{2n+1}{2n}})\right).
\end{multline}
In view of \eqref{def:JR}, we obtain from \eqref{eq:prmatching} and \eqref{def:Pl} that, for large negative $x$,
\begin{equation*}
(J_R)_{12}(\zeta)=\left\{
                \begin{array}{ll}
                  -\frac{\rho\nu}{h_1}e^{2n|x|^{\frac{2n+1}{2n}}i/(2n+1)}\frac{\beta(\zeta)^2}{\eta(\zeta)}|x|^{-\frac{2n+1}{4n}}+\Boh(|x|^{-\frac{2n+1}{2n}}), & \hbox{$\zeta \in \partial D(1/2,\delta)$,} \\
                  -\frac{h_1}{\rho}e^{-2n|x|^{\frac{2n+1}{2n}}i/(2n+1)}\frac{1}{\beta(-\zeta)^{2}\eta(-\zeta)}|x|^{-\frac{2n+1}{4n}}+\Boh(|x|^{-\frac{2n+1}{2n}}), & \hbox{$\zeta \in \partial D(-1/2,\delta)$.}
                \end{array}
              \right.
\end{equation*}
Inserting the above estimates into \eqref{eq:qAS1}, it is readily seen that
\begin{multline*}
q_{\AS,n}((-1)^{n+1}x;\rho) =\frac{1}{\pi|x|^{\frac{2n-1}{4n}}}\Bigg(\frac{\rho \nu}{h_1}e^{2|x|^{\frac{2n+1}{2n}}i/3}\int_{\partial D(1/2,\delta)}\frac{\beta(s)^2}{\eta(s)}ds
\\
+
\frac{h_1}{\rho}e^{-2|x|^{\frac{2n+1}{2n}}i/3}\int_{\partial D(-1/2,\delta)}\frac{1}{\beta(-s)^2\eta(-s)}ds
\Bigg)
+
\Boh(|x|^{-1}).
\end{multline*}
By \eqref{eq:etalocal} and \eqref{def:beta}, we have
\begin{align*}
\int_{\partial D(1/2,\delta)}\frac{\beta(s)^2}{\eta(s)}ds & =\frac{\pi}{\sqrt{2n}}e^{\frac{3}{2}\nu\pi i+\frac{3}{4}\pi i}\left(8n|x|^{\frac{2n+1}{2n}}\right)^{\nu},
\\
\int_{\partial D(-1/2,\delta)}\frac{1}{\beta(-s)^2\eta(-s)}ds & = -\frac{\pi}{\sqrt{2n}}e^{-\frac{3}{2}\nu\pi i+\frac{3}{4}\pi i}\left(8n|x|^{\frac{2n+1}{2n}}\right)^{-\nu}.
\end{align*}
Thus,
\begin{align}\label{eq:qAS3}
& q_{\AS,n}((-1)^{n+1}x;\rho)
\nonumber
\\
&= \frac{e^{\frac34 \pi i}}{\sqrt{2n}|x|^{\frac{2n-1}{4n}}}\Bigg(\frac{\rho \nu}{h_1}e^{\frac{2n|x|^{(2n+1)/(2n)}}{2n+1}i+\frac{3}{2}\nu\pi i}\left(8n|x|^{\frac{2n+1}{2n}}\right)^{\nu}
\nonumber
\\
& \quad -
\frac{h_1}{\rho}e^{-\frac{2n|x|^{(2n+1)/(2n)}}{2n+1}i-\frac{3}{2}\nu\pi i}\left(8n|x|^{\frac{2n+1}{2n}}\right)^{-\nu}
\Bigg)
+
\Boh(|x|^{-1})
\nonumber
\\
&= \sqrt{\frac{\pi}{n}}\frac{1}{\rho |x|^{\frac{2n-1}{4n}}}\Bigg(\frac{1}{\Gamma(\nu)}e^{\frac{2n|x|^{(2n+1)/(2n)}}{2n+1}i-\frac{\nu}{2}\pi i+\nu \ln\left(8 n |x|^{\frac{2n+1}{2n}}\right)+\frac{1}{4}\pi i}
\nonumber
\\
& \quad +
\frac{1}{\Gamma(-\nu)}e^{-\frac{2n|x|^{(2n+1)/(2n)}}{2n+1}i-\frac{\nu}{2}\pi i-\nu \ln\left(8 n |x|^{\frac{2n+1}{2n}}\right)-\frac{1}{4}\pi i}
\Bigg)
+
\Boh(|x|^{-1})
\nonumber
\\
&=\frac{2}{\sqrt{n}|x|^{\frac{2n-1}{4n}}}\re \left[ \frac{\sqrt{\pi}}{\rho \Gamma(\nu)}e^{-\frac{\nu}{2}\pi i+ \frac{2n|x|^{(2n+1)/(2n)}}{2n+1}i+\nu \ln\left(8 n |x|^{\frac{2n+1}{2n}}\right)+\frac{1}{4}\pi i} \right]+
\Boh(|x|^{-1}),
\end{align}
where we have made use of the facts that $h_1=\frac{\sqrt{2 \pi}}{\Gamma(-\nu)}e^{i\pi \nu}$, $\nu=-\frac{\ln(1-\rho^2)}{2\pi i}$ and
\begin{equation}\label{eq:GammaProduct}
\Gamma(\nu)\Gamma(-\nu)=-\frac{\pi}{\nu\sin(\nu \pi)}=\frac{2\pi}{\rho^2 \nu i}e^{-\nu \pi i}
\end{equation}
in the second equality, and since $\nu$ is purely imaginary, the third equality follows immediately. To this end, we note that $|\Gamma(\nu)|^2=\Gamma(\nu)\Gamma(-\nu)$, it is then easily seen from \eqref{eq:GammaProduct} that
\begin{equation*}
\Gamma(\nu)^{-1}=\frac{|\rho|}{\sqrt{2\pi}}\sqrt{-\ln(1-\rho^2)}e^{\frac{\nu}{2}\pi i+i\arg \Gamma(-\nu)}.
\end{equation*}
Substituting this formula into \eqref{eq:qAS3}, we then obtain the asymptotics \eqref{eq:asypminusinfty}--\eqref{eq:connection} by straightforward calculations, which completes the proof of Theorem \ref{prop:connection}.
\qed

\subsection{Proofs of Theorems \ref{mainth1} and \ref{mainth2}}
We will prove Theorems \ref{mainth1} and \ref{mainth2} for {\bf Case~~I} and {\bf Case~~II} in a unified way. Thus, it is always assumed that $(x,t) \in {\bf Case~~I^{T}} \cup {\bf Case~~II^{T}}$ for some $T$ and the same notations might have different definitions for different cases.

To proceed, we observe from Lemmas \ref{wlemmaIV} and \ref{wlemmaIVg} that
\begin{align}\label{hatwLinftyIV}
\|\hat{w}\|_{(L^1 \cap L^\infty)(\hat{\Gamma})} \leq Ct^{-\frac{1}{2n+1}}.
\end{align}
In particular, $\|\hat{w}\|_{L^\infty(\hat{\Gamma})} \to 0$ uniformly as $t \to \infty$, which implies that there exist a suitable constant $T$ such that
\begin{equation}\label{eq:esthatC}
\|\hat{\mathcal{C}}_{\hat{w}}\|_{\mathcal{B}(L^2(\hat{\Gamma}))} \leq C \|\hat{w}\|_{L^\infty(\hat{\Gamma})} \leq 1/2
\end{equation}
for all $(x,t)$, where $$\hat{\mathcal{C}}_{\hat{w}}f := \hat{\mathcal{C}}_-(f\hat{w})$$ with $\hat{\mathcal{C}}$ being the Cauchy transform associated with the contour $\hat{\Gamma}$; see Section \ref{notation}. By \eqref{eq:esthatC}, the operator $I-\hat{\mathcal{C}}_{\hat{w}} \in \mathcal{B}(L^2(\hat{\Gamma}))$ is invertible and the RH problem \ref{RHmhat} has a unique solution given by
\begin{align}\label{mhatrepresentation}
\hat{m}(x, t, \lambda) = I + \hat{\mathcal{C}}(\hat{\mu} \hat{w}) = I + \frac{1}{2\pi i}\int_{\hat{\Gamma}} (\hat{\mu} \hat{w})(x, t, s) \frac{ds}{s - \lambda},
\end{align}
where $\hat{\mu}(x, t, \lambda) \in I + L^2(\hat{\Gamma})$ is defined by $\hat{\mu} = I + (I - \hat{\mathcal{C}}_{\hat{w}})^{-1}\hat{\mathcal{C}}_{\hat{w}}I$.
Also, it is readily seen from \eqref{hatwLinftyIV} that
\begin{align*}
\|\hat{\mu}(x,t,\cdot) - I\|_{L^2(\hat{\Gamma})}
\leq  \frac{C\|\hat{w}\|_{L^2(\hat{\Gamma})}}{1 - \|\hat{\mathcal{C}}_{\hat{w}}\|_{\mathcal{B}(L^2(\hat{\Gamma}))}}\leq C t^{-\frac{1}{2n+1}}.
\end{align*}

For $\lambda \in \hat{\Gamma}$, it is easily seen from \eqref{m0invexpansionIV}, \eqref{m0circleIV-2}, Lemmas \ref{wlemmaIV} and \ref{wlemmaIVg} that as $t \to \infty$,
\begin{equation}\label{eq:hatwexp}
\hat{w}(x, t,\lambda) = \frac{\hat{w}_1(y,\lambda)}{t^{\frac{1}{2n+1}}}+ \frac{\hat{w}_2(y,\lambda)}{t^{\frac{2}{2n+1}}} + \cdots + \frac{\hat{w}_N(y,\lambda)}{t^{\frac{N}{2n+1}}} + \frac{\hat{w}_{err}(x,t,\lambda)}{t^{\frac{N+1}{2n+1}}},
\end{equation}
where the coefficients $\{\hat{w}_j\}_{j=1}^N$ are nonzero only for $\lambda \in \partial D(0,\epsilon)$, and for $1 \leq p \leq \infty$,
\begin{align*}
\begin{cases}
\|\hat{w}_j(y,\cdot)\|_{L^p(\hat{\Gamma})}  \leq C, \qquad   j = 1, \dots, N,\\
\|\hat{w}_{err}(x,t,\cdot)\|_{L^p(\hat{\Gamma})}  \leq C.
\end{cases}
\end{align*}
As a consequence, we also have
$$\hat{\mathcal{C}}_{\hat{w}} = \frac{\hat{\mathcal{C}}_{\hat{w}_1}}{t^{\frac{1}{2n+1}}} + \frac{\hat{\mathcal{C}}_{\hat{w}_2}}{t^{\frac{2}{2n+1}}} + \cdots + \frac{\hat{\mathcal{C}}_{\hat{w}_N}}{t^{\frac{N}{2n+1}}} + \frac{\hat{\mathcal{C}}_{\hat{w}_{err}}}{t^{\frac{N+1}{2n+1}}}.$$
Since $\hat{\mu} = \sum_{j=0}^N \hat{\mathcal{C}}_{\hat{w}}^jI + (I-\hat{\mathcal{C}}_{\hat{w}})^{-1}\hat{\mathcal{C}}_{\hat{w}}^{N+1}I$, it is readily seen from the above expansion that
\begin{equation}\label{eq:muexp}
\hat{\mu}(x, t,\lambda) = I + \frac{\hat{\mu}_1(y,\lambda)}{t^{\frac{1}{2n+1}}} + \cdots
+ \frac{\hat{\mu}_N(y,\lambda)}{t^{\frac{N}{2n+1}}} + \frac{\hat{\mu}_{err}(x,t,\lambda)}{t^{\frac{N+1}{2n+1}}},
\end{equation}
where the coefficients $\{\hat{\mu}_j(y,\lambda)\}_{j=1}^N$ are smooth functions of $y \in [0,\infty)$ and
\begin{align*}
\begin{cases}
 \|\hat{\mu}_j(y,\cdot)\|_{L^2(\hat{\Gamma})} \leq C, \qquad j = 1, \dots, N,
	\\
\|\hat{\mu}_{err}(x,t,\cdot)\|_{L^2(\hat{\Gamma})} \leq C.
\end{cases}
\end{align*}
By inverting the transformations \eqref{m1def} and \eqref{mhatdef}, we obtain from \eqref{mhatrepresentation}, \eqref{eq:hatwexp} and \eqref{eq:muexp} that
\begin{align*}
& \lim_{\lambda\to \infty} \lambda(m(x,t,\lambda)  - I) = \lim_{\lambda\to \infty} \lambda( \hat m(x,t,\lambda)  - I)
\\
& = -\frac{1}{2\pi i}\int_{\hat{\Gamma}} \hat{\mu}(x, t, \lambda) \hat{w}(x, t, \lambda) d\lambda
= - \sum_{j=1}^{N} \frac{h_j(y)}{t^{\frac{j}{2n+1}}} + \Boh\big(t^{-\frac{N+1}{2n+1}}\big), \quad t\to \infty, 
\end{align*}
uniformly for $(x,t)$, where $\{h_j(y)\}_{j=1}^N$ are smooth functions with $h_1(y) = g_1(y)$. In view of \eqref{recoveru}, this leads to the asymptotic formula
\begin{align*}
u(x,t)  = 2 \lim_{\lambda\to \infty}\lambda(m(x,t,\lambda))_{21}
= \sum_{j=1}^{N} \frac{u_j(y)}{t^{\frac{j}{2n+1}}}+ \Boh\big(t^{-\frac{N+1}{2n+1}}\big),
\end{align*}
where
\begin{align}\label{ujexpressionIV}
u_j(y) = -2 (h_j)_{21}(y), \qquad j = 1, \dots, N,
\end{align}
are smooth functions of $y \in [0,\infty)$. Thus, we have proved \eqref{asymptoticsinIV}.

Since $h_1(y) = g_1(y)$, we obtain \eqref{u1expression} from \eqref{g1explicitIV} and \eqref{g1explicitII}. If it is further assumed that $r(0) = 0$, then $\hat{w}_1 = \hat{\mu}_1 = 0$, and so $h_j(y) = g_j(y)$ for $j = 1,2,3$. This, together with \eqref{ujexpressionIV} and \eqref{eq:g1}--\eqref{g1g2g3explicit}, implies \eqref{u2u3def} and \eqref{u2u3def2}. We thus complete the proof of Theorems \ref{mainth1} and \ref{mainth2} for ${\bf Case~~I}$ and ${\bf Case~~II}$.
\qed
\appendix

\section{Model RH problems}\label{sec:appen}

\subsection{The parabolic cylinder parametrix}\label{sec:para}
The parabolic cylinder parametrix $\Psi^{(\PC)}(\zeta)=\Psi^{(\PC)}(\zeta;\nu)$ with $\nu$ being a real or complex parameter is a solution of the following RH problem; see \cite[Chapter 9]{FIKN2006}.

\begin{rhp}\label{rhp:PC}
\hfill
\begin{enumerate}[label=\emph{(\alph*)}, ref=(\alph*)]
\item   $\Psi^{(\PC)}(\zeta)$ is analytic in $\mathbb{C} \setminus \{\cup^4_{j=0}\widehat\Sigma_j\cup\{0\}\}$, where the contours $\widehat\Sigma_j$, $j=0,\ldots,4,$ are indicated in Figure \ref{fig:jumps-Z}.

  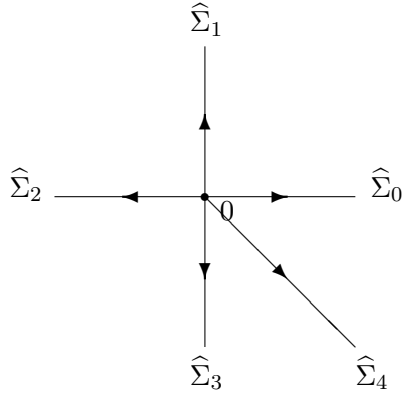
\begin{figure}[h]
\begin{center}
   \setlength{\unitlength}{1truemm}
   \begin{picture}(100,70)(-5,2)
       \put(40,40){\line(0,-1){20}}
       \put(40,40){\line(0,1){20}}
       \put(40,40){\line(-1,0){20}}
       \put(40,40){\line(1,0){20}}
      \put(40,40){\line(1,-1){20}}

       \put(40,50){\thicklines\vector(0,1){1}}
       \put(30,40){\thicklines\vector(-1,0){1}}
       \put(50,40){\thicklines\vector(1,0){1}}
       \put(40,30){\thicklines\vector(0,-1){1}}
      \put(50,30){\thicklines\vector(1,-1){1}}


       \put(42,36.9){$0$}
             \put(62,40){$\widehat \Sigma_0$}
             \put(38,62){$\widehat \Sigma_1$}
             \put(14,40){$\widehat \Sigma_2$}
             \put(38,15){$\widehat \Sigma_3$}
          \put(60,15){$ \widehat \Sigma_4$}

       \put(40,40){\thicklines\circle*{1}}
\end{picture}
   \caption{The jump contours for the RH problem \ref{rhp:PC} for $\Psi^{(\PC)}$.}
   \label{fig:jumps-Z}
\end{center}
\end{figure}

  \item $\Psi^{(\PC)}$ satisfies the following jump condition:
  \begin{equation}\label{HJumps}
  \Psi^{(\PC)}_+(\zeta)=\Psi^{(\PC)}_-(\zeta)\left\{
                         \begin{array}{ll}
                           H_j, & \hbox{$\zeta \in \widehat\Sigma_j$, $j=0,\ldots,3$,} \\
                           e^{2\nu \pi i \sigma_3}, & \hbox{$\zeta \in \widehat\Sigma_4$,}
                         \end{array}
                       \right.
  \end{equation}
  where
  \begin{equation*}
  H_0 =
  \begin{pmatrix}
    1  &  0
   \\
   h_0 &  1
  \end{pmatrix}, \quad
  H_1 = \begin{pmatrix}
    1 & h_1 \\
    0 & 1
    \end{pmatrix}, \quad
    H_{i+2}(z) = e^{i\pi(\nu+\frac12)\sigma_3}H_ie^{-i\pi(\nu+\frac12)\sigma_3}, ~~ i=0,1,
  \end{equation*}
 with
\begin{equation*}
h_0=-i\frac{\sqrt{2\pi}}{\Gamma(\nu+1)},\qquad h_1=\frac{\sqrt{2 \pi}}{\Gamma(-\nu)}e^{i\pi \nu}.
\end{equation*}

\item As $\zeta \to \infty$, we have
\begin{multline*}
 \Psi^{(\PC)}(\zeta)=
\frac{\zeta^{-\sigma_3/2}}{\sqrt{2}}
 \left(\begin{pmatrix}
1 & 1
\\
1 & -1
\end{pmatrix}+
\frac{1}{2\zeta^2}
\begin{pmatrix}
(\nu+1)(\nu+2) & -\nu(\nu-1)
\\
(\nu+1)(\nu-2) & \nu(\nu+3)
\end{pmatrix}
+\Boh(\zeta^{-4})\right)
\\
\times e^{(\frac{\zeta^2}{4}-(\nu+\frac12)\ln \zeta)\sigma_3}.
\end{multline*}
\end{enumerate}
\end{rhp}

From \cite[Section 9.4]{FIKN2006}, it follows that the above RH problem can be solved explicitly in the following way. For $z$ belonging to the region bounded by the rays $\widehat \Sigma_4$ and $\widehat \Sigma_0$,
\begin{equation*}
\Psi^{(\PC)}(\zeta)=2^{-\frac{\sigma_3}{2}}
\begin{pmatrix}
D_{-\nu-1}(i\zeta) & D_\nu(\zeta)
\\
\frac{d}{d\zeta}D_{-\nu-1}(i\zeta) & \frac{d}{d\zeta}D_\nu(\zeta)
\end{pmatrix}
\begin{pmatrix}
e^{\frac{(\nu+1)\pi}{2}i} & 0
\\
0 & 1
\end{pmatrix},
\end{equation*}
where $D_{\nu}$ stands for the parabolic cylinder function (cf. \cite[Chapter 12]{DLMF}).
The explicit formula of $\Psi^{(\PC)}(\zeta)$ in the other sectors is then determined by using the jump condition \eqref{HJumps}.

\subsection{The Model RH problem for {\bf Case I} and {\bf Case III}}\label{sec:A2}

The long-time asymptotics of $m$ in {\bf Case I} and {\bf Case III} is related to the solution $m^{\Upsilon^*}$ of the following model RH problem parameterized by $(y,t)$.
\begin{rhp}\label{RHmYIV}
\hfill
\begin{enumerate}[label=\emph{(\alph*)}, ref=(\alph*)]
\item
$m^{\Upsilon^*}(y, t, \cdot) \in I + \dot{E}^2(\C \setminus \Upsilon^*)$, where
$$\Upsilon^*:=\Upsilon_1\cup\Upsilon_{2n+1}^*\cup\Upsilon_{2n+2}^*\cup\Upsilon_{4n+2},$$
and where the contours $\Upsilon_1,\Upsilon_{2n+1}^*,\Upsilon_{2n+2}^*$ and $\Upsilon_{4n+2}$ are illustrated in Figure \ref{criticalpointsfig}.
\item
For a.e. $\lambda \in \Upsilon^*$, we have
\begin{equation*}
m_+^{\Upsilon^*}(y, t, z) =  m_-^{\Upsilon^*}(y, t, z) v^{\Upsilon^*}(y, t, z),
\end{equation*}
where
\begin{equation}\label{vYdefIV}
v^{\Upsilon^*}(y, t, z) = \begin{cases}
 \begin{pmatrix}
 1	& 0 \\
p_N(t, z)e^{2i\left(y z + \frac{(2z)^{2n+1}}{4n+2}\right)}  & 1
\end{pmatrix}, & z \in \Upsilon_1\cup \Upsilon_{2n+1}^*,
	\\
\begin{pmatrix} 1 & -p_N^*(t, z)e^{-2i\left(y z + \frac{(2z)^{2n+1}}{4n+2}\right)}	
\\
0	& 1
\end{pmatrix}, &  z \in \Upsilon_{2n+2}^*\cup\Upsilon_{4n+2},
\end{cases}
\end{equation}
and  where
\begin{align}\label{psumIV}
p_N(t,z) = s + \sum_{j=1}^N \frac{p_{j}z^j}{t^{\frac{j}{2n+1}}}, \qquad n\in \N,
\end{align}
is a polynomial of degree $N$ in $z t^{-1/(2n+1)}$ with coefficients $s =  i\rho $, $-1<\rho<1$, and $\{p_j\}_1^N \subset \C$.
\end{enumerate}
\end{rhp}

Note that if $p_j$, $j=1,\ldots, N$, in \eqref{psumIV} all vanishes, the above RH problem reduces to the RH problem \ref{rhp:psi} for the Painlev\'e II hierarchy, up to a conjugation. For general $p_j$, we have the following lemma about the properties of $m^{\Upsilon^*}$.

\begin{lemma}\label{YlemmaIV}
There exists a suitable $T$ such that the RH problem \ref{RHmYIV} has a unique solution $m^{\Upsilon^*}(y, t, z)$ whenever $y \geq 0$ and $t \geq T$. Moreover, the following properties of $m^{\Upsilon^*}$ hold.
\begin{enumerate}[label=\emph{(\alph*)}, ref=(\alph*)]
\item For each integer $N \geq 1$, there are smooth functions $\{\Phi_{jl}^{\Upsilon^*}(y)\}$ of $y \in [0,\infty)$ such that, as $z\to \infty$,
\begin{equation}\label{mYasymptoticsIV}
 m^{\Upsilon^*}(y, t, z) = I + \sum_{j=1}^N \sum_{l=0}^N \frac{\Phi_{jl}^{\Upsilon^*}(y)}{z^j t^{\frac{l}{2n+1}}}  + \Boh \left(\frac{t^{-\frac{N+1}{2n+1}}}{|z|} + \frac{1}{|z|^{N+1}}\right),
\end{equation}
uniformly with respect to $\arg z \in [0, 2\pi]$, $y \geq 0$, and $t \geq T$.

\item $m^{\Upsilon^*}$ is uniformly bounded, i.e.,
\begin{align}\label{mYboundedIV}
\sup_{y \geq 0} \sup_{t \geq T} \sup_{z \in \C\setminus \Upsilon^*} |m^{\Upsilon^*}(y, t, z)|  < \infty,
\end{align}
and satisfies the symmetry relation
\begin{align}\label{mYsymmIV}
m^{\Upsilon^*}(y, t, z) = \sigma_1\overline{m^{\Upsilon^*}(y, t, \bar{z})} \sigma_1.
\end{align}
If $p_N(t,z) = -\overline{p_N(t,-\bar{z})}$, we also have
\begin{align}\label{mYsymmII}
m^{\Upsilon^*}(y, t, z) = \sigma_1\sigma_3m^{\Upsilon^*}(y, t, -z) \sigma_3\sigma_1.
\end{align}

\item For the leading coefficient in \eqref{mYasymptoticsIV}, we have
\begin{align}\label{m10Yexplicit}
(\Phi_{10}^{\Upsilon^*})_{12}(y) = (\Phi_{10}^{\Upsilon^*})_{21}(y)= -\frac{1}{2}q_{AS,n}((-1)^{n+1}y;\rho),
\end{align}
where $q_{\AS,n}$ is the generalized Ablowitz-Segur solution of the Painlev\'{e} II hierarchy \eqref{def:PIIhierar} as stated in Theorem \ref{prop:connection}. Furthermore, if $s = 0$, $p_1 \in \R$, and $p_2 \in i\R$, the first few terms in (\ref{mYasymptoticsIV}) are given by
\begin{align}\label{eq:phi10}
& \Phi_{10}^{\Upsilon^*}(y) = 0,
	\\
& \Phi_{11}^{\Upsilon^*}(y) = \frac{p_1}{4}\Ai_{2n+1}'(y)\sigma_1, \label{eq:phi11}
	\\
& \Phi_{12}^{\Upsilon^*}(y) = \frac{p_1^2}{8i} \bigg(\int_y^{\infty} (\Ai_{2n+1}'(y'))^2dy'\bigg) \sigma_3
+ \frac{p_2}{8i} \Ai_{2n+1}''(y) \sigma_1, \label{eq:phi12}
	\\ \label{mYexplicitIV}
& \Phi_{21}^{\Upsilon^*}(y) = -\frac{p_1}{8i}\Ai_{2n+1}''(y) \sigma_3\sigma_1.
\end{align}
\end{enumerate}
\end{lemma}
\begin{proof}
We follow the strategy in \cite{CL-2019}. Define
\begin{align}\label{def:mp}
m^{P}(z) := \begin{pmatrix}1&0\\0&i\end{pmatrix}\Psi(y,\rho,z)e^{i \Xi(y,z) \sigma_3}
\begin{pmatrix}
1&0
\\
0&i
\end{pmatrix}^{-1},
\end{align}
where $\Psi$ solves the RH problem \ref{rhp:psi} and  $\Xi$ is given in \eqref{def:Xi}. It is then readily seen that $m^{P}$ satisfies the following RH problem.
\begin{rhp}\label{rhp:mp}
\hfill
\begin{enumerate}[label=\emph{(\alph*)}, ref=(\alph*)]
\item $m^{P}(z)=m^{P}(y, \rho,z)$ is defined and analytic in $\mathbb{C} \setminus \Upsilon^* $.
\item For $z \in \Upsilon^*$, we have
\begin{equation}\label{jumps:psi}
m^{P}_{+}(z) = m^{P}_{-}(z)v^{P},
\end{equation}
where
\begin{equation}\label{def:JmP}
v^{P}(z):= \left\{
        \begin{array}{ll}
          \begin{pmatrix} 1 & 0 \\ i \rho e^{2i\left(y z + \frac{(2z)^{2n+1}}{4n+2}\right)} & 1 \end{pmatrix}, & \qquad \hbox{$z \in \Upsilon_1\cup\Upsilon^{*}_{2n+1} $,} \\
          \begin{pmatrix} 1 & i\rho e^{-2i\left(y z + \frac{(2z)^{2n+1}}{4n+2}\right)} \\ 0 & 1 \end{pmatrix}, & \qquad \hbox{$z \in \Upsilon^{*}_{2n+2}\cup \Upsilon_{4n+2}$.}
        \end{array}
      \right.
\end{equation}
\item As $z\to \infty$, we have
\begin{equation}\label{eq:asymp}
 m^{P}(z)= I+ \sum_{j=1}^N\frac{m^{P}_j(y)}{z^j} +\mathcal O(z^{-N-1}),
\end{equation}
where $m^{P}_j=\begin{pmatrix}1&0\\0&i\end{pmatrix}\Psi_{j}\begin{pmatrix}1&0\\0&i\end{pmatrix}^{-1}$ with $\Psi_j$ given in \eqref{mPasymptotics} for $j=1,\ldots,N$.

\item $ m^{P}(z)$ is bounded near the origin.
\end{enumerate}
\end{rhp}

By \eqref{def:mp} and \eqref{eq:qPsi}, we have
 \begin{equation}\label{eq:qmP}
q_{\AS,n}(y;\rho):=\left\{
             \begin{array}{ll}
               -2( m^{P}_1)_{12}(y)=-2( m^{P}_1)_{21}(y), & \hbox{$n$ odd,} \\
               -2( m^{P}_1)_{12}(-y)=-2( m^{P}_1)_{21}(-y), & \hbox{$n$ even,}
             \end{array}
           \right.
\end{equation}
and
\begin{align}\label{mPbounded}
\sup_{y \geq -C_1} \sup_{z \in \C\setminus \Upsilon^*} |m^P(y,z)|  < \infty.
\end{align}
for each $C_1>0$.

It is straightforward to check that the matrix-valued function $m^{\Upsilon^*}$ satisfies RH problem \ref{RHmYIV} if and only if
\begin{align}
\label{def:mhat}
\hat{m}^{\Upsilon^*} := m^{\Upsilon^*} (m^{{P}})^{-1}
\end{align}
satisfies the following RH problem.
\begin{rhp}\label{RHmYhatIV}
\hfill
\begin{enumerate}[label=\emph{(\alph*)}, ref=(\alph*)]
\item
$\hat{m}^{\Upsilon^*} (y, t, \cdot) \in I + \dot{E}^2(\C \setminus \Upsilon^*)$.
\item
For a.e. $\lambda \in \Upsilon^*$, we have
\begin{equation*}
m_+^{\Upsilon^*}(y, t, z) =  m_-^{\Upsilon^*}(y, t, z) \hat{v}^{\Upsilon^*} (y, t, z),
\end{equation*}
where
$$\hat{v}^{\Upsilon^*}=m_{-}^{P} v^{\Upsilon^*} (m_+^{P})^{-1},$$
and where $v^{\Upsilon^*}$ is given in \eqref{vYdefIV}.
\end{enumerate}
\end{rhp}

By setting
$$\hat{w}^{\Upsilon^*}  := \hat{v}^{\Upsilon^*}  - I= m_{-}^{P} (v^{\Upsilon^*}  - v^P)  (m_+^{P})^{-1},$$
we have
\begin{align}\label{wYexpansionIV}
\hat{w}^{\Upsilon^*}(y, t, z) = \frac{\hat{w}_1^{\Upsilon^*}(y,z)}{t^{\frac{1}{2n+1}}} + \cdots + \frac{\hat{w}_{{N}}^{\Upsilon^*}(y,z)}{t^{\frac{N}{2n+1}}},
\end{align}
where, for $j=1,\ldots,N$,
\begin{align}\label{wjYIV}
\hat{w}_j^{\Upsilon^*}(y, z) = \begin{cases}
m_-^P \begin{pmatrix}
 0	& 0 \\
p_{j}z^je^{2i\left(y z + \frac{(2z)^{2n+1}}{4n+2}\right)}  & 0
\end{pmatrix}(m_+^P)^{-1}, &  z \in \Upsilon_1 \cup \Upsilon^{*}_{2n+1},
	\\
m_-^P\begin{pmatrix} 0 & - {p}^*_{j}z^je^{-2i\left(y z + \frac{(2z)^{2n+1}}{4n+2}\right)} 	\\
0	& 0
\end{pmatrix}(m_+^P)^{-1}, &   z \in \Upsilon^{*}_{2n+2} \cup \Upsilon_{4n+2}.
\end{cases}
\end{align}
Since $y\geq 0$, it follows that
\begin{align*}
\begin{cases}
\left|e^{ 2i\left(y z + \frac{(2z)^{2n+1}}{4n+2}\right)}\right|
\leq e^{-\frac{2^{2n+1}}{2n+1}|z|^{2n+1}},\quad & z \in \Upsilon_1 \cup  \Upsilon^{*}_{2n+1}, \\
\left|e^{-2i\left(y z + \frac{(2z)^{2n+1}}{4n+2}\right)}\right|
\leq e^{-\frac{2^{2n+1}}{2n+1}|z|^{2n+1}},\quad & z \in \Upsilon^{*}_{2n+2} \cup  \Upsilon_{4n+2}.
\end{cases}
\end{align*}
This, together with \eqref{wjYIV}, implies that for any integer $m\geq 0$,
\begin{align*}
|z^m\hat{w}_j^{\Upsilon^*}(y,z)| \leq &\; C |z|^{m+j} e^{-\frac{2^{2n+1}}{2n+1}|z|^{2n+1}} \leq C(n) e^{-c(n) |z|^{2n+1}}, \quad z \in \Upsilon^*,
\end{align*}
and hence, for any $1 \leq p \leq \infty$,
\begin{align}\label{wYLpestIV}
\begin{cases}
\|z^m\hat{w}_j^{\Upsilon^*}(y,z)\|_{L^p(\Upsilon^*)} \leq C,
	\\
\|z^m\hat{w}^{\Upsilon^*}(y,t,z)\|_{L^p(\Upsilon^*)} \leq Ct^{-\frac{1}{2n+1}}, \quad & t \geq 1.
\end{cases}
\end{align}
We then particularly have
\begin{equation*}\|\mathcal{C}_{\hat{w}^{\Upsilon^*}(y,t,\cdot)}^{\Upsilon^*}\|_{\mathcal{B}(L^2(\Upsilon^*)}
\leq C \|\hat{w}^{\Upsilon^*}\|_{L^\infty(\Upsilon^*)}, \quad t \geq 1.
\end{equation*}
Therefore, 
there exists a suitable $T \geq 1$ such that the RH problem \ref{RHmYhatIV} has a unique solution which is given by
\begin{align}
\hat{m}^{\Upsilon^*}(y, t, z) &= I + \mathcal{C}^{\Upsilon^*}(\hat{\mu}^{\Upsilon^*} \hat{w}^{\Upsilon^*})\nonumber\\
\label{mYrepresentationIV}
 &= I + \frac{1}{2\pi i}\int_{\Upsilon^*} (\hat{\mu}^{\Upsilon^*} \hat{w}^{\Upsilon^*})(y, t, s) \frac{ds}{s - z}, \quad t \geq T,
\end{align}
where $\hat{\mu}^{\Upsilon^*}(y, t, \cdot) \in I + L^2(\Upsilon^*)$ is defined by
\begin{align}\label{muYdefIV}
\hat{\mu}^{\Upsilon^*} = I +\mathcal{C}^{\Upsilon^*}_{\hat{w}^{\Upsilon^*}}\hat{\mu}^{\Upsilon^*}
=I + (I - \mathcal{C}^{\Upsilon^*}_{\hat{w}^{\Upsilon^*}})^{-1}\mathcal{C}^{\Upsilon^*}_{\hat{w}^{\Upsilon^*}}I.
\end{align}
Note that
$$\mathcal{C}^{\Upsilon^*}_{\hat{w}^{\Upsilon^*}} = \frac{\mathcal{C}^{\Upsilon^*}_{\hat{w}_1^{\Upsilon^*}}}{t^{\frac{1}{2n+1}}} + \frac{\mathcal{C}^{\Upsilon^*}_{\hat{w}_2^{\Upsilon^*}}}{t^{\frac{2}{2n+1}}} + \cdots + \frac{\mathcal{C}^{\Upsilon^*}_{\hat{w}_{ N}^{\Upsilon^*}}}{t^{\frac{N}{2n+1}}},\qquad N\in\N, $$
it follows from \eqref{muYdefIV} that
\begin{align}\nonumber
\hat{\mu}^{\Upsilon^*}(y, t,z) & = \sum_{r=0}^{N} (\mathcal{C}^{\Upsilon^*}_{\hat{w}^{\Upsilon^*}})^rI
+ (I - \mathcal{C}^{\Upsilon^*}_{\hat{w}^{\Upsilon^*}})^{-1}(\mathcal{C}^{\Upsilon^*}_{\hat{w}^{\Upsilon^*}})^{N+1}I	
	\\\label{muYexpansionIV}
& = I + \sum_{j=1}^N \frac{\hat{\mu}_j^{\Upsilon^*}(y,z)}{t^{\frac{j}{2n+1}}} + \frac{\hat{\mu}_{err}^{\Upsilon^*}(y,t,z)}{t^{\frac{N+1}{2n+1}}},
\end{align}
where  $\hat{\mu}_{j}^{\Upsilon^*}$ is a sum of terms of the form $\prod_{i=1}^{r}\mathcal{C}^{\Upsilon^*}_{\hat{w}_{j_i}^{\Upsilon^*}}I$
with $\Sigma_{i=1}^{r}j_i = j$, and the $\hat{\mu}_{err}^{\Upsilon^*}$ involves terms of the same form (but with $j_1 + \cdots j_r \geq N+1$).  By (\ref{wYLpestIV}),
it follows that
\begin{align}\label{mujYestIV}
\begin{cases}
 \|\hat{\mu}_j^{\Upsilon^*}(y,\cdot)\|_{L^2(\Upsilon^*)} \leq C, & j = 1, \dots, N,
	\\
 \|\hat{\mu}_{err}^{\Upsilon^*}(y,t,\cdot)\|_{L^2(\Upsilon^*)} \leq C, \quad & t \geq T,
 \end{cases}
\end{align}
for some $C>0$.

We now plug \eqref{wYexpansionIV} and \eqref{muYexpansionIV} back into \eqref{mYrepresentationIV}, it is readily seen from  \eqref{wYLpestIV} and  \eqref{mujYestIV} that, as $z\to \infty$,
\begin{align}\nonumber
\hat{m}^{\Upsilon^*}(y, t, z) = &\; I - \sum_{j=1}^{N} \frac{1}{2\pi i z^j} \bigg\{\sum_{l=1}^N t^{-\frac{l}{2n+1}} \int_{\Upsilon^*} s^{j-1} \bigg(\hat{w}_l^{\Upsilon^*} + \sum_{i=1}^{l-1} \hat{\mu}_{l-i}^{\Upsilon^*} \hat{w}_i^{\Upsilon^*} \bigg)ds
 + \Boh\big(t^{-\frac{N+1}{2n+1}}\big)\bigg\}
	\\\label{mYexpansionIV}
& + \Boh\big(|z|^{-N-1}t^{-\frac{1}{2n+1}}\big),
\end{align}
uniformly for $y \geq 0$ and $t \geq T$. Denote by
\begin{equation}\label{def:hatphi}
\hat{\Phi}_{jl}^{\Upsilon^*}(y) = - \frac{1}{2\pi i} \int_{\Upsilon^*} s^{j-1} \bigg(\hat{w}_l^{\Upsilon^*} + \sum_{i=1}^{l-1}\hat{\mu}_{l-i}^{\Upsilon^*} \hat{w}_i^{\Upsilon^*} \bigg)(y,s)ds, \qquad 1 \leq j,l \leq N.
\end{equation}
we can rewrite \eqref{mYexpansionIV} as

\begin{align}\label{mhatYasymptoticsIV}
\hat{m}^{\Upsilon^*}(y, t, z) = I +\sum_{j=1}^N \sum_{l=1}^N \frac{\hat{\Phi}_{jl}^{\Upsilon^*}(y)}{z^j t^{\frac{l}{2n+1}}}+\Boh \biggl(\frac{t^{-\frac{N+1}{2n+1}}}{|z|} + \frac{t^{-\frac{1}{2n+1}}}{|z|^{N+1}}\biggr),
\end{align}
uniformly with respect to $\arg z \in [0, 2\pi]$, $y \geq 0$ and $t \geq T$. Moreover, from \eqref{def:hatphi}, \eqref{wjYIV} and \eqref{muYexpansionIV}, we obtain the smoothness of $\hat{\Phi}_{jl}^{\Upsilon^*}(y)$. A combination of (\ref{eq:asymp}), \eqref{def:mhat} and  \eqref{mhatYasymptoticsIV} gives us \eqref{mYasymptoticsIV} with
\begin{equation}\label{eq:phij0mj}
\Phi_{j0}^{\Upsilon^*}(y)=m_j^P(z), \qquad j=1,\ldots,N.
\end{equation}
The bound (\ref{mYboundedIV}) follows from \eqref{mPbounded}, \eqref{mhatYasymptoticsIV} and the fact that the contour can be deformed.
The symmetries \eqref{mYsymmIV} and \eqref{mYsymmII} follow from the analogous symmetries for the jump $v^{\Upsilon^*}$.

We finally prove Item (c) of the lemma. The relation \eqref{m10Yexplicit} follows directly from \eqref{eq:qmP} and \eqref{eq:phij0mj}. If $s = 0$, $p_1 \in \R$, and $p_2 \in i\R$, it is easily seen from \eqref{def:mp} and Remark \ref{rk:PII} that $m^P \equiv I$, and by \eqref{wjYIV}, we have
\begin{align}\label{w2YIV1}
\hat{w}_1^{\Upsilon^*} &= \begin{pmatrix}
 0	& -p_1ze^{-2i\left(y z + \frac{(2z)^{2n+1}}{4n+2}\right)} 1_{\Upsilon_{2n+2}^{*} \cup \Upsilon_{4n+2}}(z) \\
p_1z e^{2i\left(y z + \frac{(2z)^{2n+1}}{4n+2}\right)} 1_{ \Upsilon_{1}\cup \Upsilon_{2n+1}^{*}}(z) & 0
\end{pmatrix},
	\\ \label{w2YIV}
\hat{w}_2^{\Upsilon^*} &= \begin{pmatrix}
 0	& p_2z^2e^{-2i\left(y z + \frac{(2z)^{2n+1}}{4n+2}\right)} 1_{\Upsilon_{2n+2}^{*} \cup \Upsilon_{4n+2}}(z)  \\
p_2z^2e^{2i\left(y z + \frac{(2z)^{2n+1}}{4n+2}\right)} 1_{ \Upsilon_{1}\cup\Upsilon^{*}_{2n+1}}(z) & 0
\end{pmatrix},
\end{align}
where $1_{A}(z)$ denotes the characteristic function of the set $A \subset \C$. In view of the function $\Ai_{2n+1}(y)$ defined in \eqref{def:kAiry2}, a change of variable shows that
\begin{align*}
\int_{\Upsilon_{1}\cup\Upsilon_{2n+1}^{*} }  e^{2i\left(y z + \frac{(2z)^{2n+1}}{4n+2}\right)} dz
= \int_{\Upsilon_{2n+2}^{*}\cup \Upsilon_{4n+2}} e^{-2i\left(y z + \frac{(2z)^{2n+1}}{4n+2}\right)} dz
= \pi \Ai_{2n+1}(y).
\end{align*}
By differentiating the above formula $j$ times with respect to $y$, it follows that
\begin{multline}\label{intAiryIV}
\int_{\Upsilon_{1}\cup\Upsilon_{2n+1}^{*} } z^j e^{2i\left(y z + \frac{(2z)^{2n+1}}{4n+2}\right)} dz
= (-1)^j \int_{\Upsilon_{2n+2}^{*} \cup \Upsilon_{4n+2}} z^j e^{-2i\left(y z + \frac{(2z)^{2n+1}}{4n+2}\right)} dz
\\
= \frac{\pi \Ai_{2n+1}^{(j)}(y)}{(2i)^j},
\end{multline}
for each integer $j \geq 0$. A combination of \eqref{def:hatphi} and \eqref{eq:phij0mj}--\eqref{intAiryIV} then implies that
\begin{align*}
\Phi_{10}^{\Upsilon^*}(y) &=0,
\\
 \Phi_{11}^{\Upsilon^*}(y) &=\hat{\Phi}_{11}^{\Upsilon^*}(y)=  -\frac{1}{2\pi i} \int_{\Upsilon^*} \hat{w}_1^{\Upsilon^*} dz
	\\
&=  - \frac{p_1}{2\pi i} \begin{pmatrix} 0 & - \int_{\Upsilon_{2n+2}^{*} \cup \Upsilon_{4n+2}}  z e^{-2i\left(y z + \frac{(2z)^{2n+1}}{4n+2}\right)} dz  \\
 \int_{ \Upsilon_{1}\cup\Upsilon_{2n+1}^{*}} z e^{2i\left(y z + \frac{(2z)^{2n+1}}{4n+2}\right)} dz & 0
 \end{pmatrix}
	\\
&=   -\frac{p_1}{2\pi i} \frac{\pi \Ai'_{2n+1}(y)}{2i}
\begin{pmatrix} 0 & 1   \\
1 & 0  \end{pmatrix},
	\\
\Phi_{21}^{\Upsilon^*}(y) &= \hat{\Phi}_{21}^{\Upsilon^*}(y)=  -\frac{1}{2\pi i}\int_{\Upsilon^*} z \hat{w}_1^{\Upsilon^*}  dz
	\\
&=  -\frac{p_1}{2\pi i} \begin{pmatrix} 0 & -\int_{\Upsilon_{2n+2} ^{*}\cup \Upsilon_{4n+2}}  z^2 e^{-2i\left(y z + \frac{(2z)^{2n+1}}{4n+2}\right)} dz   \\
\int_{ \Upsilon_{1}\cup\Upsilon_{2n+1}^{*}} z^2 e^{2i\left(y z + \frac{(2z)^{2n+1}}{4n+2}\right)} dz & 0 \end{pmatrix}
	\\
&= \frac{p_1}{2\pi i} \frac{\pi \Ai_{2n+1}''(y)}{4}  \begin{pmatrix} 0 & -1   \\
1 & 0  \end{pmatrix},
\end{align*}
which are \eqref{eq:phi10}, \eqref{eq:phi11} and \eqref{mYexplicitIV}. It remains to show \eqref{eq:phi12}. By \eqref{def:hatphi} and \eqref{muYexpansionIV}, we have
\begin{equation}\label{eq:phi12u*}
\Phi_{12}^{\Upsilon^*}(y) =  -\frac{1}{2\pi i}\int_{\Upsilon^*} (\hat w_2^{\Upsilon^*}(y,z) + \hat\mu_1^{\Upsilon^*}(y,z) \hat w_1^{\Upsilon^*}(y,z)) dz,
\end{equation}
and
\begin{align*}
\hat \mu_1^{\Upsilon^*}(y,z) & = \mathcal{C}^{\Upsilon^*}_{\hat w_1^{\Upsilon^*}}I
= \mathcal{C}^{\Upsilon^*}_-(\hat w_1^{\Upsilon^*})
= \frac{1}{2\pi i}\int_{\Upsilon^*} \frac{\hat w_1^{\Upsilon^*}(y,s)ds}{s-z_-}
	\\
& = \frac{p_1}{2\pi i} \begin{pmatrix}0 & -\int_{\Upsilon_{2n+2}^{*} \cup \Upsilon_{4n+2}} \frac{se^{-2i\left(y z + \frac{(2z)^{2n+1}}{4n+2}\right)}}{s-z_-}ds \\
\int_{\Upsilon_{1}\cup\Upsilon_{2n+1}^{*} } \frac{se^{2i\left(y z + \frac{(2z)^{2n+1}}{4n+2}\right)}}{s-z_-}ds & 0
\end{pmatrix}.
\end{align*}
We now define
\begin{align*}
F(y) = \int_{\Upsilon^*} \hat \mu_1^{\Upsilon^*} \hat w_1^{\Upsilon^*} dz
= -\frac{p_1^2}{2\pi i} \begin{pmatrix} F_{1}(y) & 0 \\
0 & F_{2}(y)
\end{pmatrix},
\end{align*}
where
\begin{align*}
 F_1(y) &= \int_{ \Upsilon_{1}\cup\Upsilon_{2n+1}^{*}} \bigg(  \int_{\Upsilon_{2n+2}^{*} \cup \Upsilon_{4n+2}}  \frac{sze^{-2i\left(y s  + \frac{(2s)^{2n+1}}{4n+2}\right)}e^{2i\left(y z  + \frac{(2z)^{2n+1}}{4n+2}\right)}}{s-z} ds \bigg)dz,
	\\
 F_2(y) &= \int_{\Upsilon_{2n+2}^{*} \cup \Upsilon_{4n+2}} \bigg( \int_{ \Upsilon_{1}\cup\Upsilon_{2n+1}^{*}} \frac{sze^{2i\left(y s  + \frac{(2s)^{2n+1}}{4n+2}\right)}e^{-2i\left(y z  + \frac{(2z)^{2n+1}}{4n+2}\right)}}{s-z} ds \bigg) dz.
\end{align*}
Since $F_2(y) = -F_1(y)$ (by Fubini's theorem), it follows that
\begin{align*}
F(y) = -\frac{p_1^2}{2\pi i}
F_{1}(y)\sigma_3.
\end{align*}
Differentiating both sides of the above two equalities with respect to $y$, we observe from
\eqref{intAiryIV} that
\begin{align*}
F'(y) & = \frac{p_1^2}{\pi}
\bigg(\int_{ \Upsilon_{1}\cup\Upsilon_{2n+1}^{*}} ze^{2i\left(y z  + \frac{(2z)^{2n+1}}{4n+2}\right)} dz\bigg)\bigg( \int_{\Upsilon_{2n+2}^{*} \cup \Upsilon_{4n+2}}  s e^{-2i\left(y s  + \frac{(2s)^{2n+1}}{4n+2}\right)} ds\bigg)\sigma_3
	\\
& = \frac{p_1^2}{\pi}
\bigg(\frac{\pi \Ai_{2n+1}'(y)}{2i}\bigg)\bigg(-\frac{\pi \Ai_{2n+1}'(y)}{2i}\bigg)
\sigma_3=\frac{\pi p_1^2}{4}(\Ai_{2n+1}'(y))^2\sigma_3.
\end{align*}
Note that $F(y) \to 0$ as $y\to \infty$, we have
$$\int_{\Upsilon^*} \hat \mu_1^{\Upsilon^*}(y,z)\hat w_1^{\Upsilon^*}(y,z)dz =F(y)= \int_{+\infty}^y F'(y') dy' = \frac{\pi p_1^2}{4} \left(\int_{+\infty}^y (\Ai_{2n+1}'(y'))^2dy'\right)
\sigma_3.$$
Combining the above formula and \eqref{w2YIV}--\eqref{eq:phi12u*}, we arrive at
\begin{align*}
\Phi_{12}^{\Upsilon^*}(y)
 & =
 - \frac{p_2}{2\pi i} \begin{pmatrix} 0 & \int_{\Upsilon_{2n+2}^{*} \cup \Upsilon_{4n+2}}  z^2 e^{-2i\left(y z + \frac{(2z)^{2n+1}}{4n+2}\right)} dz  \\
 \int_{ \Upsilon_{1}\cup\Upsilon_{2n+1}^{*}} z^2 e^{2i\left(y z + \frac{(2z)^{2n+1}}{4n+2}\right)} dz & 0 \end{pmatrix}
	\\
& \quad  -\frac{1}{2\pi i}\frac{\pi p_1^2}{4} \left(\int_{+\infty}^y (\Ai_{2n+1}'(y'))^2dy'\right) \sigma_3
	\\	
 & =  \frac{p_2}{2\pi i} \frac{\pi}{4} \Ai_{2n+1}''(y) \sigma_1
 +\frac{p_1^2}{8i} \left(\int_{y}^{\infty} (\Ai_{2n+1}'(y'))^2dy'\right) \sigma_3,
\end{align*}
which is \eqref{eq:phi12}.

This completes the proof of Lemma \ref{YlemmaIV}.
\end{proof}

\subsection{The Model RH problem for {\bf Case II} and {\bf Case IV}}\label{IVgeqapp}
For each $z_0 \geq 0$, we define
\begin{equation}
Z_1 = z_0+\Upsilon_1, \quad Z_2 =  -z_0+\Upsilon_{2n+1}^*, \quad Z_3 =  -z_0+\Upsilon_{2n+2}^*,\quad Z_4 =  z_0+\Upsilon_{4n+2},
\end{equation}
where  the contours $\Upsilon_1,\Upsilon_{2n+1}^*,\Upsilon_{2n+2}$ and $\Upsilon_{4n+2}$ are shown in Figure \ref{criticalpointsfig}. The long-time asymptotics of $m$ in {\bf Case II} and {\bf Case IV} is related to the solution $m^Z$ of the following model RH problem parameterized by $y \leq 0$, $t \geq 0$, and $z_0 \geq 0$.
\begin{rhp}\label{modelRHPII}
\hfill
\begin{enumerate}[label=\emph{(\alph*)}, ref=(\alph*)]
\item $m^Z(y, t, z_0, \cdot) \in I + \dot{E}^2(\C \setminus Z)$, where
\begin{equation}\label{def:Z}
Z:=\cup_{j=1}^4 Z_j \cup (-z_0,z_0);
\end{equation}
see Figure \ref{Z.pdf} for an illustration and the orientation.
\item For a.e. $\lambda\in Z$, we have
\begin{equation}
m_+^Z(y, t, z_0, z) =  m_-^Z(y, t, z_0, z) v^Z(y, t, z_0, z),
\end{equation}
where
\begin{align}\label{vZdefIVg}
&v^Z(y, t, z_0, z)
\nonumber
\\
&= \begin{cases}
 \begin{pmatrix}
 1	& 0 \\
p_N(t,z) e^{2i\left(y z + \frac{(2z)^{2n+1}}{4n+2}\right)}  & 1
\end{pmatrix}, &  z \in Z_1 \cup Z_2,
	\\
\begin{pmatrix} 1 & -p_N^*(t, z)e^{-2i\left(y z + \frac{(2z)^{2n+1}}{4n+2}\right)}	\\
0	& 1
\end{pmatrix}, &   z \in Z_3 \cup Z_4,
  	\\
\begin{pmatrix} 1 - |p_N(t, z)|^2 & -p_N^*(t, z)e^{-2i\left(y z + \frac{(2z)^{2n+1}}{4n+2}\right)} \\
p_N(t,z) e^{2i\left(y z + \frac{(2z)^{2n+1}}{4n+2}\right)}	& 1 \end{pmatrix}, &  z \in (-z_0,z_0),
\end{cases}
\end{align}
with $p_N(t,z)$ given in \eqref{psumIV}.
\end{enumerate}
\end{rhp}

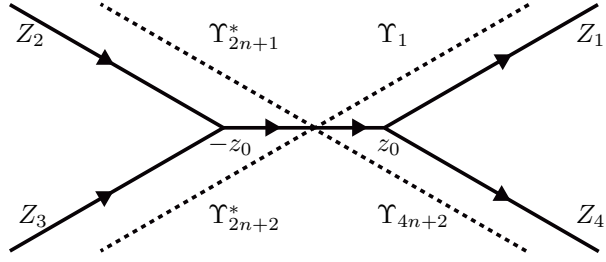
\begin{figure}
\begin{center}
 \begin{overpic}[width=.55\textwidth]{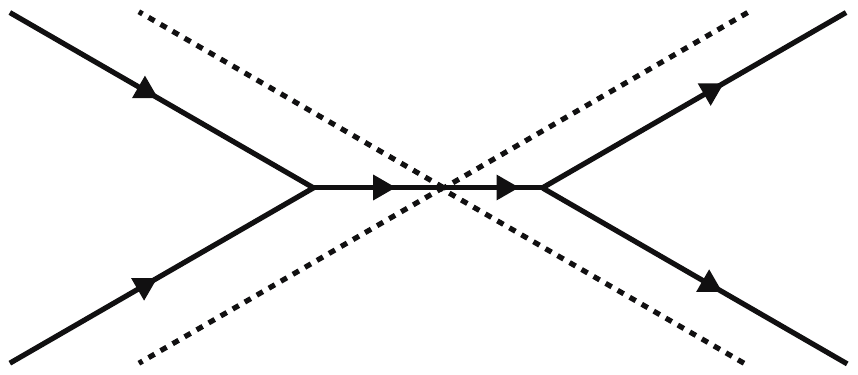}
 \put(95,36){\small $Z_1$}
 \put(2,36){\small $Z_2$}
 \put(2.5,6){\small $Z_3$}
 \put(95,6){\small $Z_4$}
 \put(62,36){\small $\Upsilon_1$}
 \put(34,36){\small $\Upsilon_{2n+1}^*$}
  \put(62,6){\small $\Upsilon_{4n+2}$}
 \put(34,6){\small $\Upsilon_{2n+2}^*$}
 \put(62,18){\small $z_0$}
 \put(34,18){\small $-z_0$}
 \end{overpic}
\caption{\label{Z.pdf}
      The contour $Z$ for the RH problem \ref{modelRHPII} for $m^Z$.
}
   \end{center}
\end{figure}

Define the parameter subset $\mathcal{P}_T$ of $\R^3$ by
\begin{align}\label{parametersetIVg}
\mathcal{P}_T = \{(y,t,z_0) \in \R^3 \, | \, -C_1 \leq y \leq 0, \, t \geq T, \, \sqrt[2n]{|y|}/2 \leq z_0 \leq C_2\},
\end{align}
where $C_1,C_2 > 0$ are constants, we have the following lemma about the properties of $m^Z$, which serves as the counterpart of Lemma \ref{YlemmaIV} in {\bf Case I} and {\bf Case III}.

\begin{lemma}\label{ZlemmaIVg}
There exists a suitable $T \geq 1$ such that the RH problem \eqref{modelRHPII} has a unique solution $m^Z(y, t, z_0, z)$ whenever $(y,t,z_0) \in \mathcal{P}_T$. Moreover, the following properties of $m^Z$ hold.
\begin{enumerate}[label=\emph{(\alph*)}, ref=(\alph*)]

\item For each integer $N \geq 1$,
\begin{align*}
&  m^Z(y, t, z_0, z) = I + \sum_{j=1}^N \sum_{l=0}^N \frac{\Phi_{jl}^\Upsilon(y)}{z^j t^{\frac{l}{2n+1}}}  + \Boh \left(\frac{t^{-\frac{N+1}{2n+1}}}{|z|} + \frac{1}{|z|^{N+1}}\right),
\end{align*}
uniformly with respect to $\arg z \in [0, 2\pi]$ and $(y,t,z_0) \in \mathcal{P}_T$ as $z \to \infty$, where $\{\Phi_{jl}^\Upsilon(y)\}$ are smooth functions of $y \in \R$ which coincide with the functions in (\ref{mYasymptoticsIV}) for $y \geq 0$ and satisfy the properties indicated in Item (c) of Lemma \ref{YlemmaIV}.

\item $m^Z$ is uniformly bounded, i.e.,
\begin{align*}
\sup_{(y,t,z_0) \in \mathcal{P}_T} \sup_{z \in \C\setminus Z} |m^Z(y, t, z_0, z)|   < \infty,
\end{align*}
and satisfies the symmetry relation
\begin{subequations}
\begin{align*}
m^Z(y, t, z_0, z) = \sigma_1\overline{m^Z(y, t, z_0, \bar{z})} \sigma_1.
\end{align*}
If $p_N(t,z) = -\overline{p_N(t,-\bar{z})}$, it also follows that
\begin{align*}
m^Z(y, t, z_0, z) = \sigma_1\sigma_3m^Z(y, t, z_0, -z)\sigma_3 \sigma_1.
\end{align*}
\end{subequations}
\end{enumerate}
\end{lemma}
\begin{proof}
Proceeding as in Section \ref{sectrans}, we introduce a matrix-valued function $m^{P1}(y,z)$ by
\begin{equation*}
m^{P1}(y,z)=\left\{
           \begin{array}{ll}
             m^{P}(z)\begin{pmatrix} 1 & 0 \\ i \rho e^{2i\left(y z + \frac{(2z)^{2n+1}}{4n+2}\right)} & 1 \end{pmatrix}, & \text{$z$ between $Z_1$ and $\Upsilon_1$}\\ & \text{and $z$ between $\Upsilon_{2n+1}^*$ and $\widetilde Z_{2}$, } \\
             m^{P}(z)\begin{pmatrix} 1 & -i \rho e^{-2i\left(y z + \frac{(2z)^{2n+1}}{4n+2}\right)} \\ 0 & 1 \end{pmatrix}, & \text{$z$ between $Z_3$ and $\Upsilon_{2n+2}^*$}
 \\
& \text{and $z$ between $Z_4$ and $\Upsilon_{4n+2}$,}  \\
             m^{P}(z), & \hbox{elsewhere,}
           \end{array}
         \right.
\end{equation*}
where $m^P$ is defined in \eqref{def:mp}. In view of the RH problem \ref{rhp:mp} for $m^{P}$, it can be easily checked that $m^{P1}$ satisfies the following RH problem.
\begin{rhp}\label{rhp:mp1}
\hfill
\begin{enumerate}[label=\emph{(\alph*)}, ref=(\alph*)]
\item $m^{P1}(y, t, z_0, \cdot) \in I + \dot{E}^2(\C \setminus Z)$, where the contour $Z$ is defined in \eqref{def:Z}.
\item For a.e. $\lambda\in Z$, we have
\begin{equation*}
m_+^{P1}(y, t, z_0, z) =  m_-^{P1}(y, t, z_0, z) v^{P1}(y, t, z_0, z),
\end{equation*}
where
\begin{align*}
&v^{P1}(y, t, z_0, z)
\nonumber
\\
&= \begin{cases}
 \begin{pmatrix}
 1	& 0 \\
i\rho e^{2i\left(y z + \frac{(2z)^{2n+1}}{4n+2}\right)}  & 1
\end{pmatrix}, &  z \in Z_1 \cup Z_2,
	\\
\begin{pmatrix} 1 & i\rho e^{-2i\left(y z + \frac{(2z)^{2n+1}}{4n+2}\right)}	\\
0	& 1
\end{pmatrix}, &   z \in Z_3 \cup Z_4,
  	\\
\begin{pmatrix} 1 - |\rho|^2 & i\rho e^{-2i\left(y z + \frac{(2z)^{2n+1}}{4n+2}\right)} \\
i\rho e^{2i\left(y z + \frac{(2z)^{2n+1}}{4n+2}\right)}	& 1 \end{pmatrix}, &  z \in (-z_0,z_0).
\end{cases}
\end{align*}
\end{enumerate}
\end{rhp}
By defining
\begin{align*}
\hat{m}^{Z} := m^{Z} (m^{P1})^{-1},
\end{align*}
one can show that, analogous to the arguments used in the proof of Lemma \ref{YlemmaIV}, that $\hat{m}^{Z}$ admits an asymptotic expansion like \eqref{mhatYasymptoticsIV}, which finally leads to the statements of Lemma \ref{ZlemmaIVg}. We omit the details here.
\end{proof}

\noindent
{\bf Acknowledgements}  {\it Lin Huang was partially supported by National Natural Science Foundation of China under grant number 11901141. Lun Zhang was partially supported by National Natural Science Foundation of China under grant number 11822104, by The Program for Professor of Special Appointment (Eastern Scholar) at Shanghai Institutions of Higher Learning, and by Grant EZH1411513 from Fudan University.
}


\end{document}